\documentclass[11pt,reqno]{amsart}
\usepackage{amsmath,amsthm,amssymb,mathrsfs,stmaryrd,color}
\usepackage[all]{xy}
\usepackage{xr}
\usepackage{url}
\usepackage{slashed}
\usepackage{geometry}

\usepackage[utf8]{inputenc}
\usepackage[T1]{fontenc}

\usepackage{relsize}
\usepackage[bbgreekl]{mathbbol}
\usepackage{amsfonts}

\DeclareMathOperator{\DHod}{\slashed{D}}
\DeclareMathOperator{\RGamma}{R \Gamma}

\DeclareMathOperator{\WCart}{WCart}

\DeclareMathOperator{\SSet}{\mathcal{S}}

\DeclareSymbolFontAlphabet{\mathbb}{AMSb}
\DeclareSymbolFontAlphabet{\mathbbl}{bbold}
\newcommand{\Prism}{\mathbbl{\Delta}}

\DeclareMathOperator{\Z}{\mathbf{Z}}
\DeclareMathOperator{\F}{\mathbf{F}}

\DeclareMathOperator{\id}{id}
\DeclareMathOperator{\Fun}{Fun}

\DeclareMathOperator{\calC}{\mathcal{C}}
\DeclareMathOperator{\calD}{\mathcal{D}}

\DeclareMathOperator{\Poly}{Poly}
\DeclareMathOperator{\CAlg}{CAlg}
\DeclareMathOperator{\dCAlg}{{\delta}\!\CAlg}
\DeclareMathOperator{\dPoly}{{\delta}\!\Poly}
\DeclareMathOperator{\anim}{an}

\newcommand{\APCref}[1]{APC.\ref*{APC-#1}\,}

\usepackage{enumitem}
\setlist[enumerate]{itemsep=2pt,parsep=2pt,before={\parskip=2pt}}

\usepackage[colorlinks=true,hyperindex, linkcolor=magenta, pagebackref=false, citecolor=cyan,pdfpagelabels]{hyperref}

\newcommand{\cosimp}[3]{\xymatrix@R=50pt@C=50pt@1{#1 \ar@<.4ex>[r] \ar@<-.4ex>[r] & {\ }#2 \ar@<0.8ex>[r] \ar[r] \ar@<-.8ex>[r] & {\ } #3 \ar@<1.2ex>[r] \ar@<.4ex>[r] \ar@<-.4ex>[r] \ar@<-1.2ex>[r] & \cdots }}

\newcommand{\colim}{\mathop{\mathrm{colim}}}

\newcommand{\adjunction}[4]{\xymatrix@R=50pt@C=50pt@1{#1{\ } \ar@<0.3ex>[r]^-{ {\scriptstyle #2}} & {\ } #3 \ar@<0.3ex>[l]^{ {\scriptstyle #4}}}}

\begin{document}

\newtheorem{theorem}{Theorem}[section]
\newtheorem*{theorem*}{Theorem}
\newtheorem*{definition*}{Definition}
\newtheorem{proposition}[theorem]{Proposition}
\newtheorem{lemma}[theorem]{Lemma}
\newtheorem{corollary}[theorem]{Corollary}
\newtheorem{conjecture}[theorem]{Conjecture}

\theoremstyle{definition}
\newtheorem{variant}[theorem]{Variant}
\newtheorem{definition}[theorem]{Definition}
\newtheorem{question}[theorem]{Question}
\newtheorem{remark}[theorem]{Remark}
\newtheorem{warning}[theorem]{Warning}
\newtheorem{example}[theorem]{Example}
\newtheorem{notation}[theorem]{Notation}
\newtheorem{convention}[theorem]{Convention}
\newtheorem{construction}[theorem]{Construction}
\newtheorem{claim}[theorem]{Claim}
\newtheorem{assumption}[theorem]{Assumption}

\newcommand{\Qc}{q-\mathrm{crys}}

\newcommand{\Shv}{\mathrm{Shv}}
\newcommand{\et}{{\acute{e}t}}
\newcommand{\crys}{\mathrm{crys}}
\renewcommand{\inf}{\mathrm{inf}}
\newcommand{\Hom}{\mathrm{Hom}}
\newcommand{\Sch}{\mathrm{Sch}}
\newcommand{\Spf}{\mathrm{Spf}}
\newcommand{\Spa}{\mathrm{Spa}}
\newcommand{\Spec}{\mathrm{Spec}}
\newcommand{\perf}{\mathrm{perf}}
\newcommand{\QSyn}{\mathrm{QSyn}}
\newcommand{\perfd}{\mathrm{perfd}}
\newcommand{\arc}{{\rm arc}}

\newcommand{\rad}{\mathrm{rad}}

\newcommand{\psh}{\mathrm{PShv}}
\newcommand{\scr}{\mathrm{sCAlg}}
\newcommand{\HT}{\mathrm{HT}}
\newcommand{\dR}{\mathrm{dR}}
\newcommand{\LdR}{\mathrm{LdR}}

\setcounter{tocdepth}{1}

\newcommand{\comment}[1]{\textcolor{red}{\footnotesize #1}}

\title{The prismatization of $p$-adic formal schemes}
\author{Bhargav Bhatt}
\author{Jacob Lurie}
\maketitle
\begin{abstract}
In this note, we introduce and study the Cartier--Witt stack $\WCart_X$ attached to a $p$-adic formal scheme $X$ as well as some variants. In particular, we reinterpret the notion of prismatic crystals on $X$ and their cohomology in terms of quasicoherent sheaf theory on $\WCart_X$ in favorable situations.
\end{abstract}

\tableofcontents

This document is a postscript to \cite{BhattLurieAPC}. In particular, the introduction below should be read in conjunction with that of \cite{BhattLurieAPC}. Moreover, this document is a preliminary version, and we hope to revisit and expand on the exposition in the future. In particular, the definition and basic properties of the key object of this study in this paper --- the prismatization $\WCart_X$ of a bounded $p$-adic formal scheme $X$ --- rely critically on the notion of mapping spaces provided by derived $p$-adic formal geometry, and some of our arguments rely on a working theory of derived formal $\delta$-schemes; neither of these theories has been systematically documented in the literature yet as far as we know.

References to \cite{BhattLurieAPC} in this paper have the form APC.x.y.z. 

\section{Introduction}

In \S \APCref{sec:WCart}, we constructed the Cartier--Witt stack $\WCart$ as a functor on $p$-nilpotent rings; one of the essential features of this construction was the equivalence between quasi-coherent complexes on $\WCart$ and crystals of $(p,I)$-complete complexes on the absolute prismatic site $\mathrm{Spf}(\mathbf{Z}_p)_\Prism$ of $\mathrm{Spf}(\mathbf{Z}_p)$ (Proposition~\APCref{proposition:DWCart-prism-description}), allowing us to simultaneously localize the study of such crystals on the stack $\WCart$ and also to understand the stack $\WCart$ via prisms. In this paper, our goal is to extend this picture to bounded $p$-adic formal schemes $X$: we shall attach a stack $\WCart_X$ to such a formal scheme $X$ with the property that the quasi-coherent sheaves on $\WCart_X$ correspond to prismatic crystals under mild assumptions, and that this correspondence is compatible with derived pushforward under mild assumptions. The format of such a theory is also explained in \cite[\S 1.1]{drinfeld-prismatic}.

More precisely, in \S \ref{ss:AbsPrism}, we introduce the Cartier--Witt stack $\WCart_X$ (as a stack on $p$-nilpotent rings) for a bounded  $p$-adic formal scheme $X$. Its construction is functorial in $X$, and its  definition relies on a small amount of derived algebraic geometry; the relevant animated rings are best understood in terms of animated prisms, which are introduced in \S \ref{ss:AnimPrism} (building on a theory of animated $\delta$-rings we discuss in Appendix \ref{ss:CAlgDeltaAnim}). When $X=\mathrm{Spf}(\mathbf{Z}_p)$, we have $\WCart_X = \WCart$, essentially by definition; on the other hand, if $X=\mathrm{Spf}(R)$ for a perfectoid ring $R$, we have $\WCart_X = \mathrm{Spf}(A)$, where $(A,I)$ is the perfect prism attached to $R$ (Example~\ref{PerfectoidPrismatize}). In fact, in general, the Cartier--Witt $\WCart_X$ is closely related to the absolute prismatic site $X_\Prism$ of $X$ (Construction~\ref{cons:pointprismatization}); using this connection for $X=\mathbf{G}_m$, in \S \ref{ss:PrismLogStack}, we reinterpret the results on the prismatic logarithm from \S \APCref{section:twist-and-log} in terms of a group scheme $G_{\WCart}$ constructed from $\WCart_{\mathbf{G}_m}$.

One can study the stack $\WCart_X$ via the map $\WCart_X \to \WCart$ coming from functoriality. As $\WCart$ can be probed using prisms via Construction~\APCref{construction:point-of-prismatic-stack}, we are therefore led to study the stacks $\WCart_X \times_{\WCart,\rho_A} \mathrm{Spf}(A)$, where $(A,I)$ is a bounded prism, and $\rho_A:\mathrm{Spf}(A) \to \WCart$ is the map from Construction~\APCref{construction:point-of-prismatic-stack}. These fibre products turn out to depend only on the $\overline{A}$-scheme $X_{\overline{A}}$. Thus, in \S \ref{ss:RelPrism}, we study relative Cartier--Witt stacks $\WCart_{Y/A}$ attached to a bounded prism $(A,I)$ and a bounded  $p$-adic formal $\overline{A}$-scheme $Y$. These stacks geometrize the study of relative prismatic cohomology: under certain syntomicity assumptions on $Y/\overline{A}$, quasi-coherent complexes on $\WCart_{Y/A}$ identify with crystals of $(p,I)$-complete complexes on the relative prismatic site $(Y/A)_\Prism$, and thus the $\mathcal{O}$-cohomology $\RGamma(\WCart_{Y/A}, \mathcal{O})$ computes the site-theoretic relative prismatic cohomology $\RGamma^{\mathrm{site}}_\Prism(Y/A)$ (Theorems~\ref{CrystalsCartWitt} and \ref{RelativePrismaticCohPrismatizeNonDer}, and Remark~\ref{PrismatizeClassical}).

\begin{remark}
Given a bounded prism $(A,I)$ and a $p$-completely smooth $\overline{A}$-algebra $R$, one of the key foundational results on prismatic cohomology is the Hodge-Tate comparison \cite{prisms}, which gives an isomorphism $H^*(\overline{\Prism}_{R/A}) \simeq \wedge^* \Omega^1_{R/\overline{A}}\{-1\}$ of graded commutative $R$-algebras. From the stacky perspective, this isomorphism has a natural explanation. Namely, the Hodge-Tate stack $\WCart^{\mathrm{HT}}_{\mathrm{Spf}(R)/A} \to \mathrm{Spf}(R)$ is a gerbe banded by the flat $R$-group scheme $T_{\mathrm{Spf}(R)/\overline{A}}\{1\}^\sharp$ given by the PD-hull of the $0$ section in the Breuil-Kisin twisted tangent bundle $T_{\mathrm{Spf}(R)/\overline{A}}\{1\}$; moreover, this gerbe is split by the choice of $(p,I)$-completely smooth $\delta$-$A$-algebra $\widetilde{R}$ lifting $R$ (Proposition~\ref{RelativeHT}). This description implies the Hodge-Tate comparison by linear algebra (Remark~\ref{HTDerivedPrismatize}). Moreover, the proof that $\WCart^{\mathrm{HT}}_{\mathrm{Spf}(R)/A} \to \mathrm{Spf}(R)$ is gerbe is itself quite conceptually straightforward from derived deformation theory. In contrast, the proof of the Hodge-Tate comparison in \cite{prisms} involved an  artificial (in hindsight) reduction to the Cartier isomorphism via the crystalline comparison theorem.
\end{remark}

To obtain stronger results as well as a better understanding, it turns out to be quite convenient to extend the preceding constructions more fully to derived algebraic geometry. Thus, given a bounded prism $(A,I)$ and a  derived $p$-adic formal scheme $Y/\overline{A}$, we construct a stack $\WCart_{Y/A}$ on derived affine schemes over $\mathrm{Spf}(A)$ in \S \ref{ss:DerRelPrism}; if $Y$ is classical and we restrict the stack to classical affines, we recover the construction discussed above. However, even when $Y$ is classical, the derived stack $\WCart_{Y/A}$ is not necessarily classical (Warning~\ref{Warn:DerNotClass}), and this is in fact a feature: the coherent cohomology $\RGamma(\WCart_{Y/A}, \mathcal{O})$ agrees with the derived prismatic cohomology $\RGamma_\Prism(Y/A)$ under quite mild assumptions on $Y/A$ (Theorem~\ref{PrismaticCohPrismatization} (1)). 

The results discussed in the previous paragraph concern the relative theory. In \S \ref{ss:DerAbsPrism}, we record their absolute counterparts. More precisely, to any  derived $p$-adic formal scheme $X$, we attach a stack on  $p$-nilpotent animated rings called the the Cartier--Witt stack $\WCart_X$ (Definition~\ref{DerCartWittAbs}); if the derived formal scheme $X$ is in fact a classical bounded $p$-adic formal scheme $X^{cl}$, then the restriction of $\WCart_X$ to discrete $p$-nilpotent rings agrees with the stack $\WCart_{X^{cl}}$ studied in \S \ref{ss:AbsPrism}, so one can regard $\WCart_X$ as a natural extension of $\WCart_{X^{cl}}$ to derived algebraic geometry.  If one further imposes syntomicity assumptions, then $\WCart_X$ is classical (Corollary~\ref{AbsWCartClassical}),  the theory of quasi-coherent sheaves on $\WCart_X$ is equivalent to the theory of crystals on the absolute prismatic site of $X$ (Proposition~\ref{QSynAbsCrys}), and this equivalence is compatible with pushforwards under relatively mild assumptions (Proposition~\ref{AbsPushWCart}).

\begin{remark}
Most of the results in \cite{BhattLurieAPC} apply either to bounded $p$-adic formal schemes or all $p$-complete animated rings. From the perspective of derived formal algebraic geometry, one can explain this somewhat curious juxtaposition: the construction that regards a $p$-complete discrete ring $R$ with bounded $p$-power torsion as a $p$-complete animated ring  extends naturally to a fully faithful embedding of bounded $p$-adic formal schemes into derived $p$-adic formal schemes that preserves and reflects the notion of open immersions. On the other hand, without the boundedness constraints, the notion of open immersions in classical formal geometry is somewhat pathological due to poor behaviour of classical completions; in particular, it may disagree with its derived counterpart. 
\end{remark}

Recall that one may heuristically view the prismatic formalism as an attempt at capturing geometry ``over $\mathbf{F}_1$''. Under this heuristic, the absolute Hodge--Tate cohomology $\overline{\Prism}_R$ of a ring $R$ is a variant of its  Hodge cohomology ``over $\mathbf{F}_1$''. In conjunction with the heuristic that regular rings ought to be smooth ``over $\mathbf{F}_1$'', this suggests that the absolute Hodge--Tate stack $\WCart_X^{\mathrm{HT}}  \subset \WCart_X$ should be well-behaved when $X$ is regular. We verify some concrete predictions of this reasoning in \S \ref{ss:HTReg}, and end with some precise questions \S \ref{ss:HTRegQ}.

\begin{warning}
We have adopted to use the same notation for the Cartier--Witt stack in two situations:
\begin{enumerate}
\item If $X$ is a  bounded $p$-adic formal scheme, then $\WCart_X$ denotes a presheaf on $p$-nilpotent rings; this is the subject of \S \ref{ss:AbsPrism} and \S \ref{ss:PrismLogStack}. 
 \item If $Y$ is a  derived $p$-adic formal scheme, then $\WCart_Y$ denotes a presheaf on $p$-nilpotent animated rings; this is the subject of \S \ref{ss:DerAbsPrism}. 
 \end{enumerate}
 These notations are compatible in the following sense: given $X$ as in (1), if $Y$ denotes the corresponding derived formal scheme, then the restriction on $\WCart_Y$ to $p$-nilpotent  rings agrees  with $\WCart_Y$. On the other hand, with the same choice of $X$ and $Y$, it is {\em not} always true that $\WCart_Y$ agrees with the natural extension of $\WCart_X$ to a presheaf on $p$-nilpotent animated rings, even after sheafification (Warning~\ref{Warn:DerNotClassAbs}): such agreement typically only holds under syntomicity assumptions. In case of disagreement, we view $\WCart_Y$ as the more fundamental object (e.g., it is closely related to derived prismatic cohomology via Corollary~\ref{AbsPushWCartZ}). Similar remarks apply in the relative case as well; the analog of (1) is the subject of \S \ref{ss:RelPrism} and \S \ref{ss:CompWCartPrism}, while the analog of (2) is studied in \S \ref{ss:DerRelPrism}. 
\end{warning}

\newpage
\section{Animated prisms}
\label{ss:AnimPrism}

In this section, we extend the notion of prisms and their basic properties to the animated setting; we shall eventually use this in \S \ref{ss:DerAbsPrism} to extend the notion of Cartier--Witt divisors from Definition \APCref{definition:generalized-Cartier-divisor} to the animated setting. One flexibility offered by the notion of animated prisms over that of classical prisms is that one can base change such a structure along essentially arbitrary maps of animated $\delta$-rings (Remark~\ref{AnimPrismBC}).  

The starting point is the observation that the notion of generalized Cartier divisors works quite well in the animated setting; in particular, there is a reasonable notion of a ``quotient'' of an animated ring $A$ by an invertible ``ideal''. We refer to \cite[\S 3]{KhanRydhVirtual} for a more thorough discussion.

\begin{construction}[Generalized Cartier divisors on animated rings]
\label{GenCartDivAnim}
Given an animated ring $A$, the following $\infty$-categories are equivalent:
\begin{enumerate}
\item The $\infty$-category of maps $I \xrightarrow{\alpha} A$ in $D(A)$ with $I$ being an invertible $A$-module; we refer to such objects as {\em generalized invertible ideals in $A$}.
\item The $\infty$-category of maps $\mathrm{Spec}(A) \to [\mathbf{A}^1/\mathbf{G}_m]$ of derived stacks.
\item The $\infty$-category of maps $A \to \overline{A}$ of animated rings whose fibre $I$ is an invertible $A$-module; we refer to such objects as {\em generalized Cartier divisors on $A$ (or on $\mathrm{Spec}(A)$)}.
\end{enumerate}
We  sketch the construction of the equivalences. The equivalence of (1) and (2) is essentially the functor of points description of $[\mathbf{A}^1/\mathbf{G}_m]$ in derived algebraic geometry. To go from (2) to (3), we  take $\mathrm{Spec}(\overline{A}) \to \mathrm{Spec}(A)$ to be the pullback of the ``origin'' $B\mathbf{G}_m \to [\mathbf{A}^1/\mathbf{G}_m]$ along the given map $\mathrm{Spec}(A) \to [\mathbf{A}^1/\mathbf{G}_m]$. To go from (3) to (1), we take the fibre $I \to A$ of the map $A \to \overline{A}$ in $D(A)$. 

For future reference, we observe that this notion has an obvious functoriality in $A$: given a map $A \to B$ of animated rings and a generalized Cartier divisor $A \to \overline{A}$, the base change $B \to B \otimes_A^L \overline{A}$ is a generalized Cartier divisor on $B$. 

We also remark that if a generalized Cartier divisor $A \to \overline{A}$ corresponds to a generalized invertible ideal $I \to A$ under the equivalence of (1) and (3), then $I \to A$ is the fibre of $A \to \overline{A}$ of $D(A)$. For this reason, we shall also often denote a generalized Cartier divisor as $A \to \overline{A}=A/I$ with the understanding that the corresponding object in (1) is given by $I \to A$.

We shall write $\mathrm{GCart}$ for the $\infty$-category of all generalized Cartier divisors $A \to \overline{A}$, defined as the full subcategory of $\mathrm{Fun}(\Delta^1,\CAlg^{\anim})$ of the $\infty$-category of arrows in $\CAlg^{\anim}$ defined by (3) above.
\end{construction}

\begin{warning}
Given a generalized Cartier divisor $A \to A/I$ where $A$ is discrete, the animated ring $A/I$ need not be discrete. In fact, $A/I$ is discrete exactly when $I \to A$ is an injective map of invertible $A$-modules; in this case, the closed subscheme $\mathrm{Spec}(\overline{A}) \subset \mathrm{Spec}(A)$ is an effective Cartier divisor in the usual sense.
\end{warning}

\begin{notation}
Given a generalized Cartier divisor $A \to \overline{A}=A/I$ and an animated $A$-algebra $B$, we make the following definitions:
\begin{itemize}
\item We say that $B$ is {\em $I$-complete} if it is complete with respect to finitely generated ideal of $\pi_0(B)$ generated by the image of $\pi_0(I) \to \pi_0(A) \to \pi_0(B)$. Similarly, we say that $B$ is $(p,I)$-complete if it is additionally also $p$-complete. 
\item We write $B \to \overline{B}=B/I_B$ for the generalized Cartier divisor on $B$ defined by base change; thus, we also write $(I_B \to B) := (I \otimes_A^L B \to B)$ for the generalized invertible ideal on $B$ defined by base change from $(I \to A)$.
\end{itemize}
\end{notation}

\begin{definition}[Animated prisms]
\label{DefAnimPrism}
\begin{enumerate}
\item An {\em animated prism} is given by an animated $\delta$-ring $A$ (Definition~\ref{definition:animated-delta-ring}) equipped with a generalized Cartier divisor $A \to \overline{A} = A/I$ on $A$ (Construction~\ref{GenCartDivAnim}) satisfying the following:
\begin{enumerate}
\item $A$ is $(p,I)$-complete.

\item Given a perfect field $k$ of characteristic $p$ and an animated $\delta$-map $A \to W(k)$ such that the corresponding map $A \to k$ of animated rings annihilates $I \to A$, we have $W(k) \otimes_A^L \overline{A} \simeq k$ as $W(k)$-algebras; note that any such identification is unique.
\end{enumerate}
\end{enumerate}

The description in (1) naturally defines an $\infty$-category of animated prisms as a full subcategory of the fibre product of
\[ \mathrm{GCart} \xrightarrow{(A \to \overline{A}) \mapsto A} \CAlg^{\anim} \xleftarrow{\text{forget}} \delta \CAlg^{\anim}. \]
Concretely, this yields the following notion of morphisms:

\begin{enumerate}[resume]
\item A morphism $(A \to \overline{A}) \to (B \to \overline{B})$ of animated prisms is given by a morphism of $f:A \to B$ of animated $\delta$-rings and an animated $A$-algebra map $\overline{A} \to \overline{B}$ (where $\overline{B}$ is viewed as an animated $A$-algebra via $f$). 

\item An animated prism $A \to A/I$ is called:
\begin{itemize}
\item {\em orientable} if $I \simeq A$ abstractly as $A$-modules.
\item {\em perfect} if $A$ is a perfect $\delta$-ring, i.e., $\phi:A \to A$ is an isomorphism.
\end{itemize}

\end{enumerate}
\end{definition}

\begin{notation}
Given an animated prism $A \to \overline{A}$, we shall write $\mathrm{Spf}(\overline{A}) \to \mathrm{Spf}(A)$ for the closed immersion of derived formal schemes obtained from $\mathrm{Spec}(\overline{A}) \to \mathrm{Spec}(A)$ by endowing both animated rings with the $(p,I)$-adic topologies, i.e., $\mathrm{Spf}(A) = \mathrm{Spec}(A) \times_{[\mathbf{A}^1/\mathbf{G}_m]} [\widehat{\mathbf{A}^1}/\mathbf{G}_m]$, where $\widehat{\mathbf{A}^1}$ is the formal completion of $\mathbf{A}^1$ (over $\mathbf{Z}$) at the closed point $\mathrm{Spec}(\mathbf{F}_p) \xrightarrow{\text{origin}} \mathbf{A}^1$.
\end{notation}

There is an evident forgetful functor from the $\infty$-category of animated prisms to the $\infty$-category of generalized Cartier divisors. The most basic examples of animated prisms are the following:

\begin{example}
\label{BasicAnimPrism}
Let $k$ be a perfect field of characteristic $p$; endow $W(k)$ with its unique $\delta$-structure. Then there is a unique (up to unique isomorphism) animated prism structure on $W(k)$, and it is given by the map $W(k) \to W(k)/p = k$. Moreover, this animated prism is perfect. 
\end{example}

We shall see in Corollary~\ref{PrismAnimPrismFF} below that if $(A,I)$ is a prism in the classical sense, then $A \to A/I$ is an animated prism, which justifies the name ``animated prism''. 

\begin{remark}[Detecting the prismatic condition on $\pi_0(-)$]
Fix an animated $\delta$-ring $A$. Recall that $\pi_0(A)$ is inherits a $\delta$-structure and $A \to \pi_0(A)$ is the universal map from $A$ to a discrete $\delta$-ring (Remark~\ref{AnimDeltapi0}). Given a generalized Cartier divisor $A \to A/I$ such that $A$ is $(p,I)$-complete, the map $A \to A/I$ is an animated prism if and only if its base change $\pi_0(A) \to \pi_0(A)/I_{\pi_0(A)}$ is an animated prism: indeed, the completeness condition on $A$ passes to $\pi_0(A)$, and condition (b) in Definition~\ref{DefAnimPrism} (1) can be checked after base to $\pi_0(A)$. 
\end{remark}

\begin{remark}[Base changing animated prism structures along $\delta$-maps]
\label{AnimPrismBC}
Fix an animated prism  $A \to A/I$. Given an animated $\delta$-$A$-algebra $B$ which is $(p,I)$-complete, the base change $B \to B/I_B$ of $A \to A/I$ is an animated prism. In particular, the map $\pi_0(A) \to \pi_0(A) / I_{\pi_0(I)} \simeq \pi_0(A)/\pi_0(I)$ (where the quotient is understood in the derived sense) is an animated prism. 

Note that even when $A \to A/I$ is a bounded prism (via the fully faithful embedding in Corollary~\ref{PrismAnimPrismFF}) and $B$ is a discrete $(p,I)$-complete $\delta$-$A$-algebra with bounded $p^\infty$-torsion, the base change $B \to B/I_B$ is typically only an animated prism (i.e., it need not lie in the image of the fully faithful embedding in Corollary~\ref{PrismAnimPrismFF}). 
\end{remark}

\begin{remark}[Rigidity of maps between animated prisms]
\label{AnimPrismStrict}
Given a map $(A \to A/I) \to (B \to B/J)$ of animated prisms, the induced animated $B$-algebra map $B/I_B \to B/J$ is an isomorphism; consequently, the induced map $(I_B \to B) \to (J \to B)$ of generalized invertible ideals in $B$ is also an isomorphism. To check this, since both $B/I_B$ and $B/J$ are perfect $B$-modules, it suffices to check isomorphy after base change to the perfected residue field $k$ at a closed point of $\mathrm{Spec}(B)$. But the structure map $B \to k$ to any such field annihilates $J$ (and thus $I_B$) by the $J$-completeness of $B$. Condition (2) in Definition~\ref{DefAnimPrism} then implies that $B/I_B \to B/J$ already becomes an isomorphism after base change along $B \to W(k)$, and thus also along $B \to k$.
\end{remark}

\begin{corollary}
The $\infty$-category of animated prisms over a fixed animated prism $(A \to A/I)$ identifies with the $\infty$-category of $(p,I)$-complete animated $\delta$-$A$-algebras via the forgetful functor.
\end{corollary}
\begin{proof}
Combine Remark~\ref{AnimPrismBC} and Remark~\ref{AnimPrismStrict}
\end{proof}

\begin{example}[Cartier-Witt divisors and animated prism structures on $W(R)$]
\label{CartWitttoAnim}
Let $R$ be a discrete $p$-nilpotent ring, and regard $W(R)$ as a $\delta$-ring with its standard $\delta$-structure. Given a Cartier-Witt divisor $\alpha:I \to W(R)$ on $R$, the corresponding generalized Cartier divisor $W(R) \to W(R)/I$ is an animated prism: the completeness conditions are automatic by the definition of a Cartier-Witt divisor, and the rest follows by base change (Remark~\ref{AnimPrismBC}) using Example~\ref{BasicAnimPrism}.   Conversely, for any animated prism $W(R) \to W(R)/I$, the corresponding generalized invertible ideal $\alpha:I \to W(R)$ is a Cartier-Witt divisor on $R$ exactly when the composite map $I \to W(R) \to R$ factors over the nilradical of $R$: this is true by definition of a Cartier-Witt divisor.
\end{example}

\begin{lemma}
\label{AnimPrismGlobalCrit}
Fix a discrete $\delta$-ring $A$ and a generalized Cartier divisor $A \to \overline{A} = A/I$.  Assume that $A$ is $(p,I)$-complete. Then $A \to \overline{A}$ gives an animated prism structure on $A$ if and only if $p \in (\alpha(I), \phi(\alpha(I)))$ in $A$.
\end{lemma}
\begin{proof}
Under the given hypotheses, we must check that requiring condition (2) in Definition~\ref{DefAnimPrism} is equivalent to requiring $p \in (\alpha(I), \phi(\alpha(I)))$ in $A$.

Assume that $p \in (\alpha(I),\phi(\alpha(I))) \subset A$. Fix a perfect field $k$ of characteristic $p$ and a map $A \to k$ annihilating $\alpha:I \to A$; lift this map to a unique $\delta$-map $A \to W(k)$. The base change $\alpha_k:I_k \to k$ is $0$ by assumption, so the base change $I_{W(k)} \to W(k)$ must factor (necessarily uniquely) over $pW(k) \hookrightarrow W(k)$. It will suffice to prove that the resulting map $I_{W(k)} \to pW(k)$ is an isomorphism of $W(k)$-modules. In fact, as this is a map of invertible $W(k)$-modules, it suffices to prove surjectivity. Writing $IW(k) \subset W(k)$ for the image of $\alpha_{W(k)}:I_{W(k)} \to W(k)$, it is enough to check $p \in IW(k)$. But any ideal in $W(k)$ is stable under $\phi$, so the containment $p \in IW(k)$ follows from the assumption $p \in (\alpha(I),\phi(\alpha(I))) \subset A$ and the $\phi$-equivariance of $A \to W(k)$. 

Conversely, assume condition (2) in Definition~\ref{DefAnimPrism} holds; we shall check $p \in (\alpha(I),\phi(\alpha(I))) \subset A$. As in \cite[proof of $(2) \Rightarrow (3)$ in Lemma 3.1]{prisms}, choose an ind-Zariski localization $A \to B$ of $\delta$-rings such that $I \otimes_A B \simeq B$ (as $B$-modules) and such that $p$ and $\alpha(I)$ lie in the Jacobson radical of $B$. Write $d \in \alpha(I)B$ for the image of a fixed generator of $I_B$ under $\alpha_B:I_B \to B$. We claim that $\delta(d)$ is a unit; granting this, the lemma will follow from \cite[Lemma 3.1, $(3) \Leftrightarrow (1)$]{prisms}. Now the condition of being a unit for an element of a commutative ring can be detected by passage to residue fields at closed points.  As $p$ and $d$ lie in the Jacobson radical of $B$,  it suffices to check the following: for any perfect field $k$ of characteristic $p$ and any map $B \to k$ annihilating $d$ and corresponding to a $\delta$-map $B \to W(k)$, the image of  $\delta(d)$ in $W(k)$ is a unit. It suffices to show that the image of $d$ in $W(k)$ has $p$-adic valuation $1$.  But the ideal $dW(k) \subset W(k)$ generated by the image of $d$ identifies with the image of $\alpha_{W(k)}:I_{W(k)} \to W(k)$ by construction; as $I \to A \to B \to W(k) \to k$ is $0$, we may apply the assumption in (2) to conclude that $dW(k) = pW(k)$, as wanted.
\end{proof}

\begin{lemma}
\label{PrismsToAnimPrism}
Let $A$ be a discrete $\delta$-ring. 
\begin{enumerate}
\item For any prism structure $(A,I)$ on $A$, the map $A \to \overline{A} := A/I$ defines an animated prism with corresponding generalized invertible ideal given by the inclusion $I \subset A$, 
\item Given an animated prism structure $A \to \overline{A}$ on $A$ whose generalized invertible ideal $\alpha:I \to A$ is an injective map of $A$-modules, the pair $(A,I)$ is a prism.
\end{enumerate}
\end{lemma}
\begin{proof}
Part (1) follows from the ``if'' direction of Lemma~\ref{AnimPrismGlobalCrit}, while part (2) follows from the ``only if'' direction of the same.
\end{proof}

\begin{corollary}[Prisms as animated prisms]
\label{PrismAnimPrismFF}
The construction in Lemma~\ref{PrismsToAnimPrism} (1) gives a fully faithful embedding of the category of prisms $(A,I)$ into the $\infty$-category of animated prisms $A \to \overline{A}$; its essential image  is exactly those animated prisms $A \to \overline{A}$ where $A$ and $\overline{A}$ are both discrete.
\end{corollary}
\begin{proof}
This follows immediately from Lemma~\ref{PrismsToAnimPrism}.
\end{proof}

\begin{lemma}
\label{AnimPrismTrivializeIp}
Let $A \to A/I$ be an animated prism. Then the invertible $A$-modules $\phi^* I$ and $I^p$ are isomorphic to $A$.
\end{lemma}
\begin{proof}
Isomorphisms between invertible $A$-modules can be detected after base change to $\pi_0(A)$ by deformation theory. We may therefore assume that $A$ is a discrete $\delta$-ring. In this case, the proof of this result for prisms given in \cite[Lemma 3.6]{prisms} applies equally well in our setting. 
\end{proof}

\begin{corollary}[Perfect animated prisms are prisms]
The functor in Corollary~\ref{PrismAnimPrismFF} identifies perfect prisms with perfect animated prisms.
\end{corollary}
\begin{proof}
It suffices to show that for any perfect animated prism $A \to A/I$, the pair $(A,I \to A)$ is a  prism. It is standard from deformation theory that for any perfect $p$-complete animated ring $A$ must have the form $W(R)$ for a perfect ring $R$ of characteristic $p$ (in fact, $R=A/p$); see
Remark~\ref{remark:perfect-animated-delta}. By Corollary~\ref{PrismAnimPrismFF}, it remains to check that $I \to A$ is injective. By Lemma~\ref{AnimPrismTrivializeIp} and the perfectness of $A$, the animated prism $A \to A/I$ is orientable. In this case, the claim follows from \cite[Lemma 2.34]{prisms}. 
\end{proof}

We shall need the following animated version of \cite[Proposition 7.11]{prisms}. 

\begin{proposition}[Lifting quasisyntomic covers, animated version]
\label{prop:LiftQSyn}
Let $A \to \overline{A} = A/I$ be an animated prism. Let $\overline{A}  \to R$ be a $p$-quasisyntomic cover. Then there exists a $(p,I)$-completely faithfully flat animated $\delta$-$A$-algebra $B$ and a factorization $\overline{A} \to R \to \overline{B}$ in animated $\overline{A}$-algebras.
\end{proposition}

Let us first sketch a proof using the notion of animated prismatic envelopes studied in \cite[\S 5]{MaoDerivedCrys}. Note that we are allowed to replace $R$ by a $p$-quasisyntomic cover. After such a replacement, we can find relatively perfect $(p,I)$-completely faithfully flat animated $\delta$-$A$-algebra $A'$ and an animated $A$-algebra map $A' \to R$ which is surjective on $\pi_0$ (see first paragraph of the proof below). Once we are in this situation, we can take $B$ to be the animated prismatic envelope of $A' \to R$ relative to $A$ in the sense of \cite[Corollary 5.25]{MaoDerivedCrys}; the conjugate filtration from \cite[Theorem 5.46]{MaoDerivedCrys} then proves the proposition. For the convenience of the reader, we try to give a relatively self-contained exposition of this sketch; nonetheless, we use the model categories of animated pairs from \cite{MaoDerivedCrys} at a crucial point to get $\delta$-structures on certain prismatic complexes.

\begin{proof}
It suffices to construct the desired map $A \to B$ after Zariski localization on $\mathrm{Spf}(A)$. We may therefore assume $I \simeq A$ is the trivial line bundle over $A$, with the structure map $I \simeq A \to A$ corresponding to an element $d \in \pi_0(A)$ that is distinguished. We may then choose a transversal prism $(A_0,I_0 = (d))$ and a map $(A_0,I_0) \to (A,I)$ of animated prisms (e.g., using Remark~\APCref{remark:universal-of-A0}).  By replacing $R$ with a $p$-quasisyntomic cover obtained by formally adjoining $p$-power roots of all elements of $\pi_0(R)$ in animated $p$-complete $\overline{A}$-algebras, we can find a map $A' := A[\{X_s^{1/p^\infty}\}_{s \in S}]^{\wedge}_{(p,I)} \to R$, surjective on $\pi_0$, for some set $S$. Note that $A'$ is naturally a $(p,I)$-completely faithfully flat and relatively perfect animated $\delta$-$A$-algebra. We shall check that $B := \Prism_{R/A'}$ (defined via animating the $\widehat{\mathcal{D}}(A_0)$-valued functor $\Prism_{B/(-)^{\wedge}_{(p,d)}}$ on finitely generated free $\delta$-rings over $A_0$) solves the problem, i.e., $B$ has a natural animated $\delta$-$A$-algebra structure,  there exists a factorization $\overline{A} \to R \to \overline{B}$  of the natural map $\overline{A} \to \overline{B}$ of animated $\overline{A}$-algebras, and that the map $R \to \overline{B}$ appearing in this factorization is a $p$-quasisyntomic cover; note that the last condition implies that $A \to B$ is $(p,I)$-completely faithfully flat by assumption on $\overline{A} \to R$.

First, we explain why $B$ admits an animated $\delta$-$A$-algebra structure. For this, consider the $\infty$-category $\mathcal{C}$ of maps $(C \to T)$ where $C$ is an animated $\delta$-$A_0$-algebra, $T$ is an animated $\overline{A_0}$-algebra, and the map $C \to T$ is an animated $A_0$-algebra map which is surjective on $\pi_0$. This $\infty$-category is compactly generated with compact generators as described in \cite[Lemma 5.24]{MaoDerivedCrys}. In particular, we can represent $(A' \to R) \in \mathcal{C}$ as the colimit of a diagram $\{A_i \to R_i\}_{i \in I}$ in $\mathcal{C}$, where each $A_i$ is a $\delta$-$A_0$-algebra of the form $A_0\{\underline{X},\underline{Y}\}^{\wedge}_{(p,d)}$ for suitable sets $\underline{X}$ and $\underline{Y}$ of free $\delta$-variables and $R_i = (A_i/(d,\underline{Y}))^{\wedge}_p$. It is then enough to endow each $\Prism_{R_i/A_i}$ with an animated $\delta$-$A_i$-algebra structure functorially in $i$. But each $\Prism_{R_i/A_i}$ is discrete by the Hodge-Tate comparison and in fact identifies with the prismatic envelope of $\ker(A_i \to R_i) \subset A_i$, which has a unique $\delta$-structure compatible with that of $A_i$ (see \cite[Lemma 7.7]{prisms}), so we win.

Next, observe that $A \to A'$ is relatively perfect, so the map $\Prism_{R/A} \to B := \Prism_{R/A'}$ is an isomorphism by the Hodge-Tate comparison. By functoriality of the Hodge-Tate structure maps for prismatic cohomology relative to $A$, we obtain the desired factorization $\overline{A} \to R \to \overline{B}$ of the natural map $\overline{A} \to \overline{B}$ of animated $\overline{A}$-algebras. 

Finally, we must show that $R \to \overline{B} \simeq \overline{\Prism}_{R/A}$ is $p$-completely faithfully flat. This follows from the Hodge-Tate comparison and the assumption that $\overline{A} \to R$ is $p$-quasisyntomic.
\end{proof}

\newpage
\section{Absolute prismatization}
\label{ss:AbsPrism}

In this section, we introduce and study the Cartier--Witt stack of a $p$-adic formal scheme as well as some important related stacks.

\begin{definition}[The prismatization of a $p$-adic formal scheme]
\label{DefPrismatizeFormal}
Let $X$ be a  bounded $p$-adic formal scheme. Its {\em prismatization} is the groupoid valued functor $\WCart_X$ on $p$-nilpotent rings defined as follows: the groupoid $\WCart_X(R)$ is the groupoid of pairs $(I \xrightarrow{\alpha} W(R), \eta:\mathrm{Spec}(\overline{W(R)}) \to X)$, where $(I \xrightarrow{\alpha} W(R)) \in \WCart(R)$ is a Cartier-Witt divisor, and $\eta$ is a morphism of derived formal schemes. 
\end{definition}

A Cartier-Witt divisor on $\alpha:I \to W(R)$ on a $p$-nilpotent ring $R$ can be viewed as an animated prism structure on $W(R)$ via Example~\ref{CartWitttoAnim}. Experience in prismatic cohomology then suggests that $\mathrm{Spf}(\overline{W(R)})$, where $\overline{W(R)}$ is endowed with the $p$-adic topology, is the natural space to attach to the animated ring $\overline{W(R)}$. On the other hand, Definition~\ref{DefPrismatizeFormal} uses seemingly larger space $\mathrm{Spec}(\overline{W(R)})$. But this is only an apparent discrepancy as the two objects coincide. This follows from the next two lemmas, which will also be used in the sequel. 

\begin{lemma}
\label{WittSquareZero}
Let $R$ be a $p$-nilpotent ring.
\begin{enumerate}
\item If $J \subset R$ is a square-zero ideal, then $W(J) := \ker(W(R) \to W(R/J))$ is a square-zero ideal of $W(R)$.
\item For any $n \geq 1$, the kernel of $W_{n+1}(R) \to W_n(R)$ is a nilpotent ideal. Moreover, the order if nilpotence is bounded in terms of the order of $p$-nilpotence of $R$.
\end{enumerate}
\end{lemma}
\begin{proof}
\begin{enumerate}
\item Each element of $W(J)$ can be written uniquely as $\sum_{i \geq 0} V^i([a_i])$ for $a_i \in J$. By continuity, it then suffices to show that for all $a,b \in J$ and $i,j \geq 0$, we have $V^i([a]) V^j([b]) = 0$. We prove this by induction on $i+j$. If $i=j=0$, then we are simply observing that $[a][b] = [ab] = 0$ as $J$ is square-zero. If $i > 0$ and $j=0$, then 
\[ V^i([a]) [b] = V(V^{i-1}([a])) [b] = V(V^{i-1}([a])) F([b]),\] 
which vanishes as $F([b]) = [b^p] = 0$ as $b^p \in J^2=0$. Finally if $i,j > 0$, then a similar manipulation gives
\[ V^i([a]) V^j([b]) =  V(V^{i-1}([a]) FV^j([b])) = V(V^{i-1}([a]) p V^{j-1}([b])) = p V(V^{i-1}([a]) V^{j-1}([b])),\] 
which vanishes as $V^{i-1}([a]) V^{j-1}([b]) = 0$ by the inductive hypothesis.
\item The claim about the order of nilpotence will follow from the proof of nilpotence. For the latter, using (1) and the $p$-adic filtration on $R$, we may assume $R$ is an $\mathbf{F}_p$-algebra. In this case, we claim that $W_{n+1}(R)$ is a square-zero extension of $W_n(R)$. Since $W_m(R) = W(R)/V^m W(R)$ for all $m$, this amounts to checking the following: for an $\mathbf{F}_p$-algebra $R$ and $a,b \in W(R)$, we have $V^n(a) V^n(b) \in V^{n+1}(R)$. We prove this by induction on $n \geq 1$. For  $n=1$, we have 
\[ V(a) V(b) = V(FV(b) a) = V(p ba) = V(1) V(ba) = V(FV(ba)) = V^2 F(ba) \in V^2 W(B).\] 
For $n > 1$, we have 
\[ V^n(a) V^n(b) = V(V^{n-1}(a) FV^n(b)) = V(p V^{n-1}(a) V^{n-1}(b)) = p V(V^{n-1}(a) V^{n-1}(b)).\] 
By induction, this has the form $p V(V^n(c)) = V^{n+1}(pc)$ for some $c \in W(B)$, as wanted.
\end{enumerate}
\end{proof}

\begin{lemma}
Let $\alpha:I \to W(R)$ be a Cartier-Witt divisor on a $p$-nilpotent ring $R$. Then $\pi_0(W(R)/I)$ is a $p$-nilpotent ring.
\end{lemma}
\begin{proof}
 Using Lemma~\ref{WittSquareZero} (1) repeatedly, we may assume that $R$ is an $\mathbf{F}_p$-algebra, and that the Cartier-Witt divisor $\alpha:I \to W(R)$ lies in the Hodge-Tate locus, i.e., $\alpha(I) \subset VW(R)$. Next, we can choose a finite Zariski open cover $\{\mathrm{Spec}(R_i)\}$ of $\mathrm{Spec}(R)$ such that $I_{W(R_i)}$ is free over $W(R_i)$. Now the functor $S \mapsto W(S)/I_{W(S)}$ is a Zariski sheaf of groupoids (or even complexes) on $\mathrm{Spec}(R)$. Since a finite limit of $p$-nilpotent complexes is $p$-nilpotent, we then reduce to checking the statement when $I \simeq W(R)$ is free. Write $V(u) \in \alpha(I) \subset VW(R)$ for the image under $\alpha$ of a fixed generator of $I$, so $\alpha(I) = V(u) W(R)$.  Note that $u \in W(R)$ must be a unit by definition of a Cartier-Witt divisor.  We must check that $p$ is nilpotent in the ring $W(R)/V(u)W(R)$; in fact, we claim that $p^2 \in V(u) W(R)$. For this, we must find $x \in W(R)$ with $p^2 = x V(u)$. Noting that $p=V(1)$ as $R$ is an $\mathbf{F}_p$-algebra, this amounts to finding $x \in W(R)$ such that $V(p) = V(Fx u)$ or equivalently that $Fx = p u^{-1}$. But then $x=Vu^{-1}$ does the job: indeed, $Fx = FVu^{-1} = pu^{-1}$.
\end{proof}

\begin{remark}[$\WCart$ as a prismatization]
We have $\WCart_{\mathrm{Spf}(\mathbf{Z}_p)} \simeq \mathrm{WCart}$ as $\mathrm{Spf}(\mathbf{Z}_p)$ is the final object in the category of $p$-adic formal schemes. 
\end{remark}

\begin{remark}[Compatibility of $\WCart_{(-)}$ with limits]
\label{PrismatizeLimits}
For any map $f:X \to Y$ of  bounded $p$-adic formal schemes, there is an induced functorial map $\WCart_f:\WCart_X \to \WCart_Y$ on prismatizations, so we can regard $\WCart_{(-)}$ as a functor from  bounded $p$-adic formal schemes to presheaves of groupoids on $p$-nilpotent rings. In fact, as $\mathrm{Spf}(\mathbf{Z}_p)$ is the final object in the category of $p$-adic formal schemes, the functor $\WCart_{(-)}$ naturally takes valued in presheaves of groupoids over $\WCart = \WCart_{\mathrm{Spf}(\mathbf{Z}_p)}$; when viewed as such, it is immediate from the definition that this functor commutes with limits of Tor-independent diagrams (i.e., those diagrams whose limit in $p$-adic formal schemes agrees with the limit in derived $p$-adic formal schemes).
\end{remark}

\begin{remark}[The Frobenius on $\WCart_X$]
For any  bounded $p$-adic formal scheme $X$, the presheaf $\WCart_X$ carries a natural  (in $X$) endomorphism $F_X:\WCart_X \to \WCart_X$ induced by the Frobenius on the Witt vectors: given a $p$-nilpotent ring $R$ and a point 
\[ (\alpha,\eta) := (I \xrightarrow{\alpha} W(R), \eta:\mathrm{Spec}(\overline{W(R)}) \to X) \in \WCart_X(R),\] we obtain a new point 
\[ F_X(\alpha,\eta) := (F^* I \xrightarrow{F^* \alpha} W(R), \mathrm{Spec}(W(R)/ F^* I ) \xrightarrow{\eta \circ \overline{F}} X) \in \WCart_X(R),\] 
where $\overline{F}$ comes from the map $\overline{W(R)} := W(R)/I \to W(R)/ F^* I$ induced by the Frobenius on $W(R)$. Note that $F_{\mathrm{Spf}(\mathbf{Z}_p)}$ equals the Frobenius on $\WCart$ from Construction~\APCref{construction:Frobenius-on-stack}. Moreover, using the equality  $F = W(\varphi)$ of endomorphisms of $W(R)$ when $R$ is an $\mathbf{F}_p$-algebra with Frobenius $\varphi$, one checks that $F_X$ is a lift of the Frobenius on $\WCart_X \otimes_{\mathbf{Z}_p} \mathbf{F}_p$;  thus, we can regard $F_X$ as a $\delta$-structure on $\WCart_X$ when the latter is known to be $\mathbf{Z}_p$-flat (e.g., when $X=\mathrm{Spf}(\mathbf{Z}_p)$). 
\end{remark}

The following variants of the prismatization construction shall also be useful.

\begin{construction}[The Hodge-Tate stack]
\label{ConsAbsHT}
For a bounded $p$-adic formal scheme $X$. Form a fibre square
\[ \xymatrix{ \WCart_X^{\mathrm{HT}} \ar[r] \ar[d] & \WCart_X \ar[d] \\
\mathrm{WCart}^{\mathrm{HT}} \ar[r] & \mathrm{WCart} }\]
defining the {\em Hodge-Tate stack} $\WCart_X^{\mathrm{HT}}$ of $X$. Given a $p$-nilpotent ring $R$ and a Cartier-Witt divisor $(I \xrightarrow{\alpha} W(R)) \in \WCart^{\mathrm{HT}}(R) \subset \WCart(R)$, the map $\alpha$ factors over $VW(R) \subset W(R)$, so there is an induced map $\overline{W(R)} \to W(R)/VW(R) \simeq R$. Consequently, given a point $( (I \xrightarrow{\alpha} W(R)), \eta:\mathrm{Spec}(\overline{W(R)}) \to X) \in \WCart_X^{\mathrm{HT}}(R)$, one obtains a map $\overline{\eta}:\mathrm{Spec}(R) \to \mathrm{Spec}(\overline{W(R)}) \xrightarrow{\eta} X$. This construction defines a map 
\[ \pi^{\mathrm{HT}}:\WCart_X^{\mathrm{HT}} \to X\] 
that we refer to as {\em the Hodge-Tate structure map}.
\end{construction}

\begin{construction}[The diffracted Hodge stack]
\label{DiffHodgeStack}
Fix a  bounded $p$-adic formal scheme $X$. We have the structure map $\WCart_X^{\mathrm{HT}} \to \WCart^{\mathrm{HT}}$. Form a fibre square
\[ \xymatrix{ X^{\DHod} \ar[r] \ar[d] & \WCart_X^{\mathrm{HT}} \ar[d] \\
   \mathrm{Spf}(\mathbf{Z}_p) \ar[r] & \WCart^{\mathrm{HT}}},\]
  where the map  $\mathrm{Spf}(\mathbf{Z}_p) \xrightarrow{\eta} \WCart^{\mathrm{HT}}$ is the $\mathbf{G}_m^\sharp$-torsor from Theorem~\APCref{theorem:describe-HT}; we call $X^{\DHod}$ the {\em diffracted Hodge stack} of $X$.  Concretely, for any $p$-nilpotent ring $S$, one has a natural identification
\[ X^{\DHod}(S) \simeq \mathrm{Map}(\mathrm{Spec}(\overline{W(S)}), X) \]
of groupoids, where the right side denotes a the space of maps in derived schemes with $\overline{W(S)} := W(S)/ V(1)$ is defined using the Cartier-Witt divisor $(W(S) \xrightarrow{V(1)} W(S))$. Under this identification, the $\mathbf{G}_m^\sharp(S)$-action on $X^{\DHod}(S)$ is induced by the natural $\mathbf{G}_m^\sharp(S)$-action on the animated ring $\overline{W(S)}$. Moreover, by construction, we have
\[ \WCart_X^{\mathrm{HT}} \simeq X^{\DHod}/\mathbf{G}_m^\sharp.\]
\end{construction}

\begin{remark}[Compatibility with \'etale localization]
\label{PrismatizeEtaleCov}
The functors $X \mapsto \WCart_X, \WCart_X^{\mathrm{HT}}, X^{\DHod}$ carry \'etale morphisms to representable \'etale morphisms, and  \'etale covers to \'etale covers. We shall explain this for $\WCart_X$, which implies the rest by base change. Fix a $p$-completely \'etale map $f:X \to Y$ of bounded $p$-adic formal schemes. We must check that $\WCart_f$ is a representable \'etale morphism, and is also an \'etale cover when $f$ is so. We give the argument when $f$ is an affine morphism; the general case is similar.

Fix a point $\mathrm{Spec}(R) \to \WCart_Y$ corresponding to a Cartier-Witt divisor $I \xrightarrow{\alpha} W(R)$ and a map $\eta:\mathrm{Spec}(W(R)/ I) \to X$ of derived formal schemes. To show representatibility, observe that the natural maps give equivalences
\[\mathrm{Spec}(R)_{\et} \gets \mathrm{Spf}(W(R))_{\et} \to \mathrm{Spec}(\overline{W(R)})_{\et},\]
where $W(R)$ is topologized using the $(p,I)$-adic topology. In particular, the pullback of $f:Y \to X$ along $\eta$ then has the form $\mathrm{Spec}(\overline{W(S)}) \to \mathrm{Spec}(\overline{W(R)})$ for a uniquely determined \'etale $R$-algebra $S$. The induced datum $((I_{W(S)} \xrightarrow{\alpha_{W(S)}} W(S)), \eta_f:\mathrm{Spec}(\overline{W(S)}) \to X)$ then yields a map $\mathrm{Spec}(S) \to \WCart_X$ fitting into a commutative square
\[ \xymatrix{ \mathrm{Spec}(S) \ar[r] \ar[d]& \mathrm{Spec}(R) \ar[d] \\
\WCart_X \ar[r] & \WCart_Y. }\]
Using a similar \'etale localization argument with the Witt vectors, one checks that this square is cartesian, proving that $f$ is representable \'etale.

It remains to prove that $\WCart_f$ is an \'etale cover if $f$ is so. This follows from the description of pullbacks given in the previous paragraph together with the observation that the equivalence $\mathrm{Spec}(R)_{\et} \simeq \mathrm{Spec}(\overline{W(R)})_{\et}$ used above preserves \'etale covers.

For future use, we note that the observations in this remark remain valid if we replace the \'etale topology with the Zariski topology. 
\end{remark}

\begin{construction}[Objects of the absolute prismatic site give points of the prismatization]
\label{cons:pointprismatization}
Fix a  bounded $p$-adic formal scheme $X$. Fix an object $(\mathrm{Spf}(A) \gets \mathrm{Spf}(\overline{A}) \to X) \in X_\Prism$ of the absolute prismatic site $X_\Prism$. For any $(p,I)$-nilpotent $A$-algebra $R$, there is a unique $\delta$-$A$-algebra structure on $W(R)$ lifting the given $A$-algebra structure on $R$. Base change along this map gives a Cartier-Witt divisor $(I_{W(R)} \to W(R))$ as in Construction~\APCref{construction:point-of-prismatic-stack} together with a map $\eta:\mathrm{Spec}(\overline{W(R)}) \to \mathrm{Spf}(\overline{A}) \xrightarrow{\eta_0} X$ of derived schemes; thus, we get a point $\mathrm{Spec}(R) \to \WCart_X$ of the prismatization. Letting $R$ vary, this construction yields an $A$-valued point 
\[ \rho_{X,A}:\mathrm{Spf}(A) \to \WCart_X\] 
of the prismatization, where $A$ is endowed with the $(p,I)$-adic topology. If there is no potential for confusion, we shall denote this point simply by $\rho_A$. 
\end{construction}

\begin{construction}[Using $\delta$-structures to probe prismatizations]
\label{DeltaSchemePrismatize}
Let $X$ be a  bounded $p$-adic formal scheme equipped with a $\delta$-structure. Then we shall construct a natural map 
\[ \pi_X:X \times \WCart \to \WCart_X\]
 over $\WCart$. Given a $p$-nilpotent ring $R$, an $R$-valued of the source corresponds to a Cartier-Witt divisor $(I \xrightarrow{\alpha} W(R))$ together with a map $\eta:\mathrm{Spec}(R) \to X$. Using the $\delta$-structure on $X$, the map $\eta$ extends uniquely to a $\delta$-map $\mathrm{Spf}(W(R)) \to X$, where $W(R)$ is endowed with the $(p,I)$-adic topology. Postcomposition with the map $W(R) \to \overline{W(R)}$ yields a map $\overline{\eta}:\mathrm{Spec}(\overline{W(R)}) \to X$. The assignment carrying $(\alpha,\eta)$ to $(\alpha,\overline{\eta})$ yields the desired map $\pi_X$. Note that the pullback of $\pi_X$ over  $\WCart^{\mathrm{HT}} \subset \WCart$ splits the Hodge-Tate structure map $\WCart_X^{\mathrm{HT}} \to X \times \WCart^{\mathrm{HT}}$.
\end{construction}

\begin{example}[The prismatization of a perfectoid]
\label{PerfectoidPrismatize}
Let $R$ be a perfectoid ring, and write $(A,I)$ for the unique perfect prism equipped with an identification $\overline{A} \simeq R$. By Construction~\ref{cons:pointprismatization}, we obtain a map 
\[ \rho_A:\mathrm{Spf}(A) \to \WCart_X.\] 
We claim this map is an isomorphism of functors. To show this, fix a $p$-nilpotent  ring $S$. We then have functorial maps
\[ \WCart_X(S) \xleftarrow{a} \mathrm{Map}(A,W(S)) \xleftarrow{b} \mathrm{Map}_{\delta}(A,W(S)) \xrightarrow{c} \mathrm{Map}(A,S),\]
where the map $c$ is obtained by postcomposition with the restriction map $W(S) \to S$, the map $b$ is the obvious inclusion, and the map $a$ is given by sending a map $A \to W(S)$ to the induced map $R = \overline{A} \to \overline{W(S)}$. Now $c$ is an equivalence as $W(-)$ is right-adjoint to the forgetful functor from $\delta$-rings to rings.  The resulting composition $abc^{-1}$ is exactly $\rho_A(S)$. The map $b$ is also an equivalence: this follows by deformation theory as $A$ is $p$-completely formally \'etale over $\mathbf{Z}$. It therefore suffices to show that $a$ is an equivalence. Given an object $x$ of $\WCart_X(S)$ given by a Cartier-Witt divisor $(J \xrightarrow{\alpha_x} W(S))$ and a map $\eta_x:R \to W(S)/ J$ of animated rings, the induced composition $A \to \overline{A} = R \xrightarrow{\eta_x} W(S)/ J$ refines uniquely (by deformation theory again) to a map $f_x:A \to W(S)$ of $\delta$-rings, giving rise to a map $(A \to A/I) \to (W(S) \to W(S)/J)$ of animated prisms. The induced map $I \otimes_A W(S) \to J$ is then an isomorphism by Remark~\ref{AnimPrismStrict}, so $x$ is the image of $f_x$ under $a$. In fact, as the constructions are functorial, it is easy to see that $x \mapsto f_x$ defines an inverse to $a$, proving $a$ is an equivalence.
\end{example}

\begin{remark}
Using Example~\ref{PerfectoidPrismatize}, prismatizing the natural map $\mathrm{Spec}(\mathbf{F}_p) \to \mathrm{Spf}(\mathbf{Z}_p)$ yields a natural map
\[ \mathrm{Spf}(\mathbf{Z}_p) \simeq \WCart_{\mathrm{Spec}(\mathbf{F}_p)} \to \WCart_{\mathrm{Spf}(\mathbf{Z}_p)} \simeq \WCart.\]
It is immediate from the definitions that this point identifies with the de Rham point from Example~\APCref{example:de-Rham-point}.
\end{remark}

\newpage
\section{Revisiting the prismatic logarithm}
\label{ss:PrismLogStack}

In this section, we reinterpret the prismatic logarithm constructed in \S \APCref{section:twist-and-log} using the Cartier--Witt stack  $\WCart_{\mathbf{G}_m}$; the discussion here is comparable to \cite[Theorem 2.7.7, \S 2.7.8]{DrinfeldFormalGroup}. 

\begin{remark}[$\delta$-structures on geometric objects]
In this section, we shall use the notion of a $\delta$-structures on bounded $p$-adic formal schemes and stacks. As $p$-completely \'etale algebras over a $p$-complete $\delta$-ring carry a unique compatible $\delta$-structure (\cite[Lemma 2.18]{prisms}), there is an evident notion of a $\delta$-structure on a bounded $p$-adic formal scheme $X$  (or even such an algebraic space). When $X$ is $\mathbf{Z}_p$-flat, the data of a $\delta$-structure on $X$ is equivalent to that of a Frobenius lift; in general, this theory is studied in \cite{BorgerWitt2}. In the context of stacks, to avoid developing the general theory here, we only use the notion of $\delta$-structures in the context of stacks which are known to be $\mathbf{Z}_p$-flat (especially $\mathrm{WCart}$ or flat group schemes over $\mathrm{WCart}$); for such stacks, a $\delta$-structure is, by definition, a lift of the Frobenius, and a morphism is one that commutes with the Frobenius lift. 
\end{remark}

\begin{construction}[The group scheme $G_{\WCart}$ of rank $1$ units of a Cartier-Witt divisor]
\label{Rank1GroupSch}
Applying Construction~\ref{DeltaSchemePrismatize} to $X=\mathbf{G}_m$ with the standard $\delta$-structure gives a map 
\[ \pi:\mathbf{G}_m \times \WCart \to \WCart_{\mathbf{G}_m}.\] 
As $\mathbf{G}_m$ has a group structure and $\WCart_{(-)}$ (valued in stacks over $\WCart$) commutes with finite products (Remark~\ref{PrismatizeLimits}), the stack $\WCart_{\mathbf{G}_m}$ is naturally a group stack over $\WCart$. Moreover, as the $\delta$-structure on $\mathbf{G}_m$ respects the group structure, the map $\pi$ is a morphism of group stacks over $\WCart$. Using Andr\'e's lemma (or simply Proposition~\ref{RelativeHT} below), the map $\pi$ can be seen to be a flat affine surjection. Its kernel $G_{\WCart}$ is thus a flat affine group scheme over $\WCart$ that sits in an exact triangle
\[ G_{\WCart} \to \mathbf{G}_m \times \WCart \to \WCart_{\mathbf{G}_m} \]
of abelian group stacks over $\WCart$. Unwinding definitions, one finds that the fibre of this sequence of stacks over a point $(I \xrightarrow{\alpha} W(R)) \in \WCart(R)$ is given by the following sequence of Picard groupoids:
\[ (1+I)_{\mathrm{rk}=1} \to R^* \xrightarrow{[\cdot]} \overline{W(R)}^*,\]
where the leftmost term can be regarded as defined by this fibre sequence. Note that we can explicitly describe this term via
\[ (1+I)_{\mathrm{rk}=1} := \mathrm{Fib}\left(R^* \xrightarrow{[\cdot]} (\overline{W(R)})^*\right) \simeq (1+I) \times_{W(R)^*} R^*\]
where we use the identification  $ W(R)^*/ (1+I) \simeq \overline{W(R)}^*$ (coming from the $I$-completeness of $W(R)$) for the last isomorphism. \
\end{construction}

\begin{remark}[$G_{\WCart}$ as a functor on prisms]
\label{GWCartPrism}
The flat group scheme $G_{\WCart} \to \WCart$ from Construction~\ref{Rank1GroupSch} supports a $\delta$-structure by functoriality, and hence can be regarded as a functor on $p$-torsionfree prisms via Construction~\APCref{construction:point-of-prismatic-stack}. Expliclty, for any bounded prism $(B,J)$ with $B$ $p$-torsionfree, we have a classifying map  $\rho_B:\mathrm{Spf}(B) \to \WCart$ of $\delta$-stacks. One can then contemplate the set  $\mathrm{Hom}_{\WCart,\delta}(\mathrm{Spf}(B),G_{\WCart})$  of $\delta$-maps $\mathrm{Spf}(B) \to G_{\WCart}$ over $\WCart$. We claim that there is a natural identification
\[ \mathrm{Hom}_{\WCart,\delta}(\mathrm{Spf}(B),G_{\WCart}) \simeq (1+J)_{\mathrm{rk}=1}.\] 
To see this, observe that the pullback $\rho_B^* G_{\WCart} \to \mathrm{Spf}(B)$ is naturally identified with $\mathrm{Spf}(C) \to \mathrm{Spf}(B)$, where $C$ is the $\delta$-$B$-algebra defined in Lemma~\APCref{lemma:add-rank-unit}: this can be seen using the description of the functor $\mathrm{Hom}_B(C,-)$ arising by composing the conclusion of Lemma~\APCref{lemma:add-rank-unit} (extended to virtual prisms $(A,I)$ with the same conclusion) with the  adjunction between the forgetful and Witt vector functors relating $\delta$-$B$-algebras to all $B$-algebras. Moreover, the structure of a $\delta$-$B$-group scheme on $\mathrm{Spf}(C)$ comes via the obvious group structure on the right side of the functor of points interpretation given in Lemma~\APCref{lemma:add-rank-unit}. These identifications show that
\[ \mathrm{Hom}_{\WCart,\delta}(\mathrm{Spf}(B),G_{\WCart}) \simeq \mathrm{Hom}_B(C,B) = (1+J)_{\mathrm{rk}=1},\]
as promised.
\end{remark}

\begin{remark}[$G_{\WCart}$ over the Hodge-Tate locus]
Let us explain why the base changed  group scheme 
\[ G_{\WCart^{\mathrm{HT}}} := G_{\WCart} \times_{\WCart} \WCart^{\mathrm{HT}} \to \WCart^{\mathrm{HT}}\] 
is a twisted form of $\mathbf{G}_a^\sharp$. Observe that for any point $(I \xrightarrow{\alpha} W(R)) \in \WCart^{\mathrm{HT}}(R) \subset \WCart(R)$, the composite map $I \xrightarrow{\alpha} W(R) \to R$ is the $0$ map. It follows that the fibre of $G_{\WCart}(R) \to \WCart(R)$ over a point  $(I \xrightarrow{\alpha} W(R)) \in \WCart^{\mathrm{HT}}(R) \subset \WCart(R)$ is identified with the group $1 +  \ker(\alpha)$ under multiplication. Subtracting $1$ identifies $1+\ker(\alpha)$ with the additive $\ker(\alpha)$ as $\ker(\alpha) = \pi_1(W(R)/I)$ inherits a square-zero multiplication. By further pulling back to the Hodge--Tate point $\mathrm{Spf}(\mathbf{Z}_p) \xrightarrow{V(1)} \WCart^{\mathrm{HT}}$, one then checks that the group scheme $G_{\WCart^{\mathrm{HT}}} \to \WCart^{\mathrm{HT}}$ is indeed a twisted form of $\mathbf{G}_a^\sharp$.
\end{remark}

\begin{remark}[The prismatic logarithm through $G_{\WCart}$]
\label{PrismLogGroupSch}
Let us reinterpret the prismatic logarithm as a homomorphism $G_{\WCart} \to \mathbf{V}(\mathcal{O}_{\WCart}\{1\})$ of group schemes over $\WCart$. Given a bounded prism $(B,J)$ with $B$ $p$-torsionfree, consider the map $(B,J) \to (C,K)$ coming from Lemma~\APCref{lemma:add-rank-unit}. The prismatic logarithm from  Proposition~\APCref{proposition:prismatic-logarithm-extension} applied to the tautological generator $w \in C$ gives an element $\log_\Prism(w) \in C\{1\}$. Via Remark~\ref{GWCartPrism}, this can be viewed as a morphism 
\[ \log_{\Prism,B}: \rho_B^* G_{\WCart} \to \mathbf{V}(\mathcal{O}_{\mathrm{Spf}(B)}\{1\}) = \underline{\mathrm{Spec}}_B(\mathrm{Sym}^* B\{-1\}) \]
 of affine formal schemes over $B$. The homomorphism property of the prismatic logarithm (Proposition~\APCref{prop:PrismLogHom}) ensures that $\log_{\Prism,B}$ is a homomorphism of group schemes. Remark~\APCref{remark:Frobenius-on-twist} shows that $\mathbf{V}(\mathcal{O}_{\mathrm{Spf}(B)}\{1\})$ has a natural Frobenius lift lying over the Frobenius lift on $\mathrm{Spf}(B)$; the compatibility of the prismatic logarithm with the Frobenius (Proposition~\APCref{proposition:log-Nygaard}) then ensures that $\log_{\Prism,B}$ respects the Frobenius lifts (and thus $\delta$-structures as $B$ is $p$-torsionfree).  As the formation of $\log_{\Prism,B}$ is evidently compatible with base change, it follows from flat descent for affine maps and Proposition~\APCref{proposition:DWCart-prism-description} that there is a homomorphism 
\[ \log_\Prism:G_{\WCart} \to \mathbf{V}(\mathcal{O}_{\WCart}\{1\})\] 
of group $\delta$-schemes over $\WCart$ characterized by the following property: for every bounded prism $(A,I)$ with $A$ $p$-torsionfree and  classifying $\delta$-map $\rho_A:\mathrm{Spf}(A) \to \WCart$, the diagram
$$ \xymatrix@R=50pt@C=50pt{ 
\mathrm{Hom}_{\WCart,\delta}(\mathrm{Spf}(A), G_{\WCart}) \ar[d]^-{\log_\Prism(\mathrm{Spf}(A))} \ar[r]^-{\simeq} & (1+I)_{\mathrm{rk}=1} \ar[dd]^-{\log_\Prism} \\
 \mathrm{Hom}_{\WCart,\delta}(\mathrm{Spf}(A), \mathbf{V}(\mathcal{O}_{\WCart}\{1\})) \ar[d]^-{\text{inc}} & \\
\mathrm{Hom}_{\WCart}(\mathrm{Spf}(A), \mathbf{V}(\mathcal{O}_{\WCart}\{1\}))\ar[r]^-{\simeq} & A\{1\} }$$
commutes; in fact, it suffices to check this compatibility for transversal prisms. 
\end{remark}

\begin{remark}[The prismatic logarithm at the de Rham point]
Let us describe the pullback $\rho_{\dR}^* \log_\Prism$ of the map $\log_\Prism$ from Remark~\ref{PrismLogGroupSch} to the de Rham point $\rho_{\dR}:\mathrm{Spf}(\mathbf{Z}_p) \to \WCart$. Observe that the  ring $C$ provided by Lemma~\APCref{lemma:add-rank-unit} over the prism $(B,J) = (\mathbf{Z}_p,(p))$ is explicitly given by
\[ C = \mathbf{Z}_p[w]\{\frac{w-1}{p}\}^{\wedge}_{(p)}.\]
Geometrically, via the adjunction between the Witt vector and forgetful functors, this tells us that $\rho_{\dR}^* G_{\WCart}$ is the group scheme over $\mathbf{Z}_p$ determined by the following functor on $p$-nilpotent rings:
\[ \rho_{\dR}^* G_{\WCart}(R) = \{ y \in W(R) \mid 1+py \in R^* \stackrel{[\cdot]}{\subset} W(R)^* \}, \]
where the Witt-vector $y$ is determined by the element $\frac{w-1}{p}$ in the $\delta$-ring $C$. The map $\log_{\Prism,\mathbf{Z}_p} = \rho_{\dR}^* \log_\Prism$ identifies (by construction) with the map
\[ \log_{\Prism,\mathbf{Z_p}}:\mathrm{Spf}(C) \to  \mathbf{V}(\mathcal{O}_{\mathrm{Spf}(\mathbf{Z}_p)}\{1\}),\]
of group schemes over $\mathbf{Z}_p$ determined by the element
\[ \log_\Prism(w)  \in C\{1\} = \mathbf{Z}_p[w]\{\frac{w-1}{p}\}^{\wedge}_{(p)}\{1\}.\]
As $(C,(p))$ is a crystalline prism, we may trivialize the Breuil--Kisin twist as in Remark~\APCref{BKcrystalline}. Under this trivialization, the formula in Corollary~\APCref{corollary:crystalline-formula-for-log} (3) allows us to rewrite the above element classically as
\[ \frac{\log(w)}{p} \in \mathbf{Z}_p[w]\{\frac{w-1}{p}\}^{\wedge}_{(p)}.\]
Using this description of $\rho_{\dR}^* \log_{\Prism}$, one can identify the map $\log_\Prism$ from Remark~\ref{PrismLogGroupSch} with the map $G_\Sigma \to \mathcal{O}_\Sigma\{1\}$ constructed in \cite{DrinfeldFormalGroup}.
\end{remark}

\newpage
\section{The relative prismatization}
\label{ss:RelPrism}

In this section, we introduce the relative Cartier--Witt stack in the setup of relative prismatic cohomology (Variant~\ref{RelativePrismatize}), and study the geometry of its Hodge--Tate locus in the smooth case (Proposition~\ref{RelativeHT}).

\begin{variant}[The relative prismatization]
\label{RelativePrismatize}
Let $(A,I)$ be a bounded prism, and let $X$ be a  $p$-adic formal scheme over $\overline{A}$. Then the relative prismatization $\WCart_{X/A}$ is the groupoid valued functor on $p$-nilpotent rings defined as follows: for any $p$-nilpotent ring $R$, the groupoid $\WCart_{X/A}(R)$ consists of the groupoid of maps $\alpha:A \to R$ carrying $I$ to a nilpotent ideal together with a map $\eta:\mathrm{Spec}(\overline{W(R)} ) \to X$ of derived formal $\overline{A}$-schemes (where $W(R)$ is regarded as a $\delta$-$A$-algebra via adjunction from the given map $A \to W(R)$). Forgetting $\eta$ yields a map $\WCart_{X/A} \to \mathrm{Spf}(A)$ that we refer to as the structure map. We then have a fibre square
\[ \xymatrix{ \WCart_{X/A} \ar[r] \ar[d] & \WCart_X \ar[d] \\
\mathrm{Spf}(A) \ar[r]^-{\rho_A} & \WCart_{\overline{A}}  }\]
where the map $\rho_A:\mathrm{Spf}(A) \to \WCart_{\overline{A}}$ comes from Construction~\ref{cons:pointprismatization} as $(A,I)$ is a prism over $\overline{A}$. Similarly, one has the relative Hodge-Tate stack
\[ \WCart_{X/A}^{\mathrm{HT}} = \WCart_{X/A} \times_{\WCart} \WCart^{\mathrm{HT}} \simeq \WCart_{X/A} \times_{\mathrm{Spf}(A)} \mathrm{Spf}(\overline{A}).\]
Moreover, the constructions $X \mapsto \WCart_{X/A}, \WCart_{X/A}^{\mathrm{HT}}$ are functorial in $X$, and preserve Tor-independent limits. 
\end{variant}

\begin{construction}[From the relative prismatic site to the relative prismatization]
\label{cons:relprismatizecompare}
Let $(A,I)$ be a bounded prism, and let $X$ be a  $p$-adic formal scheme over $\overline{A}$. For any prism $(B,IB) \in (X/A)_\Prism$, Construction~\ref{cons:pointprismatization} gives natural maps $\rho_{X,B}:\mathrm{Spf}(B) \to \WCart_X$ and $\mathrm{Spf}(B) \to \mathrm{Spf}(A)$ inducing canonically isomorphic maps $\rho_{\overline{A},B}:\mathrm{Spf}(B) \to \WCart_{\overline{A}}$. Consequently, we obtain map $\rho_{X/A,B}:\mathrm{Spf}(B) \to \WCart_{X/A}$. It is immediate from the definitions that this map is Frobenius equivariant. Varying through all objects of $(X/A)_\Prism$ and taking a limit, we obtain a natural Frobenius equivariant map
\[ \RGamma( \WCart_{X/A}, \mathcal{O}_{ \WCart_{X/A}}) \to \RGamma_\Prism^{\mathrm{site}}(X/A)\]
of commutative algebras in $\widehat{\mathcal{D}}(A)$, where the right hand side is the site-theoretic relative prismatic cohomology (see Notation~\APCref{notation:comparison-with-site-theoretic}).
\end{construction}

\begin{remark}[Relative and absolute prismatizations coincide over perfect prisms]
\label{RelativeAbsolutePerfect}
Let $(A,I)$ be a perfect prism, and let $X$ be a  $p$-adic formal scheme over $\overline{A}$. By Example~\ref{PerfectoidPrismatize}, the map $\rho_A: \mathrm{Spf}(A) \to \WCart_{\overline{A}} $ is an isomorphism. Consequently, by base change, we also see that the projection $\WCart_{X/A} \to \WCart_X$ is an isomorphism. 
\end{remark}

\begin{remark}[The absolute prismatization via the relative ones]
Let $X$ be a  bounded $p$-adic formal scheme. Write $f:X \to \mathrm{Spf}(\mathbf{Z}_p)$ for the structure map, giving rise to an induced map $\WCart_f:\WCart_X \to  \WCart$ on prismatizations. For any bounded prism $(A,I)$, the compatibility of $\WCart_{(-)}$ with finite limits (Remark~\ref{PrismatizeLimits}) and the base change formula in Variant~\ref{RelativePrismatize} show that
\[ \WCart_X \times_{\WCart,\rho_A} \mathrm{Spf}(A) \simeq \WCart_{X_{\overline{A}}/A}.\]
We shall later use this description in conjunction with Proposition~\APCref{proposition:DWCart-prism-description} to relate the pushforward $R \WCart_{f,*} \mathcal{O}_{\WCart_X}$ to the object $\mathcal{H}_\Prism(X)$ from Variant~\APCref{variant:prismatic-sheaf-globalized}.
\end{remark}

\begin{remark}[The diffracted Hodge stack as a relative Hodge-Tate stack]
\label{DiffHodHT}
Let $X$ be a  bounded $p$-adic formal scheme. By Proposition~\APCref{proposition:old-diagram}, the diffracted Hodge stack $X^{\DHod}$ (Construction~\ref{DiffHodgeStack}) is identified with the relative Hodge-Tate stack $\WCart_{X/A}^{\mathrm{HT}}$ (as stacks over $X$) for the prism $(A,I)  = (\mathbf{Z}_p \llbracket \tilde{p} \rrbracket, (\tilde{p}))$. 
\end{remark}

Our next goal is to provide a deformation-theoretic perspective on the relative Hodge-Tate stack that is useful in calculations.

\begin{notation}[Truncations of the kernel of $F$ on $W$]
Let $W_n$ and $W = \lim_n W_n$ denote the ring scheme of $n$-truncated Witt vectors and all Witt vectors respectively; we shall work in derived algebraic geometry, and refer to Notation~\ref{notation:Witt-vectors-revisited} and Remark~\ref{remark:restriction-maps} for the the derived analogues of these constructions.  For $n \geq 2$, consider the Frobenius map $F:W_n \to W_{n-1}$. This is naturally a map of $W$-algebras, and we define $W_n[F]$ as the kernel
\[ W_n[F] := \ker(W_n \xrightarrow{F} W_{n-1}).\] 
Thus, $W_n[F]$ is a $W$-module scheme. Moreover, the formula $xVy = V(Fx \cdot y)$ shows that $W_n[F]$ is annihilated by the ideal $VW \subset W$, and is thus naturally a module over $\mathbf{G}_a = W/VW$. By Variant~\APCref{Gasharpdef}, we have an identification $\mathbf{G}_a^\sharp = W[F] = \lim_n W_n[F]$ of group schemes. For any of these group schemes $G$, write $BG$ for the classifying stack of $G$-torsors in the $p$-quasisyntomic topology.
\end{notation}

\begin{lemma}[Cohomology of $\mathbf{G}_a^\sharp$ and {$W_n[F]$}]
\label{cohdimpdga}
Fix a $p$-complete animated ring $S$ and a finite projective $W(S)$-module bundle $E$.
\begin{enumerate}
\item For each $n \geq 2$, there is a natural isomorphism
\[ \RGamma(\mathrm{Spf}(S), E \otimes_W W_n[F])[1] \simeq \mathrm{Cone}\left(E \otimes_W W_n(S) \xrightarrow{\mathrm{id}_E \otimes F} E \otimes_W F_* W_{n-1}(S)\right).\]

\item There is a natural isomorphism
\[ \RGamma(\mathrm{Spf}(S), E \otimes_W W[F])[1] \simeq \mathrm{Cone}\left(E \otimes_W W(S) \xrightarrow{\mathrm{id}_E \otimes F} E \otimes_W F_* W(S)\right).\]

\end{enumerate}

In particular, the complexes in (1) and (2) are connective.
\end{lemma}

\begin{proof}
We explain the proof of (2) as (1) is analogous. Cosndier the exact sequence
\[ 0 \to W[F] \simeq \mathbf{G}_a^\sharp \to W \xrightarrow{F} F_* W \to 0 \]
as a sequence of $W(S)$-module schemes over $\mathrm{Spf}(S)$  via base change. Tensoring with the $W(S)$-module $E$, the claim reduces to showing the following (applied to $G = E$ and $G = F^*E$): we have $\RGamma(\mathrm{Spf}(S), G \otimes_{W(S)} W)\in D^{\leq 0}$ for any finite projective $W(S)$-module $G$. Writing $G$ as a retract of a finite free module, we reduce to the case $G=W(S)$, so we must check that $\RGamma(\mathrm{Spf}(S), W) \in D^{\leq 0}$, which is standard from the vanishing of quasi-coherent sheaf cohomology on the derived formal affine scheme $\mathrm{Spf}(W_n(S))$, the formula $W = \lim_n W_n$ where the inverse limit takes place over the restriction maps, and the fact that the restriction maps $W_{n+1}(S) \to W_n(S)$ are surjective on $\pi_0$ for all $n$. 
\end{proof}

\begin{remark}
It follows from Lemma~\ref{cohdimpdga} that for any $p$-complete animated ring $S$, we can identify
\[ B W[F](S) \simeq \RGamma(\mathrm{Spf}(S), W[F][1]),\]
and similarly after twisting by line bundles or replacing $W[F]$ with $W_n[F]$.
\end{remark}

\begin{remark}
Fix a $p$-complete ring $S$. For any fixed $n \geq 1$, Lemma~\ref{cohdimpdga} (1) shows that the complex $\RGamma(\mathrm{Spf}(S), W_n[F])$ is concentrated in degree $0$ exactly when $W_n(S) \xrightarrow{F} W_{n-1}(S)$ is surjective; the latter condition (imposed for all $n \geq 1$) has been studied by Davis-Kedlaya \cite{DavisKedlayaWittPerfect}, who dub it {\em Witt-perfectness}. Note that Witt-perfectness of $S$ is in general weaker than demanding surjectivity of $F:W(S) \to W(S)$; for instance, $\mathcal{O}_{\mathbf{C}_p}$ is Witt-perfect even though $F:W(\mathcal{O}_{\mathbf{C}_p}) \to W(\mathcal{O}_{\mathbf{C}_p})$ is not surjective (see \cite[Example 4.4]{DavisKedlayaWittPerfect}).
\end{remark}

The following construction is fundamental to the deformation-theoretic perspective on the Hodge-Tate stack:

\begin{construction}[A square-zero extension of $\mathbf{G}_a$ by {$B\mathbf{G}_a^\sharp\{1\}$} attached to any prism]
\label{sqzeroHTprism}
Let $(A,I)$ be a bounded prism. Given a $p$-complete animated $\overline{A}$-algebra $S$, the composition $A \to \overline{A} \to S$ lifts to a unique $\delta$-$A$-algebra map $A \to W(S)$. Thus, we obtain an $\overline{A}$-algebra map $\overline{W(S)} \to S$. We claim this map is a square-zero extension. More precisely:
\begin{itemize}
\item[$(\ast)$]
On the $\infty$-category of $p$-complete  animated $\overline{A}$-algebras, the map $\overline{W(S)} \to S$ admits a natural (in $S$) structure of a square-zero extension of $S$ by the connective $S$-complex $(B\mathbf{G}_a^\sharp\{1\})(S) \simeq \RGamma(\mathrm{Spf}(S), I \otimes_W W[F])[1] \in \mathcal{D}(S)$ (where the last equality comes from Lemma~\ref{cohdimpdga}).
\end{itemize}
To show $(\ast)$, it suffices to prove the analog for $W_n$ instead. In fact, we shall show the following slightly stronger assertion for each $n \geq 2$:
\begin{itemize}
\item[$(\ast_n)$]
On the $\infty$-category of all $p$-complete animated $\overline{A}$-algebras, the map $\overline{W_n(S)} \to S$ admits a natural (in $S$ and $n$) structure of a square-zero extension of $S$ by the connective $S$-complex $(BW_n[F])(S) \simeq \RGamma(\mathrm{Spf}(S), I \otimes_W W_n[F])[1] \in \mathcal{D}(S)$ (where the last equality comes from Lemma~\ref{cohdimpdga}).
\end{itemize}

The proof follows a familiar pattern: using compatibility with sifted colimits and descent, we can reduce to proving $(\ast_n)$ for a particularly small class of discrete rings $S$ where the claim is essentially obvious. 

Let $\mathcal{C}$ be the category of discrete $p$-complete $\overline{A}$-algebras $S$ with $F:W_n(S) \to W_{n-1}(S)$ surjective for all $n \geq 1$; in this paragraph, we prove the restriction of $(\ast_n)$ to the subcategory $\mathcal{C}$ for all $n$. For $S \in \mathcal{C}$, we claim that the map $\alpha:I_{W_n(S)} \to VW_n(S) \simeq F_* W_{n-1}(S)$ linearizes to an isomorphism 
\[ \widetilde{\alpha}:I_{W_n(S)} \otimes_{W_n(S)} F_* W_{n-1}(S) \simeq F_* W_{n-1}(S) \simeq VW_n(S).\] 
Granting this, using the surjectivity of $F:W_n(S) \to W_{n-1}(S)$, we learn that  
\[ \pi_0(\overline{W_n(S)}) \simeq W_n(S)/VW_n(S) \simeq S\]
and
\[ \pi_1(\overline{W_n(S)}) = I \otimes_A \ker(W_n(S) \xrightarrow{F} W_{n-1}(S)) \simeq I \otimes_A W_n(S)[F](S).\] 
The restriction of $(\ast_n)$ to $\mathcal{C}$ then follows from the general fact that any $1$-truncated animated ring is naturally and uniquely a square-zero extension of its $\pi_0$ by $\pi_1[1]$. To prove isomorphy of $\widetilde{\alpha}$, it  is enough to check surjectivity as any surjection of invertible modules is an isomorphism. Moreover, we can work Zariski locally on $\mathrm{Spf}(W_n(S))$, so we may assume $I=(d)$ is oriented; in this case, $\alpha(d) =V(u)$ for a unit $u \in W_n(S)$, so the surjectivity follows by unwinding definitions from the assumed surjectivity of $F:W_n(S) \to W_{n-1}(S)$. 

Next, as the functors $\RGamma(\mathrm{Spf}(-),I \otimes_W W_n[F])[1]$ are $p$-quasisyntomic sheaves, we deduce $(*_n)$ compatibly in $n$ for the category of $p$-completely smooth discrete $\overline{A}$-algebras via descent from the claim in the previous paragraph.

Finally, as the functor  $\RGamma(\mathrm{Spf}(-),I \otimes_W W_n[F])[1]$ from $p$-complete animated $\overline{A}$-algebras to $\widehat{\calD}(\overline{A})$ commutes with sifted colimits (by the description in Lemma~\ref{cohdimpdga}), the claim follows in general.
\end{construction}

\begin{remark}[Relating $\overline{W}$ to the obstruction to lifting modulo $I^2$]
\label{rmk:comparesqzero}
Fix a bounded prism $(A,I)$. In Construction~\ref{sqzeroHTprism}, we showed that presheaf $\overline{W(-)}$ of animated $\overline{A}$-algebras is naturally a square-zero extension of $\mathbf{G}_a$ by $\mathbf{G}_a^\sharp\{1\}[1]$ on the $\infty$-category of all $p$-nilpotent animated $\overline{A}$-algebras. The animated $\overline{A}$-algebra $\overline{W(-)}$ also admits a canonical deformation  to an $A/I^2$-algebra given by $W(-)/ I^2_{W(-)}$. Consequently, if we let $\mathcal{R}(-)$ denote the square-zero extension of $\mathbf{G}_a$ by $\mathbf{G}_a\{1\}[1]$ in $p$-nilpotent animated $\overline{A}$-algebras classifying the obstruction to lifting an $\overline{A}$-algebra to $A/I^2$, we obtain a natural map 
\[ \alpha:\overline{W(-)} \to \mathcal{R}(-)\]
of square-zero extensions of $\mathbf{G}_a$. This map induces a $\mathbf{G}_a$-module map $\mathbf{G}_a^\sharp\{1\}[1] \to \mathbf{G}_a\{1\}[1]$ on the ``ideals'' of the extension that one can check that this is a unit multiple of the standard map\footnote{Indeed, as $\mathrm{Hom}_W(\mathbf{G}_a^\sharp, \mathbf{G}_a) \simeq \mathrm{Hom}_W(\mathbf{G}_a^\sharp, \mathbf{G}_a) \simeq \overline{A}$ (see \cite[Lemma 3.8.1 (ii)]{drinfeld-prismatic}),  the map $\pi_1(\alpha)\{-1\}:\mathbf{G}_a^\sharp \to \mathbf{G}_a$ of $\mathbf{G}_a$-modules has the form $f \cdot \mathrm{can}$ for some $f \in \overline{A}$. We must show that $f$ is a unit. If not, then we can find a closed point $x$ of $\mathrm{Spec}(\overline{A}/p\overline{A})$ where $f$ vanishes. Base changing along the resulting $\delta$-map $A \to W(\kappa(x)_{\perf})$, we may then assume that $(A,I) = (W(k),(p))$ for a perfect field $k$ of characteristic $p$ and that the map $\pi_1(\alpha)$ is $0$. By deformation theory, the $\overline{W(-)}$-algebra map $\pi_0(\overline{W(-)}) \simeq \mathbf{G}_a \xrightarrow{\pi_0(\alpha)} \pi_0(\mathcal{R}) \simeq \mathbf{G}_a$ (which is simply the identity) lifts to a $\overline{W(-)}$-algebra map $\pi_0(\overline{W(-)}) \simeq \mathbf{G}_a \xrightarrow{\beta} \mathcal{R}$: the obstruction is the point of $\mathrm{Map}_{\mathbf{G}_a}(L_{\mathbf{G}_a/\overline{W(-)}}, \mathbf{G}_a\{1\}[2])$ determined by postcomposition of $\pi_1(\alpha)$ with the canonical map 
\[ L_{\mathbf{G}_a/\overline{W(-)}} \to \tau^{\geq 2}L_{\mathbf{G}_a/\overline{W(-)}} \simeq \left(\pi_1(\overline{W(-)})[1]\right)[1] \simeq \mathbf{G}_a^\sharp\{1\}[2].\]
But the existence of $\beta$ means that the square-zero extension $\mathcal{R} \to \mathbf{G}_a$ is split; by the moduli interpretation for $\mathcal{R}$, this implies that any $k$-algebra $S$ has a canonical lift to $W_2(k)$, which is clearly false.}. In other words, the square-zero extension $\mathcal{R} \in \mathrm{SqZeroExt}_{\overline{A}}(\mathbf{G}_a, \mathbf{G}_a\{1\}[1])$ is obtained from $\overline{W(-)} \in \mathrm{SqZeroExt}_{\overline{A}}(\mathbf{G}_a, \mathbf{G}_a^\sharp\{1\}[1])$ via pushout along a unit multiple of the standard map $\mathbf{G}_a^\sharp\{1\}[1] \to \mathbf{G}_a\{1\}[1]$; in fact, we expect that the phrase ``unit multiple of'' is not necessary. 
\end{remark}

Using the square-zero extension from Construction~\ref{sqzeroHTprism}, one can study the relative Hodge-Tate stack fairly explicitly.

\begin{proposition}[The Hodge-Tate gerbe]
\label{RelativeHT}
Fix a bounded prism $(A,I)$ and a smooth $p$-adic formal $\overline{A}$-scheme $X$. Then the Hodge-Tate structure map $\pi^{\mathrm{HT}}:\WCart_{X/A}^{\mathrm{HT}} \to X$ is naturally a gerbe banded by the flat $X$-group scheme $T_{X/\overline{A}}\{1\}^\sharp$. Moreover, this gerbe splits if $X$ is affine (and thus always Zariski locally on $X$). 
\end{proposition}

\begin{proof}
For the first part,  fix a  $p$-nilpotent $\overline{A}$-algebra $S$ and a point $\eta:\mathrm{Spec}(S) \to X$. Our task is to show that the fibre $F$ of 
\[ \pi^{\mathrm{HT}}(S):\WCart^{\mathrm{HT}}_{X/A}(S) = X(\overline{W(S)}) \to X(S)\]
over the point $\eta$ is a torsor for $B(T_X\{1\}^\sharp)(S)$, i.e,, that $F$  has a natural action of the group object
\[ B(T_X\{1\}^\sharp)(S) \simeq \mathrm{Map}_S(\eta^* \Omega^1_{X/\overline{A}}, B\mathbf{G}_a^\sharp\{1\}(S))\]
in groupoids,  and moreover this action is simply transitive if $F \neq \emptyset$. This follows from Construction~\ref{sqzeroHTprism}: the map $\overline{W(S)} \to S$ is a square-zero extension of $S$ by $B\mathbf{G}_a^\sharp\{1\}(S)$, so the fibre $F$ has the required structure by derived deformation theory.

For the second part, set $X=\mathrm{Spf}(R)$. By deformation theory, we can first choose a $(p,I)$-completely smooth $A$-algebra $\widetilde{R}$ lifting the $\overline{A}$-algebra $R$, and then choose a $\delta$-structure on $\widetilde{R}$ lifting that on $A$. For any $p$-nilpotent $\overline{A}$-algebra $S$, we then obtain  isomorphisms
\[ \mathrm{Map}_{\overline{A}}(R, S) \simeq \mathrm{Map}_A(\widetilde{R},S) \simeq \mathrm{Map}_{A,\delta}(\widetilde{R},W(S)),\]
where the second isomorphism comes from the $\delta$-structure on $\widetilde{R}$. Forgetting the $\delta$-structure and composing with the map $W(S) \to \overline{W(S)}$ then gives a natural map
\[  \mathrm{Spf}(R)(S) \simeq \mathrm{Map}_{\overline{A}}(R, S) \simeq \mathrm{Map}_{A,\delta}(\widetilde{R},W(S)) \to \mathrm{Map}_A(\widetilde{R}, \overline{W(S)}) \simeq \mathrm{Map}_{\overline{A}}(R, \overline{W(S)}) \simeq \WCart_{\mathrm{Spf}(R)/A}^{\mathrm{HT}}(S).\]
It follows from the construction that this map is inverse to $\pi^{\mathrm{HT}}(S)$, which proves (2). 
\end{proof}

\begin{remark}[The class of the relative Hodge-Tate gerbe]
\label{HTGerbeProperties}
Fix $(A,I)$ and $X$  as in Proposition~\ref{RelativeHT}. We then have the class $\alpha_{HT} \in \mathrm{H}^2(X, T_X\{1\}^\sharp)$ of the gerbe $\pi^{\mathrm{HT}}:\WCart_{X/A}^{\mathrm{HT}} \to X$. This class seems to be an interesting invariant of $X$. We record some observations about this class:
\begin{enumerate}
\item If $X$ admits a lift to a $(p,I)$-completely flat formal $\delta$-$A$-scheme, then $\alpha_{HT} = 0$; this follows from the proof of the second part of Proposition~\ref{RelativeHT}. Note that this is the case if $X$ is affine. 

\item The image of $\alpha_{\mathrm{HT}}$ in $\mathrm{H}^2(X, T_X\{-1\})$ (via the canonical map $T_X\{-1\}^\sharp \to T_X\{-1\}$) is, up to a unit multiple, the obstruction to lifting $X$ to $A/I^2$; this follows by the compatibility in Remark~\ref{rmk:comparesqzero}. In particular, if $\alpha_{HT} = 0$, then $X$ has a lift to $A/I^2$.

\item Assume we are in the crystalline case (i.e., so $I=(p)$, and thus the Breuil-Kisin twist can be ignored); we explain a (conjectural) partial converse to (1) in this case. As $A/p$ has characteristic $p$, one has a factorization $\mathbf{G}_a^\sharp \twoheadrightarrow \alpha_p \subset \mathbf{G}_a$ of the natural map of $A/p$-group schemes. Tensoring with $T_X$ and pushing out $\alpha_{HT}$ along this map gives a class $\alpha_F \in \mathrm{H}^2(X, T_X \otimes_{\mathbf{G}_a} \alpha_p)$. To understand this better, we use the identification
\[ \mathrm{H}^2(X, T_X \otimes_{\mathbf{G}_a} \alpha_p) \simeq \mathrm{H}^1(X, T_X \otimes_{\mathcal{O}_X} F_* \mathcal{O}_X/\mathcal{O}_X)\]
obtained by pushing forward the Kummer sequence for $\alpha_p$ to the Zariski site. In this case, the following seems quite plausible:
\begin{conjecture}
\begin{enumerate}
\item The class $\alpha_F \in  \mathrm{H}^1(X, T_X \otimes_{\mathcal{O}_X} F_* \mathcal{O}_X/\mathcal{O}_X)$ measures the failure of $F$-liftability of $X$ relative to $A$ (i.e., the failure to lift $X$ to $A/p^2$ compatibly with the Frobenius).
\item The image of $\alpha_F$ in $\mathrm{H}^2(X, T_X)$ under the map induced by the map $F_* \mathcal{O}_X/\mathcal{O}_X \to \mathcal{O}_X[1]$ classifying the extension $F_* \mathcal{O}_X$ is the obstruction class measuring the failure to lift $X$ to $A/p^2$. 
\item Write $X'$ for the Frobenius twist of $X$. Consider the image $\overline{\alpha}$ of $\alpha_F$ under the natural map
\begin{eqnarray*}
\mathrm{H}^1(X, T_X \otimes_{\mathcal{O}_X} F_* \mathcal{O}_X/\mathcal{O}_X) & \rightarrow &  \mathrm{H}^1(X, T_X \otimes_{\mathcal{O}_X} F_* \mathcal{O}_X/F_{A/p,*} \mathcal{O}_{X'}) \\
& \simeq & \mathrm{H}^1(X', T_{X'} \otimes B^1 \Omega^\bullet) \\
&  = & \mathrm{Ext}^1_{X'}(\Omega^1_{X'}, B^1 \Omega^\bullet),
\end{eqnarray*}
where $\Omega^\bullet = F_{X/(A/p),*} \Omega^\bullet_{X/(A/p)}$ is the de Rham complex of $X$ (regarded as a quasi-coherent complex on $X'$ via restriction of scalars along the relative Frobenius) and $B^1$ denotes the sheaf of $1$-cycles. Then the class $\overline{\alpha}$ coincides with the canonical class $\alpha_{can}$ in $\mathrm{Ext}^1_{X'}(\Omega^1_{X'}, B^1 \Omega^\bullet)$ arising from the complex $\Omega^\bullet$ via the Cartier isomorphism $\Omega^1_{X'} \simeq \mathcal{H}^1(\Omega^\bullet)$.
\end{enumerate}
\end{conjecture}
It is a standard fact that the class $\alpha_{can}$ appearing in part (c) above is known the satisfy the analog of parts (a) and (b) with $X$ replaced by the Frobenius twist $X'$ (see \cite[Appendix, Proposition 1]{MehtaSrinivas} for the case $A/p$ is perfect); from this optic, the conjecture can be regarded as a Frobenius descent of this standard fact. Note that property (2) above implies part (b) up to a unit multiple.
\end{enumerate}

Assuming the conjecture in (3), it follows that demanding the triviality of the Hodge-Tate gerbe $\pi^{\mathrm{HT}}:\WCart_{X/A}^{\mathrm{HT}} \to X$ puts rather strong constants on $X$. For instance, over any bounded prism $(A,I)$, this gerbe is non-trivial for any proper smooth curve of genus $\geq 2$: specializing to a crystalline prism, this follows as no such curve over a perfect field $k$ of characteristic $p$ is $F$-liftable.
\end{remark}

\begin{example}[The Hodge-Tate stack for $\mathbf{Z}/p^n$]
\label{HTZpn}
Fix an integer $n \geq 2$. We shall construct an isomorphism 
\begin{equation}
\label{HTZpnisom}
 \WCart_{\mathrm{Spec}(\mathbf{Z}/p^n)}^{\mathrm{HT}} \otimes \mathbf{F}_p \simeq \mathbf{G}_a^\sharp / \mathbf{G}_m^\sharp \otimes \mathbf{F}_p,
 \end{equation}
 of stacks over $\mathbf{F}_p$, where the quotient on the right is formed for the natural $\mathbf{G}_m^\sharp$-action on $\mathbf{G}_a^\sharp$. In other words, we shall identify the functor of points of $\WCart_{\mathrm{Spec}(\mathbf{Z}/p^n)}^{\mathrm{HT}}$ on $\mathbf{F}_p$-algebras $R$ with $\mathbf{G}_a^\sharp / \mathbf{G}_m^\sharp$. Using the isomorphism
\[ \WCart_{\mathrm{Spec}(\mathbf{Z}/p^n)}^{\mathrm{HT}} \simeq \mathrm{Spec}(\mathbf{Z}/p^n)^{\DHod} / \mathbf{G}_m^\sharp\]
from Construction~\ref{DiffHodgeStack}, it suffices to show that for any $\mathbf{F}_p$-algebra $R$, we have a $\mathbf{G}_m^\sharp(R)$-equivariant natural identification
\[ \mathrm{Spec}(\mathbf{Z}/p^n)^{\DHod}(R) \simeq \mathbf{G}_a^\sharp(R).\]
By definition, the left side is given by
\[\mathrm{Spec}(\mathbf{Z}/p^n)^{\DHod}(R) \simeq \mathrm{Map}(\mathbf{Z}/p^n, W(R)/V(1)),\]
where the mapping space is computed in $p$-complete animated rings. Now the animated rings $\mathbf{Z}/p^n$ and $W(R)/V(1)$ are obtained from $\mathbf{Z}_p$ and $W(R)$ by freely setting $p^n$ and $V(1)$ to be zero respectively. As $\mathbf{Z}_p$ is the initial object in $p$-complete animated rings, the above then simplifies to 
\[\mathrm{Spec}(\mathbf{Z}/p^n)^{\DHod}(R) \simeq \{ x \in W(R) \mid p^n = xV(1)\}.\]
Now $R$ is an $\mathbf{F}_p$-algebra, so $p^n=p^{n-1} V(1) = V(p^{n-1})$. Using $xV(1) = V(Fx)$ and injectivity of $V$, the above then simplifies to 
\[\mathrm{Spec}(\mathbf{Z}/p^n)^{\DHod}(R) \simeq \{ x \in W(R) \mid Fx=p^{n-1}\},\]
with the action of $\mathbf{G}_m^\sharp(R) \simeq W^*[F](R)$ given by scalar multiplication action of $W^*[F](R)$ on $W(R)$. It is easy to see that the right side above is a torsor for $\mathbf{G}_a^\sharp(R) = W[F](R)$ for the additive translation action of $W[F](R)$ on $W(R)$; moreover, $x=p^{n-1}$ solves $Fx=p^{n-1}$. Thus, explicitly, we have an isomorphism
\[ \mathbf{G}_a^\sharp(R) = W[F](R) \stackrel{a \mapsto a+p^{n-1}}{\simeq}  \{ x \in W(R) \mid Fx=p^{n-1}\}.\]
Using the fact that $n-1 \geq 1$, one then checks that the resulting bijection $\mathrm{Spec}(\mathbf{Z}/p^n)^{\DHod}(R) \simeq \mathbf{G}_a^\sharp(R)$ intertwines the natural $\mathbf{G}_m^\sharp(R)$-action on either sides, as wanted. For future reference, we remark that the composite map
\[ B\mathbf{G}_m^\sharp \otimes \mathbf{F}_p \xrightarrow{\text{origin}} \mathbf{G}_a^\sharp/\mathbf{G}_m^\sharp \otimes \mathbf{F}_p \simeq \WCart_{\mathrm{Spec}(\mathbf{Z}/p^n)}^{\mathrm{HT}} \otimes \mathbf{F}_p  \xrightarrow{can} \WCart^{\mathrm{HT}} \otimes \mathbf{F}_p \]
arising from the isomorphism \eqref{HTZpnisom} is simply the isomorphism from Theorem~\APCref{theorem:describe-HT}.
\end{example}

\begin{remark}[Deligne--Illusie via the Sen operator]
Fix a perfect field $k$ of characteristic $p$. In Remark~\APCref{remark:Deligne-Illusie}, we explained why the mod $p$ reduction of the map $\eta:B\mathbf{G}_m^\sharp \simeq \WCart^{\mathrm{HT}} \to \WCart$ gives rise to a Sen operator on the de Rham complex of a smooth $k$-scheme $X$ equipped with a lift to a smooth $p$-adic formal scheme $\mathfrak{X}/W(k)$; this gave a refinement of the Deligne-Illusie decomposition \cite{MR894379} in this case. Example~\ref{HTZpn} shows that the mod $p$ reduction of $\eta$ factors through the mod $p$ reduction of the canonical map $\WCart_{\mathbf{Z}/p^2} \to \WCart$. Using this, one can show that the analysis in Remarks~\APCref{remark:Deligne-Illusie}, \APCref{remark:Achinger}, and \APCref{rmk:Sen-nilpotent} goes through using only the data of a flat lift to $W_2(k)$, as in \cite{MR894379}. 
\end{remark}

\newpage
\section{Comparison with prismatic cohomology}
\label{ss:CompWCartPrism}

In this section, we prove the comparison between the stacky approach and the classical prismatic theory at the level of cohomology (Theorem~\ref{RelativePrismaticCohPrismatizeNonDer}) and crystals (Theorem~\ref{CrystalsCartWitt}) in the lci case; the main ingredients in the proof are the calculation of the relative Cartier--Witt stack in certain semiperfectoid cases (Lemma~\ref{RSPSpfd}) as well as a covering property of the Cartier--Witt stack (Lemma~\ref{PrismatizeCover}).

\begin{lemma}[The prismatization of a regular semiperfectoid]
\label{RSPSpfd}
Let $(A,I)$ be a perfect prism.  Let $R = A/J$ be for some ideal $J$ containing $I$ such that $J/I \subset A/I = \overline{A}$ is generated by a Koszul-regular sequence. Assume $R$ has bounded $p^\infty$-torsion, so $\Prism_R$ is a prism over $R$ by \cite[Example 7.9]{prisms}. Then the natural map
\[ \rho_{\Prism_R}:\mathrm{Spf}(\Prism_R) \to \WCart_{\mathrm{Spf}(R)}\]
from Construction~\ref{cons:pointprismatization} is an isomorphism of functors over $A$.
\end{lemma}

Corollary~\ref{PrismSemiPerf} shall extend the above result to general semiperfectoid rings $R$: for such $R$, the stack $\WCart_{\mathrm{Spf}(R)}$ on $p$-nilpotent rings is corepresented by $\mathrm{H}^0(\Prism_R)$. 

\begin{proof}
Fix a generator $d \in I$. Choose a sequence $\underline{x} := \{x_1,...,x_r\}$ in $J$ whose image in $J/I$ is a Koszul-regular generating set. We shall describe both the source and target of $\rho_{\Prism_R}$ in terms of these choices  and see that the descriptions coincide.

First, recall from \cite[Proposition 3.13, Example 7.9]{prisms} that $\Prism_R$ is the $(p,I)$-complete $\delta$-$A$-algebra obtained from $A$ by freely adjoining $\{\frac{x_1}{d},...,\frac{x_r}{d}\}$ (even in the animated world). Consequently, for any $(p,d)$-nilpotent $A$-algebra $S$, we have a natural identification
\[  \mathrm{Map}_A(\Prism_R,S) \simeq \mathrm{Map}_{A,\delta}(\Prism_R,W(S)) \simeq \prod_{i=1}^r \{ h \in W(S) \mid hd=x_i\}.\]
Next, by Example~\ref{PerfectoidPrismatize}, for any $(p,d)$-nilpotent $A$-algebra $S$, we have a natural identification
\[ \WCart_{\mathrm{Spf}(R)}(S) \simeq \mathrm{Map}_{\overline{A}}(R,\overline{W(S)}).\]
Using the fact that $R$ is obtained from $\overline{A}$ by freely setting $x_i=0$, we can rewrite this as 
\[  \WCart_{\mathrm{Spf}(R)}(S)  \simeq \prod_{i=1}^r \mathrm{Path}(\overline{W(S)}; x_i,0),\]
where the right side is the space of paths connecting $x_i$ to $0$ in the groupoid $\overline{W(S)}$ for each $i$ separately.  Using the description of $\overline{W(S)}$ as the Picard groupoid attached to the complex $\left(W(S) \xrightarrow{d} W(S)\right)$, one sees immediately that 
\[ \mathrm{Path}(\overline{W(S)}; x_i,0) \simeq \{ h \in W(S) \mid hd=x_i\},\]
which gives the lemma.
\end{proof}

\begin{variant}[The prismatization of a relative regular semiperfectoid]
\label{RelativeRSPComp}
Let $(A,I)$ be a bounded prism.  Let $S$ be a $\overline{A}$-algebra that can be written as a quotient $R/J$, where $R$ is $p$-completely flat over $\overline{A}$ with $A/(p,I) \to R/p$ being relatively perfect and $J \subset R$ is generated by a Koszul-regular sequence. Assume $S$ has bounded $p^\infty$-torsion, so $\Prism_{S/A}$ is a prism over $S$ by  \cite[Example 7.9]{prisms}. Then the natural map
\[ \mathrm{Spf}(\Prism_{S/A}) \to \WCart_{\mathrm{Spf}(S)/A}\]
is an isomorphism. To see this, set $B = \Prism_{R/A}$, so $(B,IB)$ is a faithfully flat and relatively perfect prism over $(A,I)$, and $S$ is endowed with the structure of an algebra over $R = B/IB$. As $L_{S/\overline{A}} \simeq L_{S/\overline{B}}$, the Hodge-Tate comparison then shows that $\Prism_{S/A} \simeq \Prism_{S/B}$ via the natural map. Similarly, by deformation theory, it is also clear that the natural map gives an isomorphism $\WCart_{\mathrm{Spf}(S)/A}  \simeq \WCart_{\mathrm{Spf}(S)/B}$ of presheaves of groupoids. We may then replace $(A,I)$ with $(B,IB)$ to assume that $S$ is  quotient of $\overline{A}$ by a Koszul-regular sequence. The proof of Lemma~\ref{RSPSpfd} now applies directly to show the claim.
\end{variant}

\begin{lemma}[Prismatization preserves $p$-quasisyntomic covers]
\label{PrismatizeCover}
Let $f:X \to Y$ be a $p$-quasisyntomic cover of bounded $p$-adic formal schemes. Then $\WCart_f:\WCart_X \to \WCart_Y$ is surjective locally in the flat topology.
\end{lemma}

\begin{proof}
As the formation of the prismatization commutes with Zariski localization (Remark~\ref{PrismatizeEtaleCov}), we may assume both $X$ and $Y$ are affine. Fix a $p$-nilpotent ring $R$ and a point $\eta \in \WCart_Y(R)$, corresponding to a Cartier-Witt divisor $(I \xrightarrow{\alpha} W(R))$ with a map $g:\mathrm{Spec}(\overline{W(R)}) \to Y$. We must show that there exists a faithfully flat map $R \to S$ and a point $\eta_S \in \WCart_X(S)$ lifting the image of $\eta$ in $\WCart_Y(S)$. Pulling back $f$ along $g$ gives a quasi-syntomic cover $\overline{W(R)}  \to \overline{W(R)}  {\otimes}^L_{\mathcal{O}(Y)} \mathcal{O}(X)$ of  $p$-nilpotent animated rings. By Proposition~\ref{prop:LiftQSyn}, there exists a $(p,I)$-completely faithfully flat $W(R) \to B$ of animated $\delta$-rings and a factorization $\overline{W(R)} \to \overline{W(R)}  {\otimes}^L_{\mathcal{O}(Y)} \mathcal{O}(X) \to \overline{B}$ with both maps being faithfully flat. Write $R \to S := B \widehat{\otimes}_{W(R)}^L R$ for the base change of $W(R) \to B$ along the restriction map $W(R) \to R$, so $R \to S$ is a faithfully flat map of $p$-nilpotent (discrete) rings. The natural $W(R)$-algebra map $B \to S$ refines uniquely to a $\delta$-$W(R)$-algebra map  $B \to W(S)$. These constructions are  summarized in the following diagram
\[ \xymatrix{ A \ar[r] \ar[d] & W(R) \ar[rr] \ar[d] & & B \ar[r] \ar[d] & W(S) \ar[d] \\
\overline{A} \ar[r] & \overline{W(R)} \ar[r] & \overline{W(R)}  {\otimes}^L_{\mathcal{O}(Y)} \mathcal{O}(X) \ar[r] & \overline{B}  \ar[r] & \overline{W(S)} \\
 & \mathcal{O}(Y) \ar[r] \ar[u] & \mathcal{O}(X) \ar[u] & & }\]
where all squares are pushout squares of animated rings, and the maps in the top row are $\delta$-maps. Moreover, the composition $W(R) \to W(S)$ in the top row is obtained by applying $W(-)$ to the map $R \to S$: indeed, the $\delta$-map $W(R) \to W(S)$ appearing in the top row is adjoint to the map $W(R) \to B \to W(S) \to S$, which also factors as $W(R) \to R \to S$ by definition of $S$.  The subdiagram
\[ \mathcal{O}(X) \to \overline{W(S)} \gets W(S) \]
then gives a point of $\WCart_Y(S)$, and the rest of the diagram shows that this point lifts $\eta$, as wanted.
\end{proof}

\begin{theorem}[Comparing relative prismatic cohomology with the relative prismatization]
\label{RelativePrismaticCohPrismatizeNonDer}
Let $(A,I)$ be a bounded prism. Let $X$ be a bounded $p$-adic formal $\overline{A}$-scheme. Assume that $X$ is $p$-completely lci over $\overline{A}$, i.e., for every affine open $\mathrm{Spf}(R) \subset X$, the $\overline{A}$-algebra $R$ can be written as $R_0/J$, where $R_0$ is a $p$-completely smooth $\overline{A}$-algebra, and $J \subset R_0$ is generated by a regular sequence.  Then the comparison map
\[ c_{X/A}:\RGamma(\WCart_{X/A}, \mathcal{O}_{\WCart_{X/A}}) \to \RGamma_{\Prism}^{\mathrm{site}}(X/A)\]
from Construction~\ref{cons:relprismatizecompare} is an isomorphism.
\end{theorem}

Theorem~\ref{PrismaticCohPrismatization} (2) and Remark~\ref{PrismatizeClassical} extend this result to all $p$-quasisyntomic $\overline{A}$-schemes $X$.

\begin{proof}
As the comparison map is defined globally and both sides form Zariski sheaves as $X$ varies, we may assume $X=\mathrm{Spf}(R)$ is affine. As $R$ is assumed to be $p$-completely lci over $\overline{A}$ with bounded $p^\infty$-torsion, we can find a $p$-quasisyntomic cover $R \to S$ such that each term of the \v{C}ech nerve $S^*$ of $R \to S$ satisfies the assumptions in Variant~\ref{RelativeRSPComp}. Functoriality of the comparison map then gives the following commutative diagram in $\mathcal{D}(A)$:
\[ \xymatrix{ \RGamma( \WCart_{\mathrm{Spf}(R)/A}, \mathcal{O}) \ar[r] \ar[d] & \RGamma_\Prism^{\mathrm{site}}(\mathrm{Spf}(R)/A) \simeq \Prism_{R/A} \ar[d] \\
\lim \RGamma( \WCart_{\mathrm{Spf}(S^*)/A}, \mathcal{O}) \ar[r] & \lim \RGamma_\Prism^{\mathrm{site}}(\mathrm{Spf}(S^*)/A) \simeq \lim \Prism_{S^*/A}, }\]
 where we have used  Theorem~\APCref{theorem:site-theoretic-equivalence} for the isomorphisms in the right column. Now the bottom horizontal map is an equivalence by  Variant~\ref{RelativeRSPComp}, while the right vertical map is an equivalence by $p$-quasisyntomic descent (Lemma~\APCref{lemma:qs-descent-relative}).  It is therefore enough to check that the left vertical map is an equivalence. For this, observe that $\WCart_{\mathrm{Spf}(S)/A} \to \WCart_{\mathrm{Spf}(R)/A}$ is a flat cover by Lemma~\ref{PrismatizeCover}. Moreover, the \v{C}ech nerve of this cover is exactly $\WCart_{\mathrm{Spf}(S^*)/A}$ by the compatibility of $\WCart_{-/A}$ with Tor-independent limits. The isomorphy of the left vertical then follows as $\RGamma(-,\mathcal{O})$ is a sheaf for the flat topology on $(p,I)$-nilpotent $A$-algebras.
\end{proof}

\begin{theorem}[Crystals via the Cartier-Witt stack]
\label{CrystalsCartWitt}
Fix a bounded prism $(A,I)$. Let $X$ be a bounded $p$-adic formal scheme over $\overline{A}$ satisfying the same assumptions as Theorem~\ref{RelativePrismaticCohPrismatizeNonDer}. Pullback along the maps in Construction~\ref{cons:relprismatizecompare} yields equivalences 
\[ \mathcal{D}_{qc}(\WCart_{X/A}) \simeq \widehat{\mathcal{D}}_{\mathrm{crys}}( (X/A)_\Prism, \mathcal{O}_\Prism) := \lim_{(B,IB) \in (X/A)_\Prism} \widehat{\mathcal{D}}(B) \]
and
\[ \mathcal{D}_{qc}(\WCart_{X/A}^{\mathrm{HT}}) \simeq \widehat{\mathcal{D}}_{\mathrm{crys}}( (X/A)_\Prism, \overline{\mathcal{O}}_\Prism) := \lim_{(B,IB) \in (X/A)_\Prism} \widehat{\mathcal{D}}(B/IB) \]
of $\infty$-categories, where the completions appearing on the right are with respect to $(p,I)$. 
\end{theorem}

\begin{proof}
This follows by similar argument to the one proving Theorem~\ref{RelativePrismaticCohPrismatizeNonDer}; we explain it for the first equivalence. Using the Zariski sheaf property for both sides, we reduce to $X=\mathrm{Spf}(R)$ with $R$ $p$-completely lci over $\overline{A}$ with bounded $p^\infty$-torsion. Choose a cover $R \to S^*$ as in the proof of Theorem~\ref{RelativePrismaticCohPrismatizeNonDer} to obtain the diagram
\[ \xymatrix{ \mathcal{D}_{qc}(\WCart_{\mathrm{Spf}(R)/A}) \ar[r] \ar[d] & \widehat{\mathcal{D}}_{\mathrm{crys}}( (\mathrm{Spf}(R)/A), \mathcal{O}_\Prism) \ar[d] \\
\varprojlim \mathcal{D}_{qc}(\WCart_{\mathrm{Spf}(S^*)/A}) \ar[r] &   \varprojlim  \widehat{\mathcal{D}}_{\mathrm{crys}}( (\mathrm{Spf}(S^*)/A), \mathcal{O}_\Prism) \simeq \varprojlim \widehat{\mathcal{D}}(\Prism_{S^*/A}) }\]
The isomorphism on the bottom right arises as $(\Prism_{S^i/A}, I\Prism_{S^i/A})$ provides an initial object of the relative prismatic site $(\mathrm{Spf}(S^i)/A)_\Prism$ for each $i$; Variant~\ref{RelativeRSPComp} then implies that the bottom horizontal map is an equivalence. Also, the left vertical map is an equivalence by flat descent for quasi-coherent complexes using the same reasoning as in Theorem~\ref{RelativePrismaticCohPrismatizeNonDer}, so it suffices to show the right vertical map is an equivalence. For this, we first observe that $\Prism_{S^*/A}$ is obtained by evaluating $\mathcal{O}_\Prism$ on the \v{C}ech nerve of the object $(\Prism_{S^0/A},I\Prism_{S^0/A}) \in (\mathrm{Spf}(R)/A)_\Prism$:  this follows as $\Prism_{-/A}$ commutes with colimits when regarded as a functor from $p$-complete animated $\overline{A}$-algebras to $(p,I)$-complete $E_\infty$-$A$-algebras by the Hodge-Tate comparison. By flat descent for crystals, it remains to check that $(\Prism_{S^0/A},I\Prism_{S^0/A}) \in (\mathrm{Spf}(R)/A)_\Prism$ covers the final object for the flat topology on $(\mathrm{Spf}(R/A)_\Prism$. Fix $(B,IB) \in (\mathrm{Spf}(R)/A)_\Prism)$. By \cite[Proposition 7.11]{prisms} (which is generalized in Proposition~\ref{prop:LiftQSyn}), there is a faithfully flat map $(B,IB) \to (C,IC)$ of prisms such that  $R \to \overline{C}$ factors as $R \to S^0 \xrightarrow{\alpha} \overline{C}$. As $(\Prism_{S^0/A},I\Prism_{S^0/A})$ is the final object of $(\mathrm{Spf}(S^0)/A)_\Prism$, the map $\alpha$ refines to a map $(\Prism_{S^0/A},I\Prism_{S^0/A}) \to (C,IC)$ in $(\mathrm{Spf}(S^0)/A)_\Prism$ and  thus also in $(\mathrm{Spf}(R)/A)_\Prism$; this shows that $\left(\mathrm{Spf}(\Prism_{S^0/A}) \to \ast\right)$ is indeed a cover of the final object for the flat topology on $(\mathrm{Spf}(R/A)_\Prism$.
\end{proof}

Let us use the above theorem to explicitly relate Hodge-Tate crystals to Higgs bundles; this relationship justifies the choice of name for the ``Higgs specialization'' in \cite[Remark 4.13]{prisms}.

\begin{corollary}[Hodge-Tate crystals and Higgs bundles]
\label{cor:HiggsVBHT}
Fix a bounded prism $(A,I)$, and let $X/\overline{A}$ be a smooth $p$-adic formal scheme. Assume that the relative Hodge-Tate stack $\WCart^{\mathrm{HT}}_{X/A}$ is identified with $BT_{X/\overline{A}}\{1\}^\sharp$  (possible if $X$ is affine by Proposition~\ref{RelativeHT}). Then we have a natural equivalence
\[ \mathrm{Vect}((X/A)_\Prism, \overline{\mathcal{O}}_\Prism) \simeq \mathrm{Vect}_{\mathrm{Higgs}}(X/\overline{A}; \Omega^1_{X/\overline{A}\{-1\}})\] 
of categories, where the target denotes the category of pairs $\{ E, \theta:E \to E \otimes \Omega^1_{X/\overline{A}}\{-1\} \}$, where $E$ is a vector bundle on $X$ and $\theta$ is Higgs field (i.e., an $\mathcal{O}_X$-linear map such that $\theta \wedge \theta = 0$) which is nilpotent modulo $p$. 
\end{corollary}

The equivalence constructed in the proof below depends on the chosen splitting of the Hodge-Tate gerbe, so the equivalence is only natural for morphisms that respect this splitting. The proof below also adapts to a similar statement for perfect complexes. In particular, for $X$ affine, one can compute the cohomology of a Hodge-Tate crystal of vector bundles via the de Rham complex of the associated Higgs bundle. This equivalence as well as the cohomological comparison was also recently observed by Tian  \cite{TianFiniteness}.

\begin{proof}
The second equivalence in Theorem~\ref{CrystalsCartWitt} restricts to an equivalence 
\[ \mathrm{Vect}((X/A)_\Prism, \overline{\mathcal{O}}_\Prism) \simeq \mathrm{Vect}(\WCart_{X/A}^{\mathrm{HT}}).\]
Using the chosen trivialization $\WCart_{X/A}^{\mathrm{HT}} \simeq BT_X\{1\}^\sharp$ of the Hodge-Tate stack, this then gives an equivalence
 \[ \mathrm{Vect}((X/A)_\Prism, \overline{\mathcal{O}}_\Prism) \simeq \mathrm{Vect}(BT_X\{1\}^\sharp).\]
To finish, it suffices to show that
\[ \mathrm{Vect}(BT_X\{1\}^\sharp) \simeq \mathrm{Vect}_{\mathrm{Higgs}}(X/\overline{A}; \Omega^1_{X/\overline{A}\{-1\}}).\] 
But this is simply an instance of the general that representations of a commutative group scheme can be identified with coherent sheaves on the Cartier dual, see Lemma~\ref{RelativeCartDual} for the explicit version we need here.
\end{proof}

\begin{lemma}
\label{RelativeCartDual}
Let $R$ be a commutative ring and let $E$ be a finite projective $R$-module. The identity functor on $R$-modules lifts to an equivalence of the following categories:
\begin{enumerate}
\item Finite projective $R$-modules $M$ equipped with an action of $\mathrm{Sym}^*(E)$ that is $0$ on $\mathrm{Sym}^{\geq n}(E)$ for $n \gg 0$.
\item Finite projective $R$-modules $M$ equipped with a coaction of $\Gamma_R^*(E^\vee)$.
\end{enumerate}
\end{lemma}
\begin{proof}
Let $A_n = \mathrm{Sym}^*(E)/\mathrm{Sym}^{\geq n}(E)$ be the truncated symmetric algebra, so $\{A_n\}_{n \geq 1}$ is a projective system of algebras in $\mathrm{Vect}(R)$. Write $A_n^\vee$ for the $R$-linear dual, so $\{A_n^\vee\}_{n \geq 1}$ is an inductive system of coalgebras in $\mathrm{Vect}(R)$. For formal reasons, the identity functor on $R$-modules gives an equivalence between $A_n$-modules in $\mathrm{Vect}(R)$ and $A_n^\vee$-comodules in $\mathrm{Vect}(R)$. Varying $n$, this gives 
\[ \colim_n \mathrm{Mod}_{A_n}(\mathrm{Vect}(R)) \simeq \colim_n \mathrm{CoMod}_{A_n^\vee}(\mathrm{Vect}(R)) = \mathrm{CoMod}_{\colim_n A_n^\vee}(\mathrm{Vect}(R)),\]
where the last equivalence follows from the $R$-finiteness of the comodules involved. The lemma now follows as the category in (1) is the left side above by definition, while the category in (2) is the right side above: we have $\colim_n A_n^\vee = \Gamma^*_R(E^\vee)$ as coalgebras since graded $R$-linear duality carries symmetric algebras to divided power coalgebras (see \cite[Appendix A]{BOgusCrys}).
\end{proof}

\newpage
\section{Derived relative prismatization}
\label{ss:DerRelPrism}

Fix a bounded prism $(A,I)$. In this section, we discuss the relative prismatization functor for derived schemes, and in the process describe the (derived, and thus classical) prismatization of a qrsp ring. The main comparison results (Theorem~\ref{PrismaticCohPrismatization} and Remark~\ref{PrismatizeClassical}) extend those in \S \ref{ss:CompWCartPrism} to their natural generality; as a byproduct, we also obtain a simple perspective on the Hodge--Tate comparison for prismatic cohomology (Remark~\ref{HTDerivedPrismatize}).

In this section all operations (quotients, tensor products, limits, colimits, etc.) are intepreted in the derived sense.

\subsection{Definitions}

\begin{construction}[Derived relative prismatization]
\label{RelPrismDerived}
For any derived  $p$-adic formal $\overline{A}$-scheme $X$, we define a presheaf $\WCart_{X/A}$ on $(p,I)$-nilpotent animated $A$-algebras as follows:
\[ \WCart_{X/A}(B) = X(\overline{W(B)}),\]
where $W(B)$ is viewed as an animated $\delta$-$A$-algebra via adjunction from the structure map $A \to B$, and the mapping space on the right is computed in derived formal $\overline{A}$-schemes. Note that $\WCart_{X/A}(-)$ is a sheaf for the \'etale topology: this follows from the fact that the base change maps give equivalences
\[ \mathrm{Spec}(B)_{\et} \gets \mathrm{Spf}(W(B))_{\et} \to \mathrm{Spec}(\overline{W(B)})_{\et} \]
of \'etale sites for an $(p,I)$-nilpotent animated $A$-algebra $B$ together with the \'etale sheaf property of $X$ itself; see Lemma~\ref{SheafFlatWCart} for a stronger statement when $X$ is qcqs. Moreover, the functor $X \mapsto \WCart_{X/A}$ preserves arbitrary limits for formal reasons.
\end{construction}

\begin{example}[Derived relative prismatizations in the \'etale case]
\label{DerivedPrismatizeEtale}
Let $R/\overline{A}$ be a $p$-complete animated $\overline{A}$ which is $p$-completely \'etale, i.e., the $p$-completion $L_{R/\overline{A}}^{\wedge}$ is $0$. By deformation theory, $R$ admits a unique lift to a $(p,I)$-complete animated $A$-algebra $\widetilde{R}$ which is $(p,I)$-completely \'etale, i.e., the $(p,I)$-completion  $L_{\widetilde{R}/A}^{\wedge}$ vanishes. We claim that $\WCart_{\mathrm{Spf}(R)/A}$ identifies with $\mathrm{Spf}(\widetilde{R})$, where $\widetilde{R}$ is endowed with the $(p,I)$-adic topology. For this, observe that for any $(p,I)$-nilpotent animated $A$-algebra $B$, the restriction map
\[ W_{n+1}(B) \to W_n(B) \]
has a nilpotent kernel: as $W_n(-)$ commutes with sifted colimits, this follows from Lemma~\ref{WittSquareZero} (2) as one can write $B$ as the geometric realization of a simplicial $A/(I,p^m)$-algebras (for some $m$). In particular, each $W_n(B)$ is a $(p,I)$-nilpotent animated $A$-algebra as well, and the fibre of the above the map is a $W_n(B)$-module and thus also $(p,I)$-nilpotent. Thanks to this observation and the $(p,I)$-complete \'etaleness of $A \to \widetilde{R}$, we learn that for any $(p,I)$-nilpotent animated $A$-algebra $B$, the natural maps
\[ \mathrm{Map}_A(\widetilde{R}, B) \gets \lim_n \mathrm{Map}_A(\widetilde{R}, W_n(B)) \to \lim_n \mathrm{Map}_A(\widetilde{R}, \overline{W_n(B)}) \gets \mathrm{Map}_A(\widetilde{R}, \overline{W(B)}) \simeq \mathrm{Map}_{\overline{A}}(R, \overline{W(B)})\]
are all equivalences. Comparing the first and last term then gives the claim. 
\end{example}

\begin{lemma}[Sheafyness of $\WCart_{X/A}$]
\label{SheafFlatWCart}
For any derived qcqs $p$-adic formal $\overline{A}$-scheme $X$, the presheaf $\WCart_{X/A}$ is a sheaf for the flat topology on $(p,I)$-nilpotent animated $A$-algebras. 
\end{lemma}
\begin{proof}
Fix a faithfully flat map $B \to C$ of  $(p,I)$-nilpotent animated $A$-algebras. Write $B \to C^*$ for its \v{C}ech nerve.  We must show that
\[ \alpha:X(\overline{W(B)}) \to \varprojlim X(\overline{W(C^*)}) \]
is an equivalence. We shall prove this statement when $B$ (and hence $C$) is discrete, and then deduce the general statement by deformation theory. Thus, assume until further notice that $B$ and $C$ are discrete.

As $B$ and $C$ are discrete, we may also assume by adjunction that $X$ is discrete. In particular, both the source and target of $\alpha$ denote mapping spaces computed in ordinary schemes.  The full faithfulness of $\alpha$ then follows by Tannaka duality\footnote{We are using the spectral version \cite[\S 9]{LurieSAG} of Tannaka duality here. A priori, this only describes mapping spaces in spectral algebraic geometry in terms of functors on perfect complexes. However, the functor from derived schemes to spectral schemes is fully faithful on derived schemes with a $1$-truncated structure sheaf, so a posteriori we obtain descriptions of mapping spaces in derived algebraic geometry as well.} once we know  the full faithfulness of
\[ D_{\perf}(\overline{W(B)}) \to \varprojlim D_{\perf}(\overline{W(C^*)}).\]
The latter can be checked on unit objects, so it would follow by reduction modulo $I$ if we knew that $W(B) \simeq \varprojlim W(C^*)$. But the latter follows by reduction to the analogous statement for $W_n(-)$, which can be proven  by induction on $n$, using the standard exact triangles building the $W_n$ functor from the $W_{n-1}$ functor and the identity functor. 

For essential surjectivity of $\alpha$, fix some map $\pi:\mathrm{Spec}(\overline{W(C^*)}) \to X$. We must show that $\pi$ is induced by a map $\mathrm{Spec}(\overline{W(B)}) \to X$. This statement is obvious when $X$ is affine since $W(B) \simeq \varprojlim W(C^*)$ and similarly after passing to (animated) quotients by $I$; we shall explain how to reduce to this case.  Fix a qc open $V \subset X$. Its preimage along $\pi$ is a Cartesian qc open $V^* \subset \mathrm{Spec}(\overline{W(C^*)})$. Now for any $(p,I)$-nilpotent animated $A$-algebra $D$, base change gives  equivalences 
\[ \mathrm{Spec}(D)_{\et} \gets \mathrm{Spf}(W(D))_{\et} \to \mathrm{Spec}(\overline{W(D)})_{\et} \]
of \'etale sites, and similarly for Zariski sites. Using this equivalence for the cosimplicial $(p,I)$-nilpotent  $A$-algebra $C^*$, we can transport $V^*$ across the equivalence to obtain a Cartesian qc open $U^* \subset \mathrm{Spec}(C^*)$. By faithful flatness of $B \to C$ and flat descent for qc open subsets, any such qc open is the pullback of a qc open $U' \subset \mathrm{Spec}(B)$. Using the above equivalence of sites for $B$ now, we learn that $V^*$ is the preimage of a qc open $V' \subset \mathrm{Spec}(\overline{W(B)})$. By the affine case of the theorem, we obtain a unique map $V' \to V$ inducing $\pi$. Glueing these constructions together for variable $V \subset X$ then gives the desired map $\mathrm{Spec}(\overline{W(B)}) \to X$ inducing $\pi$.

It remains to explain how to pass from the discrete case to the general case via deformation theory, so assume now that $B$ is an arbitrary $(p,I)$-nilpotent animated $A$-algebra. Then for each integer $n \geq 1$, the natural animated $A$-algebra map $W(\tau_{\leq n} B) \to W(\tau_{\leq n-1} B)$ has fibre $F_n(B)$ concentrated in homological degree $n$: the fibre identifies with $\prod_{\mathbf{N}} \pi_n(B)[n]$ as a space via the Witt component maps. In particular, this map naturally a square-zero extension of $W(\tau_{\leq n-1} B)$ by $F_n(B)$ in animated $A$-algebras. Reducing modulo $I$ shows that the map $\overline{W(\tau_{\leq n} B)}  \to \overline{W(\tau_{\leq n-1} B)}$ is naturally a square-zero extension of $\overline{W(\tau_{\leq n-1} B)}$ by $F_n(B) \otimes_A \overline{A}$ in animated $\overline{A}$-algebras. Applying a similar analysis to the Postnikov tower for $C^*$, the claim follows by deformation theory once we observe that that the map $W(\tau_{\leq n}(B \to C))$ yields an isomorphism $F_n(B) \simeq \varprojlim F_n(C^*)$ by the flatness assumption on $B \to C$.
\end{proof}

\begin{warning}[Derived prismatization might differ the classsical one]
\label{Warn:DerNotClass}
Assume $X$ is an ordinary  $p$-adic formal $\overline{A}$-scheme; for the purpose of this warning, write $X^{\mathrm{der}}$ for the result of viewing $X$ as a derived formal $\overline{A}$-scheme. Construction~\ref{RelPrismDerived} gives a presheaf $\WCart_{X^{\mathrm{der}}/A}$ on all $(p,I)$-nilpotent animated $A$-algebras whose restriction to the discrete ones coincides with the relative prismatization $\WCart_{X/A}$ from Variant~\ref{RelativePrismatize}. Write  $\mathcal{L}( \WCart_{X/A})$ for the extension of  $\WCart_{X/A}$ to a presheaf on all $(p,I)$-nilpotent animated $A$-algebras (via left Kan extension of the usual inclusion on affine objects), so there is a natural map $\mathcal{L}(\WCart_{X/A}) \to \WCart_{X^{\mathrm{der}}/A}$. In general, this map fails to be an isomorphism or even identify the target with the flat sheafification of the source; for instance, one can show\footnote{To see this, observe that $\RGamma(\mathcal{L}(\WCart_{X/A}), \mathcal{O}) \simeq \RGamma(\WCart_{X/A}, \mathcal{O})$ is cococonnective as it is the cohomology complex of a sheaf of discrete rings in a topos. On the other hand, by Theorem~\ref{PrismaticCohPrismatization} (1)), $\RGamma( \WCart_{X^{\mathrm{der}}/A}, \mathcal{O})$ is identified with the derived crystalline cohomology of $\mathbf{F}_p[x,y]/(x,y)^2$, and thus has cohomology in (infinitely many) negative degrees by the derived Cartier isomorphism.} this disagreement happens over $(A,I) = (\mathbf{Z}_p,(p))$ with $X=\mathrm{Spec}(\mathbf{F}_p[x,y]/(x,y)^2)$. In particular, the results we shall soon prove on $\RGamma( \WCart_{X^{\mathrm{der}}/A}, \mathcal{O})$ (such as the comparison with prismatic cohomology in Theorem~\ref{PrismaticCohPrismatization} (1)) do not apply to $\RGamma( \WCart_{X/A}, \mathcal{O})$ in general. 

On the positive side, if the affine opens in $X$ satisfy condition $(\ast)$ from Remark~\APCref{remark:hypothesis-for-discreteness}, then the natural map $\mathcal{L} (\WCart_{X/A}) \to \WCart_{X^{\mathrm{der}}/A}$ exhibits the target as the flat sheafification of source  (see Remark~\ref{PrismatizeClassical}); in particular, this  extends Theorem~\ref{RelativePrismaticCohPrismatizeNonDer} to all such $X$'s. 
\end{warning}

We shall need the following variant of Lemma~\ref{PrismatizeCover}:

\begin{proposition}[Prismatization preserves $p$-quasisyntomic covers]
\label{PrismatizeCover2}
Let $f:X \to Y$ be a $p$-quasisyntomic cover of derived $p$-adic formal $\overline{A}$-schemes. Then $\WCart_f:\WCart_{X/A} \to \WCart_{Y/A}$ is surjective locally in the flat topology.
\end{proposition}
\begin{proof}
This follows by the same proof as the one given for Lemma~\ref{PrismatizeCover}.
\end{proof}

\subsection{The derived Hodge-Tate stack}

\begin{construction}[The derived Hodge-Tate stack]
\label{DerivedHTStack}
For any derived  $p$-adic formal $\overline{A}$-scheme $X$, we define 
\[ \WCart_{X/A}^{\mathrm{HT}} := \WCart_{X/A} \times_{\mathrm{Spf}(A)} \mathrm{Spf}(\overline{A}).\] 
Since we are working with derived schemes, we have\footnote{Given a presheaf $F$ on $(p,I)$-nilpotent animated $A$-algebras, we can write $F = \colim_{R \in  \mathcal{C}_F} \mathrm{Spec}(R)$ tautologically as the colimit of its $\infty$-category $\mathcal{C}_F$ of points, so we have $\RGamma(F,\mathcal{O}) \simeq \varprojlim_{R \in \mathcal{C}_F} R$. Its pullback $F \times_{\mathrm{Spf}(A)} \mathrm{Spf}(\overline{A})$ to $p$-nilpotent $\overline{A}$-algebras can then be described as $\colim_{R \in  \mathcal{C}_F} \mathrm{Spec}(R/t)$, yielding 
\[ \RGamma(F \times_{\mathrm{Spf}(A)} \mathrm{Spf}(\overline{A}), \mathcal{O}) \simeq \varprojlim_{R \in \mathcal{C}_F} R/t \simeq (\varprojlim_{R \in \mathcal{C}_F} R)/t \simeq \RGamma(F,\mathcal{O})/t.\]
Note that it is critical here that the operation of reduction mod $t$ is intepreted everywhere in the animated sense, and thus commutes with (derved) limits. In particular, this argument does not work in ordinary algebraic geometry.} the base change formula
\begin{equation}
\label{DerHTBC}
 \RGamma(\WCart_{X/A}^{\mathrm{HT}}, \mathcal{O}) \simeq \RGamma(\WCart_{X/A}, \mathcal{O}) \otimes^L_A \overline{A}
 \end{equation}
for derived global sections. Moreover, the proof of Proposition~\ref{RelativeHT} applies in the animated context as Construction~\ref{sqzeroHTprism} was formulated and proven for all $p$-nilpotent animated $\overline{A}$-algebras. Thus, we learn: for any smooth $p$-adic formal $\overline{A}$-scheme $X$, the structure map $\WCart_{X/A}^{\mathrm{HT}} \to X$ is a gerbe banded by $T_{X/\overline{A}}\{1\}^\sharp$ that splits if $X$ comes endowed with a lift to a smooth $(p,I)$-adic formal $A$-scheme equipped with a $\delta$-structure. 
\end{construction}

\begin{remark}[A direct proof of the Hodge-Tate comparison via the prismatization]
\label{HTDerivedPrismatize}
Fix a $p$-completely smooth $\overline{A}$-algebra $R$. Using results from prismatic cohomology (especially the Hodge-Tate comparison and its consequences such as Andr\'e's lemma), we shall prove in the forthcoming Theorem~\ref{PrismaticCohPrismatization} that $\Prism'_{R/A} := \RGamma(\WCart_{\mathrm{Spf}(R)/A}, \mathcal{O})$ is naturally identified with $\Prism_{R/A}$. On the other hand, it is possible to prove the basic results on prismatic cohomology directly for $\Prism'_{R/A}$ using the stacky perspective, with arguably more conceptual proofs. Let us sketch how to do so for the Hodge-Tate comparison. 

First,  using \eqref{DerHTBC}, we construct the comparison map
\[ c_R:(\Omega^*_{R/\overline{A}}, d) \to (\mathrm{H}^*(\Prism'_{R/A} \otimes_A I^*/I^{*+1}),\beta_I)\]
of commutative differential graded algebras as in \cite[Construction 4.9]{prisms}; when $p=2$, one needs to know that the target is a strict cdga (i.e., odd degree elements square to $0$), which will actually follow from the subsequent discussion. To prove this map is an isomorphism, we are allowed to make choices, so choose a lift of $R$ to $(p,I)$-completely smooth $A$-algebra $\widetilde{R}$ endowed with a $\delta$-structure. As mentioned in Construction~\ref{DerivedHTStack}, the map $\WCart_{\mathrm{Spf}(R)/A}^{\mathrm{HT}} \to \mathrm{Spf}(R)$ is the trivial gerbe banded by $T_{\mathrm{Spf}(R)/\overline{A}}\{1\}^\sharp$ (with trivialization determined by the choice $\widetilde{R}$ with its $\delta$-structure), so we get a natural isomorphism
\[ \Prism_{R'/A} \otimes_A \overline{A} \simeq \RGamma(\WCart_{\mathrm{Spf}(R)/A}^{\mathrm{HT}},\mathcal{O}) \simeq \RGamma(BT_{\mathrm{Spf}(R)/\overline{A}}\{1\}^\sharp, \mathcal{O})\]
of commutative algebras in $\mathcal{D}(R)$. By Lemma~\ref{CohRingBEsharp}, we have an  isomorphism
\[ \Prism'_{R/A} \otimes^L_A \overline{A} \simeq \bigoplus_i \Omega^i_{R/\overline{A}}\{-i\}[-i]\]
of $E_\infty$-$R$-algebras. One can then check that the induced isomorphism on cohomology rings agrees (after Breuil-Kisin untwisting) with the map $c_R$, thus showing the latter is an isomorphism, and thus proving the Hodge-Tate comparison. 
\end{remark}

\begin{lemma}
\label{CohRingBEsharp}
Given a commutative ring $R$ and a finite projective $R$-module $E$, there is a natural isomorphism
\[ \RGamma(BE^\sharp, \mathcal{O}) \simeq \bigoplus_i \wedge^i E^\vee[-i]\]
of $E_\infty$-$R$-algebras.
\end{lemma}
\begin{proof}
The induced statement on cohomology rings is standard and we omit to the proof. To obtain formality (i.e., pass from cohomology rings to $E_\infty$-$R$-algebras), observe that there is a natural $\mathbf{G}_m$-action on $E$ (giving $E$ weight $1$), which induces an action on all objects encountered above, so $\RGamma(BE^\sharp, \mathcal{O})$ is naturally a $\mathbf{G}_m$-equivariant $E_\infty$-$R$-algebra. It is therefore enough to show that for any $\mathbf{G}_m$-equivariant $E_\infty$-$R$-algebra $C$ with the property that $H^i(C)$ is a flat $R$-module with $\mathbf{G}_m$-weight $-i$, there is a unique $\mathbf{G}_m$-equivariant isomorphism $C \simeq \oplus_i H^i(C)[-i]$ of $E_\infty$-$R$-algebras inducing the identity on cohomology. This follows by weight considerations: the $\infty$-category $\mathcal{C}_1$ of $\mathbf{G}_m$-equivariant $R$-complexes $K$ with $H^i(K)$ being $R$-flat with $\mathbf{G}_m$-weight $-i$ is an ordinary category,  the symmetric monoidal cohomology functor identifies $\mathcal{C}_1$ with the ordinary category $\mathcal{C}_2$ of graded $R$-modules $M$ where $M_i$ is $R$-flat with $\mathbf{G}_m$-weight $-i$, and the inverse $\mathcal{C}_2 \to \mathcal{C}_1$ is given by sending a graded $R$-module $\oplus_i M_i \in \mathcal{C}_2$ to the complex $\oplus_i M_i[-i] \in \mathcal{C}_1$. 
\end{proof}

\begin{remark}[Understanding crystalline cohomology via ring stacks]
The ring stack perspective on $p$-adic cohomology theories emphasized by Drinfeld \cite{drinfeld-stacky, drinfeld-prismatic, DrinfeldFormalGroup} makes certain structures easy to see. We have already seen one example in our analysis of the Hodge-Tate gerbe (Proposition~\ref{RelativeHT}), which relies crucially on the square-zero extension from Construction~\ref{sqzeroHTprism}. Another basic structural result in this direction is the isomorphism 
\begin{equation}
\label{RingStackCrys} \mathbf{G}_a/ \mathbf{G}_a^\sharp \simeq   W/ p
\end{equation}
of ring stacks on $p$-nilpotent rings. This follows from the sequence of identifications of ring stacks
\[ \mathbf{G}_a/ \mathbf{G}_a^\sharp \simeq W/ (VW \oplus \mathbf{G}_a^\sharp) \stackrel{F}{\simeq} W / FVW \simeq W/ p\]
where the first identification comes from writing $\mathbf{G}_a = W/VW$, the second is induced by the Frobenius on $W$ and the observation that $F_* W = W/\mathbf{G}_a^\sharp$, and the last uses $FV=p$ (see \cite[Proposition 3.5.1]{drinfeld-prismatic}). Let us briefly summarize why \eqref{RingStackCrys} encodes some important features of crystalline cohomology. 
\begin{enumerate}
\item Given a $\mathbf{Z}_p$-flat $p$-adic formal scheme $X$, one may define its crystallization $X^{\crys}$ as the presheaf on all $p$-nilpotent rings defined by the formula
\[ X^{\crys}(S) = X\left((\mathbf{G}_a/^L \mathbf{G}_a^\sharp)(R)\right).\] 
There is a natural map $\epsilon:X \to X^{\crys}$ induced by the map $\mathbf{G}_a \to  \mathbf{G}_a/ \mathbf{G}_a^\sharp$ of ring stacks. When $X$ is a smooth formal $W(k)$-scheme for a perfect field $k$, one can show that $\epsilon$ is a quasisyntomic surjection, and the $n$-th term of its \v{C}ech nerve identifies with the PD-envelope of $X \subset X^n$. By \v{C}ech-Alexander theory, it follows that 
\begin{equation}
\label{CrysCompStack}
 \RGamma(X^{\crys}, \mathcal{O}) \simeq \RGamma_{\crys}(X/W(k)),
 \end{equation}
justifying the notation. To proceed further, we use the crystalline-de Rham comparison to rewrite the previous isomorphism as 
\begin{equation}
\label{dRCompStack}
 \RGamma(X^{\crys}, \mathcal{O}) \simeq \RGamma_{dR}(X/W(k)).
 \end{equation}
Now the ring stack $\mathbf{G}_a/\mathbf{G}_a^\sharp$ is annihilated by $p$: this is clear from \eqref{RingStackCrys} and may also be seen directly by observing that $p \in \mathbf{Z}_p = \mathbf{G}_a(\mathbf{Z}_p)$ lies in $\mathbf{G}_a^\sharp(\mathbf{Z}_p) \subset \mathbf{G}_a(\mathbf{Z}_p)$. It follows that the stack $X^{\crys}$ only depends on the mod $p$ reduction $X_{p=0}$ of $X$. Via \eqref{dRCompStack}, this provides a stacky explanation for the fact that the de Rham cohomology complex $\RGamma_{dR}(X/W(k))$ only depends on $X_{p=0}$ and not its lift $X$. 

\item Fix a perfect field $k$ of characteristic $p$ and a smooth $k$-scheme $Y$. Define the presheaf $Y^{\crys}$ on $k$-algebras via
\[ Y^\crys(R) = Y((\mathbf{G}_a/^L \mathbf{G}_a^\sharp)(R)).\] 
The isomorphism \eqref{RingStackCrys} yields an identification
\[ Y^{\crys} \simeq \WCart_{Y/W(k)}^{\mathrm{HT}}\]
of presheaves on $k$-algebras. Thus, the quasi-coherent derived $\infty$-category $\mathcal{D}_{qc}(Y^{\crys})$ of $Y^{\crys}$ identifies with the $\infty$-category $\mathcal{D}( (Y/W(k))_\Prism, \overline{\mathcal{O}}_\Prism)$ of Hodge-Tate crystals on $(Y/W(k))_\Prism$ from Theorem~\ref{CrystalsCartWitt}. In particular, as in Corollary~\ref{cor:HiggsVBHT}, the category $\mathrm{Vect}(Y^{\crys})$ of vector bundles on $Y^{\crys}$ may be described via Higgs bundles. 

On the other hand, the \v{C}ech-Alexander argument given to justify \eqref{CrysCompStack} in (1) above shows that the category $\mathrm{Vect}(Y^{\crys})$ can identified with the category of crystals on $(Y/k)_\crys$, i.e., with the category $\mathrm{Vect}^{\nabla}(Y/k)$ of vector bundles on $Y/k$ equipped with a flat connection relative to $k$ whose $p$-curvature is nilpotent. 

Combining the previous two paragraphs gives the promised relationship between flat vector bundles and Higgs bundles. This relationship has been studied in depth in upcoming work of Ogus.

\end{enumerate}
\end{remark}

\subsection{Affineness of $\WCart_{X/A}$}

Our next goal is to explain why the prismatization $\WCart_{X/A}$ is an ``affine stack'' (see \cite{ToenAffine} for the non-derived analog) when $X$ is affine; to even formulate the statement, we need the notion of a ``non-connective animated ring'' that was recently discovered by Mathew.

\begin{notation}[Derived rings, following Mathew]
\label{not:DerivedRings}
For any (animated) ring $R$, write $\mathrm{DAlg}_R$ for Mathew's $\infty$-category of derived $R$-algebras (see the exposition in \cite{RaskitdR}). Recall that this notion extends the theory of animated rings to the non-connective setting. In particular, any derived $R$-algebra has an underlying $E_\infty$-$R$-algebra, the connective derived $R$-algebras exactly the animated ones, and the $\infty$-category $\mathrm{DAlg}_R$ has all limits and colimits whose formation commutes with the forgetful functor to $E_\infty$-$R$-algebras. In particular, given a $p$-complete animated $\overline{A}$-algebra $S$, the $E_\infty$-$A$-algebra $\Prism_{S/A}$ is naturally a derived $A$-algebra.
\end{notation}

\begin{construction}[Relating classical and animated prismatic sites]
\label{AnimPrismaticSite}
Let $R$ be a $p$-complete animated $\overline{A}$-algebra. Consider the $\infty$-category $(R/A)_{\Prism}^{\mathrm{an}}$ of pairs $(B,v)$ where $B$ is a $(p,I)$-complete animated $\delta$-$A$-algebra and  $v:R \to \overline{B}$ is a map of animated $\overline{A}$-algebras; this $\infty$-category contains the relative prismatic site $(\pi_0(R)/A)_\Prism$ of $\pi_0(R)$ as the full subcategory spanned by $(B,v)$ with $B$ being discrete and $I$-torsionfree. Consequently, there is a comparison map
\[ c_R:\varprojlim_{(B,v) \in (R/A)_{\Prism}^{\mathrm{an}}} B \to \varprojlim_{(B,v) \in (\pi_0(R)/A)_\Prism} B\]
in $\mathrm{DAlg}_A$.
\end{construction}

\begin{proposition}
\label{PrismaticCohomologyAnimatedLimit}
The comparison map $c_R$ from Construction~\ref{AnimPrismaticSite} is an isomorphism if $R$ is $p$-completely smooth over $\overline{A}$. 
\end{proposition}

\begin{proof}
We first recall the \v{C}ech-Alexander approach to prismatic cohomology in \cite[Construction 4.17]{prisms}. Take a surjection $B_0 \to B_0/J = R$ with $B_0$ a polynomial ring over $A$, let $B'$ be the free $\delta$-$A$-algebra on $B_0$, and let $B$ to be the $(p,I)$-completed prismatic envelope of $JB' \subset B'$. Then $(B,IB)$ is weakly initial in $(R/A)_\Prism$ and  its \v{C}ech nerve $(C^*,IC^*)$ in $(R/A)_\Prism$ satisfies that the derived $A$-algebra $C^*$ computes the limit $\varprojlim_{(B,v) \in (R/A)_\Prism} B$. 

We wish to apply the same analysis to conclude that $C^*$ also computes $\varprojlim_{(B,v) \in (R/A)_{\Prism}^{\mathrm{an}}} B$. For this, we need to verify that $(B,IB)$ is weakly initial in $(R/A)_{\Prism}^{\anim}$, and that its \v{C}ech nerve $(C^*,IC^*)$ in $(R/A)_\Prism$ coincides with the animated \v{C}ech nerve in $ (R/A)_{\Prism}^{\mathrm{an}}$. The first holds true as free $\delta$-algebras on polynomial rings are also free in the animated sense, and because the prismatic envelope appearing in the formation of $B$ from $B'$ coincides with its animated version by the construction in \cite[Corollary 3.14]{prisms}. The second property follows again from the same observation concerning the formation of prismatic envelopes. 
\end{proof}

\begin{corollary}
\label{CorPrismaticCohomologyAnimatedLimit}
Given a $p$-complete animated $\overline{A}$-algebra $R$ and some $(C,w) \in (R/A)_{\Prism}^{\mathrm{an}}$ as in Construction~\ref{AnimPrismaticSite}, there is a natural comparison map $\Prism_{R/A} \to C$ of derived $A$-algebras, uniquely characterized by the requirement that it agrees for $p$-completely smooth $\overline{A}$-algebra $R$ with the map 
\[ \Prism_{R/A} \simeq  \varprojlim_{(B,v) \in (R/A)_\Prism} B \stackrel{c_R}{\simeq} \varprojlim_{(B,v) \in (R/A)_{\Prism}^{\mathrm{an}}} B  \xrightarrow{\text{project}} C,\] 
where the first isomorphism is the agreement of derived and site-theoretic prismatic cohomology for such $R$.
\end{corollary}

\begin{remark}
Remark~\ref{PrismatizeClassical} below extends Proposition~\ref{PrismaticCohomologyAnimatedLimit} extends to a larger class of $\overline{A}$-algebras $R$: we only need to assume that $R$ satisfies condition $(\ast)$ from Remark~\APCref{remark:hypothesis-for-discreteness}. On the other hand, using  Theorem~\ref{PrismaticCohPrismatization} (1) as well as the failure of derived prismatic cohomology to coincide with its site-theoretic version in general shows that the proposition fails, e.g., when $(A,I)= (\mathbf{Z}_p,(p))$ and $R=\mathbf{F}_p[x,y]/(x,y)^2$.
\end{remark}

\begin{definition}[The spectrum  of $\Prism_{R/A}$]
For any $p$-complete animated $\overline{A}$-algebra $R$, write 
\[ \mathrm{Spf}(\Prism_{R/A}) = \mathrm{Map}_{\mathrm{DAlg}_{A}}(\Prism_{R/A},-) \]
for the $\mathcal{S}$-valued functor corepresented by $\Prism_{R/A}$ on $(p,I)$-nilpotent animated $A$-algebras. 
\end{definition}

\begin{construction}[Relating $\mathrm{Spf}(\Prism_{R/A})$ to the prismatization]
\label{RelatePrismatizeAffine}
Fix a $p$-complete animated $\overline{A}$-algebra with formal spectrum $X=\mathrm{Spf}(R)$. Given a point of $\WCart_{X/A}$ corresponding to a $(p,I)$-nilpotent animated $A$-algebra $B$ and an $\overline{A}$-algebra map $v:R \to \overline{W(B)}$, we obtain an induced map 
\[\Prism(v):\Prism_{R/A} \to \Prism_{\overline{W(B)}/A}\]
 of derived $A$-algebras. On the other hand,  regarding $W(B)$ as an animated $(p,I)$-complete $\delta$-$A$-algebra equipped with the identity map $\overline{W(B)} \xrightarrow{id} \overline{W(B)}$, Corollary~\ref{CorPrismaticCohomologyAnimatedLimit} yields a natural map 
 \[\mathrm{can}_B:\Prism_{\overline{W(B)}/A} \to W(B)\] 
 of derived $A$-algebras. Postcomposing with the restriction map gives a natural (in $v$) composition
\[ \alpha(v):\Prism_{R/A} \xrightarrow{\Prism(v)} \Prism_{\overline{W(B)}/A} \xrightarrow{\mathrm{can}_B} W(B) \xrightarrow{\text{restrict}} B\]
of maps in $\mathrm{DAlg}_A$. The construction carrying $v$ to $\alpha(v)$ yields a natural transformation
\[ \alpha_R:\WCart_{\mathrm{Spf}(R)/A} \to \mathrm{Spf}(\Prism_{R/A})\]
of presheaves on  $(p,I)$-nilpotent animated $A$-algebras.
\end{construction}

\begin{theorem}[Affineness of the prismatization]
\label{PrismAffineStack}
The map $\alpha_R$ from Construction~\ref{RelatePrismatizeAffine} is an isomorphism for all $p$-completed animated $\overline{A}$-algebras $R$.
\end{theorem}

\begin{proof}
We first observe that the proof of Variant~\ref{RelativeRSPComp} applies directly in this setting to solve the problem when $R$ is a relative regular semiperfectoid as in Variant~\ref{RelativeRSPComp}. We now begin proving the theorem in general by reducing to this case. As both the source and target of $\alpha_R$ carry colimits in $R$ to limits,  we may reduce to checking the statement for $R=\overline{A}[x]^{\wedge}$ being the $p$-completed polynomial ring. Write $S =\overline{A}[x^{1/p^\infty}]^{\wedge}$ for the naive perfection of $R$. Note that the map $R \to S$ is a $p$-quasisyntomic cover; write $S^*$ for its \v{C}ech nerve. We then have a map of augmented simplicial objects:
\[ \xymatrix{ \WCart_{\mathrm{Spf}(S^*)/A} \ar[r]^{\alpha_{S^*}} \ar[d] & \mathrm{Spf}(\Prism_{S^*/A}) \ar[d] \\
\WCart_{\mathrm{Spf}(R)/A} \ar[r]^-{\alpha_R} & \mathrm{Spf}(\Prism_{R/A}). }\]
Now the augmented simplicial object given by each column is a \v{C}ech nerve: this follows for the left side as $\WCart_{-/A}$ commutes with limits, and for the right side as $\Prism_{-/A}$ commutes with colimits. Moreover, the augmented simplicial object on the left identifies the target with the colimit of the source after flat sheafification by Proposition~\ref{PrismatizeCover2}. If we can show the same for the augmented simplicial object on the right, we will be done as the map $\alpha_{S^*}$ is an isomorphism of simplicial objects by the first sentence of the proof. Thus, we are reduced to checking that the map
\[ \mathrm{Spf}(\Prism_{S/A}) \to \mathrm{Spf}(\Prism_{R/A}) \]
is surjective for the flat  topology. To show this statement, it is enough to show the following:

\begin{itemize}
\item[$(\ast)$] For any $(p,I)$-complete animated $A$-algebra $B$ equipped with a map $\Prism_{R/A} \to B$ of derived $A$-algebras, the base change $B' := B \widehat{\otimes}_{\Prism_{R/A}} \Prism_{S/A}$ is a connective derived $A$-algebra (and thus an animated $A$-algebra) which is a $(p,I)$-completely faithfully flat over $B$. 
\end{itemize}

Note that to prove $(\ast)$, we may ignore the derived algebra structures and simply work at the level of modules. In this case, it suffices to prove the following stronger assertion purely about the map $\Prism_{R/A} \to \Prism_{S/A}$: the $\overline{\Prism}_{R/A}/p$-module $\overline{\Prism}_{S/A}/p$ admits an increasing exhaustive $\mathbf{N}$-indexed filtration whose graded pieces are  free over $\overline{\Prism}_{R/A}/p$. For this, observe that the $\overline{A}$-algebra map $R \to S$ lifts to a $\delta$-map $\widetilde{R} := A[x]^{\wedge} \to \widetilde{S} := A[x^{1/p^\infty}]^{\wedge}$ of $(p,I)$-completely flat $\delta$-$A$-algebras (where $\delta(x)=0$). It follows that the Hodge-Tate gerbe $\WCart^{\mathrm{HT}}_{R/A} \to \mathrm{Spf}(R)$ for $\mathrm{Spf}(R)$ (Proposition~\ref{RelativeHT}) admits a trivialization (Remark~\ref{HTGerbeProperties} (1)) compatibly with the isomorphism $\WCart^{\mathrm{HT}}_{S/A} \to \mathrm{Spf}(S)$ (Example~\ref{DerivedPrismatizeEtale}), i.e., we obtain a commutative diagram
\[ \xymatrix{ \mathrm{Spf}(S) \ar[r] \ar[d]^-{s_{\widetilde{S}}} & \mathrm{Spf}(R) \ar[d]^-{s_{\widetilde{R}}}  \\
\WCart^{\mathrm{HT}}_{\mathrm{Spf}(S)/A} \ar[r] \ar[d]^{\pi^{\mathrm{HT}}_S} & \WCart^{\mathrm{HT}}_{\mathrm{Spf}(R)/A} \ar[d]^{\pi^{\mathrm{HT}}_R} \\
\mathrm{Spf}(S) \ar[r] & \mathrm{Spf}(R) }\]
where the maps labelled $\pi^{\mathrm{HT}}$ are the Hodge-Tate structure maps and are gerbes for $0$ and $T_{\mathrm{Spf}(R)/A}\{1\}$ respectively, the maps labelled $s_{\widetilde{S}}$ and $s_{\widetilde{R}}$ are are sections to $\pi_S$ and $\pi_R$ arising from the choices of the $\delta$-lifts $\widetilde{S}$ and $\widetilde{R}$ respectively, the left vertical maps are all isomorphisms, the horizontal maps are the structure maps induced by the $\overline{A}$-algebra map $R \to S$, and the commutativity of the top square in the diagram arises the existence of the $\delta$-$A$-algebra map $\widetilde{R} \to \widetilde{S}$ lifting $R \to S$. Using $s_{\widetilde{R}}$ to trivialize the gerbe $\pi^{\mathrm{HT}}_R$ gives an isomorphism
\[ \overline{\Prism}_{R/A} \simeq \bigoplus_i \Omega^i_{R/\overline{A}}\{-i\}[-i]\]
$E_\infty$-$R$-algebras (see second half of Remark~\ref{HTDerivedPrismatize}). By the commutative diagram above, this isomorphism  carries  the map $\overline{\Prism}_{R/A} \to \overline{\Prism}_{S/A} \simeq S$ to the composition
\[ \bigoplus \Omega^i_{R/\overline{A}}\{-i\}[-i] \to R \to S \]
where first map is the canonical augmentation and the second map is the structure map. As $S$ is clearly $p$-completely free over $R$, it suffices to show that the $\left(\bigoplus_i \Omega^i_{R/\overline{A}}\{-i\}[-i]\right)/p$-module $R/p$ admits an increasing exhaustive $\mathbf{N}$-indexed filtration whose graded pieces are  free; this is a standard calculation, see \cite[Proof of Lemma 8.6]{prisms}.
\end{proof}

\begin{corollary}[The prismatization of a semiperfectoid]
\label{PrismSemiPerf}
Say $R$ is a $p$-complete $\overline{A}$-algebra such that $\Omega^1_{(\pi_0(R)/p)/(\overline{A})} = 0$; for instance, this holds true if $\pi_0(R)/p$ is a semiperfect ring. Then $\Prism_{R/A}$ is an animated $A$-algebra and corepresents $\WCart_{\mathrm{Spf}(R)/A}$.
\end{corollary}
\begin{proof}
The assumption on $R$ ensures that the $p$-completion $\left(\wedge^i L_{R/(\overline{A})}\right)[-i]^{\wedge}$ is connective. The Hodge-Tate comparison then shows that $\Prism_{R/A}$ is a connective derived $A$-algebra, whence it is an animated $A$-algebra. The rest follows from Theorem~\ref{PrismAffineStack}.
\end{proof}

\subsection{Comparison with derived prismatic cohomology}

\begin{construction}[Relating $\Prism_{R/A}$ with {$\RGamma( \WCart_{\mathrm{Spf}(R)/A}, \mathcal{O})$}]
Let $X$  be a  $p$-adic formal derived $\overline{A}$-scheme. Following Construction~\APCref{construction:prismatic-complex-of-scheme}, define
\[ \RGamma_\Prism(X/A) = \varprojlim_{\mathrm{Spf}(R) \to X} \Prism_{R/A}, \]
where the inverse limit is indexed by the $\infty$-category of points of $X$; by Zariski descent for $\Prism_{-/A}$, we could also take the inverse limit over the poset of affine opens in $X$ without changing the result. (Note that when $X$ is classical, this object coincides with the one from Construction~\APCref{construction:prismatic-complex-of-scheme}, so there is no conflict of notation: this is clear when $X$ is affine, and the rest follows from the Zariski descent property of both constructions.)

When $X=\mathrm{Spf}(S)$ is affine, the construction $v \mapsto \alpha(v)$ in Construction~\ref{RelatePrismatizeAffine} yields on passage to limits a natural comparison map
\[ \beta_X: \Prism_{S/A} \to  \RGamma( \WCart_{X/A}, \mathcal{O}).\]
As both sides are sheaves in the Zariski topology, this globalizes to a comparison map 
\[ \beta_X:\RGamma_\Prism(X/A) \to \RGamma( \WCart_{X/A}, \mathcal{O})\]
for all $X$.
\end{construction}

\begin{theorem}[Relating prismatic cohomology to the prismatization]
\label{PrismaticCohPrismatization}
Let $X$ be a $p$-adic formal derived $\overline{A}$-scheme. Assume that one the following conditions holds true: 
\begin{enumerate}
\item (Finite type) For every affine open $\mathrm{Spf}(R) \subset X$, the $\pi_0(R)/p$-module $\Omega^1_{(\pi_0(R)/p)/(\overline{A})}$ is finitely generated. 
\item (Quasi-lci) The derived formal scheme $X$ is classical,  bounded, and $L_{X/\overline{A}} \in \mathcal{D}(X)$ has $p$-complete Tor amplitude  $\geq -1$ (i.e., for each affine open $\mathrm{Spf}(R) \subset X$, the $\overline{A}$-algebra $R$ satisfies condition $(\ast)$ from Remark~\APCref{remark:hypothesis-for-discreteness}).
\end{enumerate}
Then the  map
\[ \beta_X:\RGamma_\Prism(X/A) \to  \RGamma( \WCart_{X/A}, \mathcal{O})\]
is an isomorphism. 
\end{theorem}

Note that the first condition is satisfied as soon as the mod $p$ reduction of the classical truncation of $X$ has finite type over $A/(I,p)$.

\begin{proof}
We may assume $X=\mathrm{Spf}(R)$ is affine. Thus, we want to show that the map 
\[ \beta_R: \Prism_{R/A} \to  \RGamma( \WCart_{\mathrm{Spf}(R)/A}, \mathcal{O})\]
is an isomorphism. We do so via descent in both cases (1) and (2). 

Assume we are in case (1).  Fix $f_1,...,f_n \in \pi_0(R)/p$ whose elements generate the module of K\"{a}hler differentials $\Omega^1_{(\pi_0(R)/p)/(\overline{A})}$. We can then choose a map $S := \overline{A}[x_1,...,x_n]^{\wedge} \to R$ carrying $x_i$ to $f_i$ in $\pi_0(R)/p$. Write $S_\infty = \overline{A}[x_1^{1/p^\infty},...,x_n^{1/p^\infty}]^{\wedge}$ for the mock perfection of $S$,  write $R_\infty = S_\infty \widehat{\otimes}_S R$ for its base change to $R$, and let $S \to S_\infty^*$ and $R \to R_\infty^*$ be the \v{C}ech nerves of $S \to S_\infty$ and  $R \to R_\infty$ respectively. We then have the commutative diagram
\[ \xymatrix{ \Prism_{R/A} \ar[r] \ar[d] &  \RGamma( \WCart_{\mathrm{Spf}(R)/A)}, \mathcal{O}) \ar[d] \\
\Prism_{R_\infty/A}^* \ar[r] &  \RGamma( \WCart_{\mathrm{Spf}(R_\infty^*)/A}, \mathcal{O}) }\]
of augmented cosimplicial objects. Note that the map $R \to R_\infty$ is a $p$-quasisyntomic cover. Since prismatization commutes with limits and carries $p$-quasisyntomic covers to flat covers (Proposition~\ref{PrismatizeCover2}), the map on the right is a limit diagram, i.e., identifies the source with the limit of the target. As each $R_\infty^*$ is semiperfect modulo $p$ on $\pi_0$, Corollary~\ref{PrismSemiPerf} shows that the bottom arrow is an isomorphism of cosimplicial objects. It is therefore enough to show that the map on the left gives a limit diagram. Now, by base change for prismatic cohomology, the map on the left is obtained from
\[ \Prism_{S/A} \to \Prism_{S_\infty^*/A}\]
by applying $- \widehat{\otimes}_{\Prism_{S/A}} \Prism_{R/A}$. As $\Prism_{-/A}$ commutes with colimits, the map $\Prism_{S/A} \to \Prism_{S^*/A}$ identifies with the \v{C}ech nerve of $\Prism_{S/A} \to \Prism_{S_\infty/A}$. But this last map is descendable modulo $(p,I)$, as explained in the proof of Theorem~\ref{PrismAffineStack}. Consequently, the $(p,I)$-completed base change of its \v{C}ech nerve to any $\Prism_{S/A}$-algebra gives a limit diagram, so we win.

Assume we are in case (2). Let $R \to R_\infty$ be the $p$-quasisyntomic cover of $R$ obtained by formally extracting $p$-power roots of all elements of $R$; let $R \to R_\infty^*$ be its \v{C}ech nerve. The map $R \to R_\infty$ is a $p$-quasisyntomic cover, and each $R_\infty^*$ is semiperfect modulo $p$. Running the same argument as the one above, it remains to verify that $\Prism_{R/A} \simeq \varprojlim \Prism_{R_\infty^*/A}$. But this follows as functor $\Prism_{-/A}$ is a sheaf for the relevant map (Theorem~\APCref{theorem:site-theoretic-equivalence}).
\end{proof}

Next, we discuss a criterion for the $\mathrm{WCart}_{X/A}$ to be a classical stack; let us first introduce the necessary language for the latter notation.

\begin{notation}[Classical sheaves on derived formal affine $\overline{A}$-schemes]
\label{ClassicalRelative}
Write $\mathrm{DAff}_{\mathrm{Spf}(A)}$ for the opposite of the $\infty$-category of $(p,I)$-nilpotent animated $A$-algebras; let $\mathrm{Aff}_{\mathrm{Spf}(A)} \subset \mathrm{DAff}_{\mathrm{Spf}(A)}$ be the full subcategory spanned by the discrete rings. We give these two $\infty$-categories the flat topology\footnote{We use the flat topology for maximal flexibility. The $p$-quasisyntomic topology (or even the pro-syntomic topology) would suffice for our purposes.}. Let us call a sheaf $\mathcal{F}$ on $\mathrm{DAff}_{\mathrm{Spf}(A)}$ {\em classical} if it can be written as a colimit of objects lying in the essential image of $\mathrm{Aff}_{\mathrm{Spf}(A)}$ under the Yoneda embedding; this is equivalent to requiring that $\mathcal{F}$ lies in the essential image of the fully faithful\footnote{To see full faithfulness, one uses the following: any faithfully flat map $U \to V$ in $\mathrm{DAff}_{\mathrm{Spf}(A)}$ with $V \in \mathrm{Aff}_{\mathrm{Spf}(A)}$ comes from a faithfully flat map  in $\mathrm{Aff}_{\mathrm{Spf}(A)}$ (i.e., $U  \in \mathrm{Aff}_{\mathrm{Spf}(A)}$) by definition of flatness.} colimit preserving functor  $\mathrm{Shv}(\mathrm{Aff}_{\mathrm{Spf}(A)}) \to  \mathrm{Shv}(\mathrm{DAff}_{\mathrm{Spf}(A)})$ extending the Yoneda embedding.
\end{notation}

\begin{proposition}[Classicality of the prismatization]
\label{QSynWCartClassical}
 Let $X$ be a bounded $p$-adic formal $\overline{A}$-scheme, regarded as derived $\overline{A}$-scheme. Assume that $L_{X/\overline{A}} \in \mathcal{D}(X)$ has $p$-complete Tor amplitude $\geq -1$  (i.e., for each affine open $\mathrm{Spf}(R) \subset X$, the $\overline{A}$-algebra $R$ satisfies condition $(\ast)$ from Remark~\APCref{remark:hypothesis-for-discreteness}). Then $\WCart_{X/A}$ is classical. 
\end{proposition}
\begin{proof}
By compatibility with \'etale localization, we may assume $X=\mathrm{Spf}(R)$ is affine.  It suffices to show that $\WCart_{X/A}$ can be written as a colimit of a simplicial object in the full subcategory $\mathrm{Aff}_{\mathrm{Spf}(A)} \subset \mathrm{DAff}_{\mathrm{Spf}(A)}$. Take $R \to R_\infty^*$ as in the second paragraph of the proof of Theorem~\ref{PrismaticCohPrismatization}. By construction, each $\Prism_{R_\infty^*/A}$ is discrete, and thus $\WCart_{\mathrm{Spf}(R_\infty^*)/A} \simeq \mathrm{Spf}(\Prism_{R_\infty^*/A})$ is classical by Theorem~\ref{PrismAffineStack}. As the prismatization functor $\WCart_{-/A}$ commutes with limits and carries $p$-quasisyntomic surjections to flat covers (Proposition~\ref{PrismatizeCover2}), it follows that $\WCart_{\mathrm{Spf}(R)/A}$ is the the colimit of $\mathrm{Spf}(\Prism_{R_\infty^*/A})$ as flat sheaves on  $\mathrm{DAff}_{\mathrm{Spf}(A)}$, so we win.
\end{proof}

\begin{remark}
\label{PrismatizeClassical}
Using  Proposition~\ref{QSynWCartClassical} and Corollary~\ref{PrismSemiPerf} instead of Variant~\ref{RelativeRSPComp}, one checks that the conclusions of Theorems~\ref{RelativePrismaticCohPrismatizeNonDer} and \ref{CrystalsCartWitt} hold true for any $X/\overline{A}$ as in Proposition~\ref{QSynWCartClassical}. 
\end{remark}

\newpage
\section{Derived absolute prismatization}
\label{ss:DerAbsPrism}

In this section, we construct a Cartier--Witt stack $\WCart_X$ for any derived $p$-adic formal scheme $X$ using the notion of a Cartier--Witt divisor in the animated context (Definition~\ref{CartWittAnimDef}). The main results are that under $p$-quasisyntomicity assumptions, the theory of quasi-coherent sheaves on these stacks geometrizes the theory of absolute prismatic crystals (Proposition~\ref{QSynAbsCrys}), and that this geometrization is compatible with the notion of pushforwards in a reasonably wide variety of situations (Corollary~\ref{AbsPushWCart}). As all the results are proven by reducing to the relative case treated in \S \ref{ss:DerRelPrism}, our proofs will be brief.

\begin{notation}[Classical sheaves on $p$-adic formal affine schemes]
Write $\mathrm{DAff}$ for the opposite of the $\infty$-category of $p$-nilpotent animated rings; let $\mathrm{Aff} \subset \mathrm{DAff}$ be the full subcategory spanned by the discrete rings. We give these two $\infty$-categories the flat topology. As in Notation~\ref{ClassicalRelative}, let us call a sheaf $\mathcal{F}$ on $\mathrm{DAff}$ {\em classical} if it  can be written as a colimit of objects lying in the essential image of $\mathrm{Aff}$ under the Yoneda embedding $\gamma:\mathrm{DAff} \hookrightarrow \mathrm{Shv}(\mathrm{DAff})$ or equivalently if $\mathcal{F}$ lies in the essential image of the fully faithful functor $\tilde{\gamma}:\mathrm{Shv}(\mathrm{Aff}) \to  \mathrm{Shv}(\mathrm{DAff})$. For $\mathcal{F}$ classical, there is a unique object $\mathcal{G} \in \mathrm{Shv}(\mathrm{Aff})$ equipped with an identification $\tilde{\gamma}(G) \simeq \mathcal{F}$; in this case, we often abuse notation and say that $\mathcal{G}$ represents $\mathcal{F}$.
\end{notation}

Recall from Construction~\ref{GenCartDivAnim} that for an animated ring $A$, the notion of a generalized Cartier divisors $A \to A/I$ on $A$ is equivalent to the notion of generalized invertible ideals $I \to A$ in $A$. Using this equivalence, we arrive at the main object of study of this section:

\begin{definition}[Cartier--Witt divisors on animated rings]
\label{CartWittAnimDef}
For a $p$-nilpotent animated ring $R$, a {\em Cartier--Witt divisor on $R$} is given by a generalized Cartier divisor $W(R) \to W(R)/I$ whose corresponding generalized invertible ideal $I \to W(R)$ induces a Cartier--Witt divisor on $\pi_0(R)$ (as in \APCref{definition:generalized-Cartier-divisor}) after base change along $W(R) \to W(\pi_0(R)) \simeq \pi_0(W(R))$. Thus, the space of Cartier-Witt divisors on $R$ is discreted as the fibre product
\[ [\mathbf{A}^1/\mathbf{G}_m](W(R)) \times_{[\mathbf{A}^1/\mathbf{G}_m](W(\pi_0(R)))} \mathrm{WCart}(\pi_0(R)). \]
\end{definition}

\begin{remark}[Cartier--Witt divisors as animated prisms]
Given a $p$-nilpotent animated ring $R$, one can also define a Cartier-Witt divisor on $R$ to be an animated prism $W(R) \to W(R)/I$ (as in Definition~\ref{DefAnimPrism}, for the usual animated $\delta$-structure on $W(R)$) such that the induced map $\pi_0(I) \to \pi_0(W(R)) \xrightarrow{\text{restriction}} R$ has nilpotent image; this follows from Remark~\ref{CartWitttoAnim}.
\end{remark}

\begin{proposition}[The Cartier--Witt stack as a derived stack]
\label{DerivedWCart}
The presheaf $\mathcal{F}$ carrying a $p$-nilpotent animated ring $R$ to the space of Cartier--Witt divisors on $R$ is a sheaf, is classical, and represented by $\WCart$.
\end{proposition}
\begin{proof}
For any $p$-nilpotent animated ring $R$, we have a natural identification of  spaces $\mathrm{Pic}(W(R)) \simeq \lim_n \mathrm{Pic}(W_n(R))$. Using this identification and deformation theory, it follows from flat descent for line bundles that the functor $\mathrm{Pic}(W(-))$ on $p$-nilpotent animated rings is a sheaf for the flat topology, is classical, and represented by $BW^*$. A similar argument with maps shows that the functor carrying a $p$-nilpotent animated ring $R$ to the space of generalized Cartier divisors on $W(R)$ is a sheaf and in fact represented by $[W/W^*](-) \simeq [\mathbf{A}^1/\mathbf{G}_m](W(-))$. The sheafyness of the functor in the proposition then follows from the corresponding statement for $p$-nilpotent discrete rings. 

For the rest, we shall check that $\mathcal{F}$ is represented by $\mathrm{WCart}$ using the presentation $\mathrm{WCart} = [\WCart_0 / W^*]$ from Proposition~\APCref{proposition:presentation-as-quotient}. Using this presentation and unwinding definitions,  it is enough to check that the functor
\[ R \mapsto \mathbf{A}^1(W(R)) \times_{\mathbf{A}^1(W(\pi_0(R)))} \mathrm{WCart}_0(\pi_0(R)) \simeq W(R) \times_{W(\pi_0(R))} \mathrm{WCart}_0(\pi_0(R)).\]
on $p$-nilpotent animated rings is represented by $\mathrm{WCart}_0$. This follows from the fact that the natural map $\mathrm{WCart}_0 \to W$ also  exhibits the former as a formal completion of the latter in derived algebraic geometry. 
\end{proof}

In the rest of this section, we abuse notation and identify the quotient stack $\WCart$ with the functor $\mathcal{F}$ on $\mathrm{DAff}$ from Proposition~\ref{DerivedWCart}. 

\begin{remark}
\label{rmk:WCartColimit}
Proposition~\ref{DerivedWCart} implies our previous results \cite{BhattLurieAPC} on $\WCart$ translate over to the derived setting. For instance, for any bounded prism $(A,I)$, we have a natural classifying map $\rho_A:\mathrm{Spf}(A) \to \WCart$ of sheaves on $\mathrm{DAff}$ by Construction~\APCref{construction:point-of-prismatic-stack}. Moreover, if  $(A^\bullet, I^\bullet)$ is the cosimplicial prism from Notation~\APCref{notation:simplicial-prism} obtained by taking global sections of the standard simplicial presentation of the quotient stack $[\WCart_0/W^*]$, then the maps $\rho_{A^\bullet}$ yield an equivalence
\[ \colim \mathrm{Spf}(A^\bullet) \simeq \WCart\]
in $\mathrm{Shv}(\mathrm{DAff})$.
\end{remark}

\begin{definition}[The Cartier--Witt stack of a derived scheme]
\label{DerCartWittAbs}
Let $X$ be a derived $p$-adic formal scheme. Its Cartier--Witt stack $\WCart_X$ is the presheaf on $p$-nilpotent animated rings as follows: $\WCart_X(R)$ is the $\infty$-groupoid of pairs $(I \xrightarrow{\alpha} W(R), \eta:\mathrm{Spec}(\overline{W(R)}) \to X)$, where $(I \xrightarrow{\alpha} W(R)) \in \WCart(R)$ is a Cartier-Witt divisor and $\eta$ is a morphism of derived formal schemes.
\end{definition}

\begin{warning}
\label{Warn:DerNotClassAbs}
Say $X$ is a bounded $p$-adic formal scheme; write $Y$ for the result of viewing $X$ as a derived $p$-adic formal scheme. Then the Cartier--Witt stack $\WCart_Y$ need not be classical (and thus does not agree with the image of $\WCart_X$ under the fully faithful embedding $\mathrm{Shv}(\mathrm{Aff}) \to  \mathrm{Shv}(\mathrm{DAff})$): this follows from the example in Warning~\ref{Warn:DerNotClass} observing that $\WCart_Y \simeq \WCart_{Y/A}$ for a derived qcqs $p$-adic formal scheme $X$ equipped with a map to $\mathrm{Spf}(\overline{A})$ for a perfect prism $(A,I)$, and similarly in the non-derived case. 
\end{warning}

\begin{example}
For the final object $X=\mathrm{Spf}(\mathbf{Z}_p)$ in $\mathrm{DAff}$, the stack $\WCart_X$ is represented by $\WCart$. By functoriality, there is an induced structure map $\WCart_Y \to \WCart$ for any qcqs derived $p$-adic formal scheme $Y$. 
\end{example}

\begin{remark}
The functor $X \mapsto \WCart_X$ on derived $p$-adic formal schemes naturally takes values in presheaves over $\WCart$, and commutes with all limits when viewed as such: this compatibility is immediate from the definition.
\end{remark}

\begin{remark}
Fix a derived $p$-adic formal scheme $X$. Then $\WCart_X$ is a sheaf for the \'etale topology as in Construction~\ref{RelPrismDerived}. Moreover, if $X$ is qcqs, then the argument in Lemma~\ref{SheafFlatWCart} applies mutatis mutandis to show that $\WCart_X$ is a sheaf for the flat topology.
\end{remark}

\begin{remark}
\label{DescentAbsWCart}
Fix a derived $p$-adic formal scheme $X$. If $(A,I)$ is a bounded prism, then we have a base change isomorphism
\[ \WCart_X \times_{\WCart,\rho_A} \mathrm{Spf}(A) \simeq \WCart_{X_{\overline{A}}/A}\]
of sheaves on $(p,I)$-nilpotent animated $A$-algebras, where the target is the derived relative prismatization studied in \S \ref{ss:DerRelPrism}. Applying this observation to the cosimplicial bounded prism  $(A^\bullet, I^\bullet)$ from Notation~\APCref{notation:simplicial-prism}  and using the observation in Remark~\ref{rmk:WCartColimit}, we learn that the maps $\rho_{A^\bullet}$ induce an equivalence
\[ \colim \WCart_{X_{\overline{A}^\bullet}/A} \simeq \WCart_X,\]
and hence that pullback along $\rho_{A^\bullet}$ induces an equivalence
\[ \mathcal{D}_{qc}(\WCart_X) \simeq \lim \mathcal{D}_{qc}(\WCart_{X_{\overline{A}^\bullet}/A})\]
of symmetric monoidal stable $\infty$-categories. 
\end{remark}

In the rest of this section, we focus on relating the theory of quasi-coherent sheaves on the Cartier--Witt stacks defined above with the absolute prismatic site. This comparison works best for the following class of schemes:

\begin{definition}
A derived $p$-adic formal scheme $X$ is called {\em $p$-quasisyntomic} if for every affine open $\mathrm{Spf}(R) \subset X$, the animated ring $R$ is classical and $p$-quasisyntomic. In particular, $X$ is a classical bounded $p$-adic formal scheme.
\end{definition}

\begin{corollary}
\label{AbsWCartClassical}
Let $X$ be a  derived $p$-adic formal scheme which is $p$-quasisyntomic. Then $\WCart_X$ is classical.
\end{corollary}
\begin{proof}
Using the equivalence $\colim \WCart_{X_{\overline{A}^\bullet}/A} \simeq \WCart_X$ from Remark~\ref{DescentAbsWCart}, it suffices to show that $\WCart_{X_{\overline{A}}/A}$ is classical for every transversal prism $(A,I)$. The hypothesis on $X$ and the transversality of $(A,I)$ ensure that the structure map $X_{\overline{A}} \to \mathrm{Spf}(\overline{A})$ satisfies the hypothesis of Proposition~\ref{QSynWCartClassical},  so the claim follows from the conclusion of that proposition. 
\end{proof}

\begin{construction}[From the absolute prismatic site to the derived Cartier--Witt stack]
\label{cons:AbsPrismDerived}
Fix a  derived $p$-adic formal scheme $X$. The absolute prismatic site $X_\Prism$ is defined to be the (ordinary) category of bounded prisms $(A,I)$ equipped with a map $\eta:\mathrm{Spf}(A/I) \to X$ of derived $p$-adic formal schemes; when $X$ is a classical bounded $p$-adic formal scheme, this definition agrees with the classical notion of the absolute prismatic site of $X$. For any object $((A,I),\eta) \in X_\Prism$, one has an induced map $\rho_{X,A}:\mathrm{Spf}(A) \to \WCart_X$ by a variant of Construction~\ref{cons:pointprismatization}.
\end{construction}

\begin{proposition}
\label{QSynAbsCrys}
Let $X$ be a  derived $p$-adic formal scheme which is $p$-quasisyntomic. Then pullback along the maps from Construction~\ref{cons:AbsPrismDerived} yields an equivalence
\[ \mathcal{D}_{qc}(\WCart_X) \simeq \lim_{(A,I) \in X_\Prism} \widehat{\mathcal{D}}(A) =: \widehat{\mathcal{D}}_{crys}(X_\Prism, \mathcal{O}_\Prism)\]
of symmetric monoidal stable $\infty$-categories.
\end{proposition}
\begin{proof}
The relative version of this was the subject of Theorem~\ref{CrystalsCartWitt} and Remark~\ref{PrismatizeClassical}. The general case can be reduced by descent to the relative case using Remark~\ref{DescentAbsWCart} on the $\WCart$ side, and flat descent for prismatic crystals on the site-theoretic side.
\end{proof}

\begin{proposition}
\label{AbsPushWCart}
Let $Y$ be a  derived $p$-adic formal scheme which is $p$-quasisyntomic. Let $f:X \to Y$ be a map of qcqs derived $p$-adic formal schemes that satisfies one of the following conditions:
\begin{enumerate}
\item Finite type: for every affine open $\mathrm{Spf}(R) \subset Y$ and $\mathrm{Spf}(S) \subset f^{-1}(\mathrm{Spf}(R)) \subset X$, the $\pi_0(R)/p$-module $\Omega^1_{\pi_0(S)/\pi_0(R)}/p$ is finitely generated.
\item $p$-quasisyntomic:  for every affine open $\mathrm{Spf}(R) \subset Y$ and $\mathrm{Spf}(S) \subset f^{-1}(\mathrm{Spf}(R)) \subset X$, the map $R \to S$ is $p$-quasisyntomic after base change to any discrete $p$-nilpotent $R$-algebra. 
\end{enumerate}
Then the right adjoint $R\WCart_{f,*}$ to the pullback $f^*:\mathcal{D}_{qc}(\WCart_Y) \to \mathcal{D}_{qc}(\WCart_X)$ carries $\mathcal{O}_{\WCart_X}$ to an object of $\mathcal{D}_{qc}(\WCart_Y)$ that identifies naturally with the prismatic crystal $Rf_{\Prism,*} \mathcal{O}_\Prism \in \mathcal{D}_{crys}(Y_\Prism, \mathcal{O}_\Prism)$ under the equivalence in Proposition~\ref{QSynAbsCrys}. 
\end{proposition}
\begin{proof}
This follows via the descent equivalence in Proposition~\ref{QSynAbsCrys} from  the corresponding result in the relative case (Theorem~\ref{PrismaticCohPrismatization}). 
\end{proof}

\begin{corollary}
\label{AbsPushWCartZ}
Let $X$ be a qcqs derived $p$-adic formal scheme. Assume that the structure map $f:X \to \mathrm{Spf}(\mathbf{Z}_p)$ satisfies one of the conditions in Proposition~\ref{AbsPushWCart}. Then $R\WCart_{f,*} \mathcal{O}_{\WCart_X} \in \mathcal{D}_{qc}(\WCart)$ is identified with $\mathcal{H}_\Prism(X)$ (Variant~\APCref{variant:prismatic-sheaf-globalized}). 
\end{corollary}
\begin{proof}
Apply Proposition~\ref{AbsPushWCart} with $Y=\mathbf{Z}_p$ and use the definition in Variant~\APCref{variant:prismatic-sheaf-globalized}. 
\end{proof}

\newpage
\section{Examples: The Hodge-Tate stack of some regular schemes}
\label{ss:HTReg}

One philoshophical perspective on prismatic cohomology is that it yields a de Rham style cohomology theory ``over $\mathbf{F}_1$''. While nonsensical, this sometimes can lead to useful predictions, e.g., one expects that regular rings of dimension $d$ are formally smooth of dimension $d$ ``over $\mathbf{F}_1$'', so their absolute prismatic cohomology should be well-behaved. In this subsection, we provide some evidence for this heuristic. The main computation is Example~\ref{ExHTOK}, showing that the Hodge-Tate stack of a possibly ramified extension of $\mathbf{Z}_p$ has a simple description. 

\begin{example}[The Hodge-Tate stack of a smooth $\mathbf{Z}_p$-scheme]
\label{ex:AbsHTSmooth}
Let $k$ be a perfect field of characteristic $p$. Let $X$ be a smooth $p$-adic formal scheme over $W(k)$. We shall describe the derived Hodge-Tate stack $\WCart_X^{\mathrm{HT}}$ explicitly. Using Corollary~\ref{AbsWCartClassical}, it in fact suffices to work with classical stacks. The diffracted Hodge stack $X^{\DHod}$ (Construction~\ref{DiffHodgeStack}) can be identified with the functor on $p$-nilpotent $W(k)$-algebras $R$ given by
\[ X^{\DHod}(S) = X(W(S)/V(1)).\]
There is a natural map $W(S)/V(1) \to S$ given by restriction on the Witt vectors. This map realies $W(S)/V(1)$ as a square-zero extension of $S$ by $(B\mathbf{G}_a^\sharp)(S)$; this can be checked directly, and also follows from Construction~\ref{sqzeroHTprism} applied to the prism $(W(k)\llbracket \tilde{p} \rrbracket, (\tilde{p}))$ by Remark~\ref{DiffHodHT}.  Consequently, we learn as in Proposition~\ref{RelativeHT} that $X^{\DHod} \to X$ is a $T_{X/W(k)}^{\sharp}$-gerbe. In this description, the $\mathbf{G}_m^\sharp$ on $X^{\DHod}$ is given by evident the $\mathbf{G}_m^\sharp$-action on $W(S)/V(1)$.  Passing to quotients and using Construction~\ref{ConsAbsHT}, we learn that $\WCart_X^{\mathrm{HT}} \to X$ is a gerbe for the group scheme $T_{X/W(k)}^{\sharp} \rtimes \mathbf{G}_m^\sharp$, where $\mathbf{G}_m^\sharp$ acts on $T_{X/W(k)}^{\sharp}$ via its natural linear action. If $X$ is additionally endowed with a $\delta$-structure, then this gerbe splits via Construction~\ref{DeltaSchemePrismatize}. Consequently, for a smooth $p$-adic formal $W(k)$-scheme $X$ equipped with a $\delta$-structure, we learn that
\begin{equation}
\label{AbsHTSmoothFormula}
\WCart_X^{\mathrm{HT}} \simeq B(T_{X/W(k)}^\sharp \rtimes \mathbf{G}_m^\sharp).
\end{equation}
In particular, we remark that the group scheme appearing here is not commutative unless $X$ has relative dimension $0$ over $W(k)$. 
\end{example}

\begin{remark}
In the context of Example~\ref{ex:AbsHTSmooth}, one can use the isomorphism \eqref{AbsHTSmoothFormula} in conjunction with Proposition~\ref{QSynAbsCrys} and linear algebra to obtain an explicit algebraic description of Hodge-Tate crystals of quasicoherent complexes on $X$. For instance, a vector bundle on $B(T_{X/W(k)}^\sharp \rtimes \mathbf{G}_m^\sharp)$ is given by a triple $(E,\psi, \Theta)$ where
\begin{itemize}
\item $E$ is a vector bundle on $X$.
\item $\Theta:E \to E$ is an $\mathcal{O}_X$-linear Sen operator (i.e., a map such that $\Theta^p-\Theta$ is nilpotent on $E/p$, c.f. Theorem~\APCref{theorem:compute-with-HT}).
\item $\psi:E \to E \otimes_{\mathcal{O}_X} \Omega^1_{X/W(k)}\{-1\}$ is a $\Theta$-linear Higgs field which is nilpotent modulo $p$, with the twist $\{-1\}$ indicating that $\Theta$ acts by $-1$ on $\Omega^1_{X/W(k)}\{-1\}$. 
\end{itemize}
This description was recently also discovered in \cite{MinWeng1, MinWeng2} under the name of ``enhanced Higgs bundles'' in the more general case where $W(k)[1/p]$ is replaced by a possibly ramified discretely valued extension $K/\mathbf{Q}_p$ with perfect residue field; we expect that this stronger statement can also be proven using the geometric approach used above by replacing $\mathbf{G}_m^\sharp$ with the group scheme $G$ (which equals $\mathbf{G}_a^\sharp$ when $K$ is ramified) arising from the calculation in Example~\ref{ExHTOK}, but we do not pursue this idea further here. 
\end{remark}

For the rest of this section, we fix a complete noetherian regular local ring $R$ with perfect residue field $k$. Let $X = \mathrm{Spf}(R)$. Our goal in this section is to describe the stack $\WCart_X^{\mathrm{HT}} \to X$ explicitly.  To formulate the result, we need the following choice.

\begin{notation}
Choose a prism $(A,I)$ and an isomorphism $\overline{A} \simeq R$ with $R$ as chosen above. Such a choice always exists by the Cohen structure theorem (see \cite[Remark 3.11]{prisms}). Note any such $A$ is necessarily formally smooth over $W := W(k)$. Indeed, $A$ is a $p$-complete regular local ring with residue field $k$ since $\overline{A}$ is so. To conclude formal smoothness, it suffices to show that $p \notin \mathfrak{m}^2_A$, where $\mathfrak{m}_A \subset A$ is the maximal ideal, which is clear as $\delta(p) \in A^*$ while $\delta(\mathfrak{m}_A^2) \subset \mathfrak{m}_A$.
\end{notation}

Using this choice, we shall describe $\WCart_X^{\mathrm{HT}}$ explicitly as a quotient of the following:

\begin{construction}[Parametrizing all Hodge-Tate points]
\label{HTParam}
Let $\mathcal{F}$ be the presheaf on $R$-algebras given by the following: for any $R$-algebra $S$, let $\mathcal{F}(S)$ be the set of all $A$-module maps $\tau:I \to W(S)$ rendering the following diagram commutative
\begin{equation}
\label{eq:HTParam}
 \xymatrix{ I \ar[r]^-{\tau} \ar[d] & I \otimes_A W(S) \ar[d] \\ A \ar[r] & W(S).}
 \end{equation}
Here all maps except $\tau$ are the fixed standard ones. Moreover, there is a preferred point $\tau_0 \in \mathcal{F}(S)$ corresponding to the map $I \to I \otimes_A W(S)$ obtained by tensoring the standard map $A \to W(S)$ with $I$. The presheaf $\mathcal{F}$ has the following structures:
\begin{enumerate}
\item There is a natural action of $\mathbf{G}_m^\sharp = W^*[F]$ on $\mathcal{F}$ by acting via scalar multiplication  on the Cartier--Witt divisor $I \otimes_W W(S) \to W(S)$ appearing in the right column of the square above.

\item There is a natural action on $\mathrm{Hom}_A(I, I \otimes_A W[F]) \simeq \mathrm{Hom}_R(I/I^2, \mathbf{G}_a^\sharp\{1\}) \simeq \mathbf{G}_a^\sharp$ on $\mathcal{F}$: given an $A$-module map $a:I \to I \otimes_A W[F](S)$ and a point $\tau:I \to I \otimes_A W(S)$ in $\mathcal{F}(S)$, the sum $\tau + a:I \to I \otimes_A W(S)$ is the point $a \cdot \tau$ of $\mathcal{F}(S)$. Moreover, acting through this action on the point $\tau_0$ can be easily seen to give an isomorphism $\mathbf{G}_a^\sharp \simeq \mathcal{F}$ of presheaves. Under this isomorphism, the action in (1) is given as follows:
\[ (t \in \mathbf{G}_m^\sharp, a \in \mathbf{G}_a^\sharp) \mapsto t(a+1)-1 \in \mathbf{G}_a^\sharp.\]
In particular, it follows that the action in (1) turns $\mathcal{F}$ into a $\mathbf{G}_m^\sharp$-torsor.

\item There is a natural action of $\mathrm{Der}(A, I \otimes_A W[F]) \simeq \mathrm{Hom}_R(\Omega^1_A \otimes_A R, \mathbf{G}_a^\sharp\{1\}) \simeq T_A^\sharp\{1\} \otimes_A R$ on $\mathcal{F}$ determined by the action in (2) via composition with the map
\[  \mathrm{Hom}_R(\Omega^1_A \otimes_A R, \mathbf{G}_a^\sharp\{1\}) \to \mathrm{Hom}_R(I/I^2, \mathbf{G}_a^\sharp\{1\}) \]
obtained from $I/I^2 \xrightarrow{d} \Omega^1_A \otimes_A R$. Explicitly, given a derivation $D:A \to I \otimes_A W[F](S)$ and a point $\tau:I \to I \otimes_A W(S)$ of $\mathcal{F}(S)$, the sum $\tau + (D|_I):I \to I \otimes_A W(S)$ is the point $D \cdot \tau$ of $\mathcal{F}(S)$.

\end{enumerate}

One then checks that the actions in (1) and (3) assemble to an action of $(T_A^\sharp\{1\} \otimes_A R) \rtimes \mathbf{G}_m^\sharp$ on $\mathcal{F}$, where the semidirect product is formed with respect to the standard multiplication action of $\mathbf{G}_m^\sharp$ on $(T_A^\sharp\{1\} \otimes_A R)$.

\end{construction}

\begin{proposition}
\label{HTStackRegularLocal}
There is a natural isomorphism $\WCart_X^{\mathrm{HT}} \simeq \mathcal{F}/( (T_A^\sharp\{1\} \otimes_A R) \rtimes \mathbf{G}_m^\sharp)$ of stacks over $X$. Moreover, we have $\WCart_X^{\mathrm{HT}} \simeq BG$, where
\[ G = \mathrm{Stab}_{ (T_A^\sharp\{1\} \otimes_A R) \rtimes \mathbf{G}_m^\sharp}(\tau_0 \in \mathcal{F}).\]
\end{proposition}

In the proof below, we use the explicit model of $1$-truncated animated rings given by \cite{drinfeld-ring-groupoid}. Specifically, we use the notion of ``quasi-ideals'' from \cite[\S 3]{drinfeld-ring-groupoid}

\begin{proof}
The commutative square in \eqref{eq:HTParam} appearing in the definition of a point $\tau \in \mathcal{F}(S)$ can be regarded as a map of quasi-ideals, and thus induces a map $f_\tau:R \simeq \overline{A} \to \overline{W(S)}$ of animated rings. The assignment carrying $\tau$ to $( (I \otimes_A W(S) \xrightarrow{std} W(S)), f_\tau)$ for varying $S$ and $\tau$ then yields a map $\mathcal{F} \to \WCart_X^{\mathrm{HT}}$. Moreover, one checks that this map is naturally equivariant for the action of $(T_A^\sharp\{1\} \otimes_A R) \rtimes \mathbf{G}_m^\sharp$ on $\mathcal{F}$; the key point is that a derivation $D:A \to I \otimes_A W[F](S)$ can be regarded as a homotopy between $f_{\tau}$ and $f_{D \cdot \tau}$. Thus, we obtain an induced map 
\[  \mathcal{F}/( (T_A^\sharp\{1\} \otimes_A R) \rtimes \mathbf{G}_m^\sharp) \to \WCart_X^{\mathrm{HT}} \]
of stacks over $X$. This map fits into the following commutative diagram of stacks over $X$:
\[ \xymatrix{ 
 \mathcal{F}/ (T_A^\sharp\{1\} \otimes_A R) \ar[r] \ar[d] & X^{\DHod,A} \ar[r] \ar[d] & X \ar[d]^-{(\mathrm{id},\overline{\rho_{A}})} \\
 \mathcal{F}/( (T_A^\sharp\{1\} \otimes_A R) \rtimes \mathbf{G}_m^\sharp) \ar[r]  & \WCart_X^{\mathrm{HT}} \ar[r]^-{\mathrm{can}}  & X \times \WCart^{\mathrm{HT}} = X \times B\mathbf{G}_m^\sharp;}\]
Here all squares are Cartesian, which defines the $A$-twisted version  $X^{\DHod,A}$ of the diffracted Hodge stack appearing in the top row. Explicitly, for any $p$-nilpotent $R$-algebra $S$, one has a square-zero extension $W(S)/I_{W(S)} \to S$ of $S$ by $B\mathbf{G}_a^\sharp\{1\}(S)$ determined by the Hodge-Tate Cartier--Witt divisor $I_{W(S)} \to W(S)$ induced by $I \to A$ via base change along the $\delta$-map $A \to W(S)$ adjoint to the composition $A \to A/I \simeq R \to S$; the fibre of $X^{\DHod,A}(S) \to X(S)$ over the structure map $R \to S$ (regarded as a point of $X(S)$) is given by space of lifts $R \to W(S)/I_{W(S)}$ of $R \to S$. 
Using deformation theory as in Remark~\ref{rmk:comparesqzero} as well as the explicit description of $\mathcal{F}$, we see that first horizontal map $\mathcal{F}/(T_A^\sharp\{1\} \otimes_A R) \to X^{\DHod,A}$ in the top row is an isomorphism. By descent, the same holds true for the first horizontal map in the bottom row, giving the first statement in the proposition. The second statement follows readily since we already know that $(T_A^\sharp\{1\} \otimes_A R) \rtimes \mathbf{G}_m^\sharp$ (or even just the subgroup $1 \rtimes \mathbf{G}_m^\sharp$) acts transitively on $\mathcal{F}$.
\end{proof}

\begin{example}[The Hodge-Tate stack of $\mathrm{Spf}(\mathcal{O}_K)$]
\label{ExHTOK}
Assume further $R$ is a $p$-torsionfree dvr, so $R=\mathcal{O}_K$ for a finite extension $K/W(k)[1/p]$. We shall describe the objects appearing in Proposition~\ref{HTStackRegularLocal} explicitly to conclude that $\WCart_X^{\mathrm{HT}}$ identifies with $B\mathbf{G}_m^\sharp$ (resp. $B \mathbf{G}_a^\sharp$) as a stack over $X$ if $K$ is unramified (resp. not unramified). Recalling that the Cartier dual of $\mathbf{G}_a^\sharp$ is the formal completion $\widehat{\mathbf{G}_a}$ of $\mathbf{G}_a$ at $0$, we obtain the following ramified version of Theorem~\APCref{theorem:compute-with-HT}: for $K$ ramified, the $\infty$-category $\widehat{\mathcal{D}}_{crys}(X_\Prism, \overline{\mathcal{O}}_\Prism)$ of Hodge-Tate crystals of quasicoherent complexes on $\mathcal{O}_K$ identifies with the full subcategory of $\mathcal{D}(\mathcal{O}_K[\Theta])$ spanned by those complexes $K$ which are $p$-complete and such that $\Theta$ acts locally nilpotently on $H^*(K/p)$.

Choosing a uniformizer $\pi \in R$, we obtain a Breuil-Kisin prism $(A,I) = (W(k)\llbracket u \rrbracket , E(u))$ over $R$. We already have $\mathcal{F} = \mathbf{G}_a^\sharp$ as explained in Construction~\ref{HTParam}. Moreover, we can identify $T_A^\sharp\{1\} \simeq \mathbf{G}_a^\sharp$ by trivializing $\Omega^1_A \simeq A \cdot du$ and $I/I^2 \simeq \overline{A} \cdot E(u)$; the action of $D \in \mathbf{G}_a^\sharp \simeq T_A^\sharp\{1\}$ on $x \in \mathcal{F} \simeq \mathbf{G}_a^\sharp$ is then given by $D \cdot x = x + E' D$, where $E' = \frac{dE}{du} \in A$. Using the explicit formula for the $\mathbf{G}_m^\sharp$-action discussed in Construction~\ref{HTParam}, it follows that
\[ \WCart_{\mathrm{Spf}(R)}^{\mathrm{HT}} = BG\]
where
\[ G = \{ (a,t) \in \mathbf{G}_a^\sharp \rtimes \mathbf{G}_m^\sharp \mid t-1 = E'a\}.\]
We now have two cases:
\begin{itemize}
\item If $K$ is absolutely unramified (i.e., $E'$ is a unit), the projection to the second component identifies $G$ with $\mathbf{G}_m^\sharp$, so $\WCart_X^{\mathrm{HT}}=X \times B\mathbf{G}_m^\sharp$. 

\item Via projection to the first component, one sees that $G$ identifies as a scheme with $\mathbf{G}_a^\sharp$, with group structure $\ast$ given by
\[ a \ast b = a + b  + E' ab.\]
If $K$ is ramified, then $E'$ is not a unit, so $(E')^n \to 0$ $p$-adically; one then checks that the maps
\[ \mathbf{G}_a^\sharp \to G \quad \text{given by} \quad  x \mapsto \frac{e^{E'x}-1}{E'} = \sum_{n \geq 1} (E')^{n-1} \frac{x^n}{n!} \]
and
\[ G \to \mathbf{G}_a^\sharp \quad \text{given by} \quad a \mapsto \frac{\log(E'a+1)}{E'} = \sum_{n \geq 1} (-E')^{n-1} \frac{x^n}{n}\]
give mutually inverse isomorphisms of group schemes, whence $\WCart_X^{\mathrm{HT}}=X \times B\mathbf{G}_a^\sharp$.
\end{itemize}
\end{example}

\begin{remark}
Take $R=\mathcal{O}_K$ as in Example~\ref{ExHTOK}; assume that $K$ is not unramified. We have seen that $\WCart_X^{\mathrm{HT}} = BG$ for a group scheme $G$ that is isomorphic to $\mathbf{G}_a^\sharp$ and comes equipped with a projection $G \to \mathbf{G}_m^\sharp$. Unwrapping the construction, the resulting map $\WCart_X^{\mathrm{HT}} = BG \to B\mathbf{G}_m^\sharp \times X$ can be identified with the canonical map $\WCart_X^{\mathrm{HT}} \to \WCart^{\mathrm{HT}} \times X$, provided we identify $\WCart^{\mathrm{HT}} \times X$ with $B\mathbf{G}_m^\sharp \times X$ using the map $X \to \WCart^{\mathrm{HT}}$ arising from (Hodge-Tate) Cartier-Witt divisor $\left(W(-) \xrightarrow{E} W(-)\right)$ on $\mathcal{O}_K$-algebras\footnote{Beware that this might not agree with the standard identification $\WCart^{\mathrm{HT}} \times X = B\mathbf{G}_m^\sharp \times X$ coming from $\left(W(-) \xrightarrow{V(1)} W(-)\right)$ if the image of $E \in W(\mathcal{O}_K)$ does not have the form $u V(1)  = VFu$ for a unit $u \in W(\mathcal{O}_K)$. This problem disappears if we base change to $\mathcal{O}_C$ as we can then replace $E$ with $\tilde{p}$.}. In explicit terms, the map $\WCart_X^{\mathrm{HT}} \to \WCart^{\mathrm{HT}} \times X$ is now identified as the map on classifying stacks induced by the homomorphism 
\[ \mathbf{G}_a^\sharp \to \mathbf{G}_m^\sharp \quad \text{given by} \quad x \mapsto 1 + E' \cdot \frac{e^{E'x}-1}{E'} = e^{E'x}.\]
Using this observation, one can show that $\overline{\Prism}_{\mathcal{O}_K}\{n\} \simeq \mathcal{O}_K/nE'[-1]$ for all $n \geq 0$. 
\end{remark}

%

\newpage
\section{Some questions on the Hodge-Tate stack and regularity}
\label{ss:HTRegQ}

In this section, we record two questions motivated by the heuristic idea of regarding prismatic cohomology as being related to de Rham cohomology ``over $\mathbf{F}_1$''. More precisely, they are motivated by regarding Hodge-Tate cohomology as  a version of Hodge cohomology ``over $\mathbf{F}_1$''.

The first conjecture, which ought to be accessible through a better understanding of the group scheme $G$ from Proposition~\ref{HTStackRegularLocal}, tries to capture the heuristic that regular rings are formally smooth ``over $\mathbf{F}_1$''.

\begin{conjecture}[Cohomological dimension of the Hodge-Tate stack of a regular ring]
\label{conj:regularHT}
Let $R$ be a $p$-complete noetherian regular local ring with perfect residue field. The functor $\RGamma(\WCart_{\mathrm{Spf}(R)}^{\mathrm{HT}},-)$ (or equivalently $\RGamma(BG,-)$ for $G$ as in  Proposition~\ref{HTStackRegularLocal}) carries $\mathcal{D}^{\leq 0}$ to $\mathcal{D}^{\leq \dim(R)}$.
\end{conjecture}

\begin{remark}
If $\dim(R)=1$, Conjecture~\ref{conj:regularHT} follows from from our previous results. Indeed, if $R$ is $p$-torsionfree, one uses Example~\ref{ExHTOK}. On the other hand, if $R$ is a regular $\mathbf{F}_p$-algebra, then one can use the identification $\WCart_{\mathrm{Spf}(R)}^{\mathrm{HT}} \simeq \WCart_{\mathrm{Spf}(R)/\mathbf{Z}_p}^{\mathrm{HT}}$ from Remark~\ref{RelativeAbsolutePerfect} together with Proposition~\ref{RelativeHT} to settle the case where $R$ is smooth over $\mathbf{F}_p$; the general case can be deduced from this one by a limit argument.
\end{remark}

\begin{remark}
For $R$ as in Conjecture~\ref{conj:regularHT}, it would be interesting to recover the vector bundles studied in the recent works \cite{SaitoFW, DKRZBpder, JeffriesHochsterMixedChar} explicitly from $\WCart_{\mathrm{Spf}(R)}^{\mathrm{HT}}$.
\end{remark}

The following (somewhat optimistic) conjecture gives  a strong converse to the previous one:

\begin{conjecture}[A regularity criterion  via the Hodge-Tate stack]
\label{conj:HTregular}
Let $X$ be a noetherian excellent $p$-adic formal scheme. Write $\WCart_X^{\mathrm{HT}}$ for the derived Hodge-Tate stack of $X$ (i.e., the Hodge--Tate locus in $\WCart_X$ as defined in Definition~\ref{ss:DerRelPrism}).  The following are equivalent:
\begin{enumerate}
\item $X$ is regular.
\item The Hodge-Tate structure map $\pi:\WCart_X^{\mathrm{HT}} \to X$ is a gerbe for a $p$-completely flat $X$-group scheme (and thus $\WCart_X^{\mathrm{HT}}$ and $\WCart_X$ are both classical). 
\item There exists an integer $N$  such that the bifunctor $\mathrm{Ext}^{>N}_{\WCart_X^{\mathrm{HT}}}(-,-)$ vanishes on $p$-torsion quasi-coherent sheaves on $\WCart_X^{\mathrm{HT}}$.
\end{enumerate}
\end{conjecture}

\begin{example}[Part of conjecture~\ref{conj:HTregular} for hypersurface singularities]
Let $W$ be a  $p$-complete and $p$-torsionfree dvr with perfect residue field. Let $A=W\llbracket x_1,...,x_n \rrbracket$, and let $R = A/J$ where $J=(g)$ for some $0 \neq g \in A$. Write $X=\mathrm{Spf}(R)$, and assume that $\WCart_X^{\mathrm{HT}} \to X$ is a gerbe for a flat $X$-group scheme. We shall check that $R$ is regular.

Choose a $\delta$-structure on $A$; using this choice, for any $R$-algebra $S$, we can regard $W(S)$ as a $\delta$-$A$-algebra compatibly with the $A$-algebra structure on $A$. Write $k$ for the residue field of $R$. Let $\mathcal{F}$ be the presheaf on $k$-algebras given by the following: for any $R$-algebra $S$, let $\mathcal{F}(S)$ be the set of all $A$-module maps $\tau:J \to W(S)$ rendering the following diagram commutative
\begin{equation}
\label{eq:HTParam2}
 \xymatrix{ J = (g) \ar[r]^-{\tau} \ar[d]^{\mathrm{std}} & W(S) \ar[d]^{V(1) (=p)} \\ A \ar[r]^-{\xi} & W(S).}
 \end{equation}
Here $\xi$ is the $\delta$-lift of the standard map $A \to R \to S$. Moreover, there is a preferred point $\tau_0 \in \mathcal{F}(k)$ corresponding to the map $J \to W(k)$ induced by the map $R \to k \simeq W(k)/^L V(1)$ of $A$-algebras. Following arguments used to prove Proposition~\ref{HTStackRegularLocal},  one then checks the folowing:
\begin{enumerate}
\item  There is a natural action of $\mathrm{Hom}_A(J,\mathbf{G}_a^\sharp)$ on $\mathcal{F}$ via additive translations (using the embedding $\mathbf{G}_a^\sharp = W[F] \subset W(-)$). Via this action, the presheaf $\mathcal{F}$ is a torsor for $\mathrm{Hom}_A(J,\mathbf{G}_a^\sharp) \simeq \mathbf{G}_a^\sharp$, and is trivlalized by $\tau_0$.

\item There is a natural map $\mathcal{F} \to X^{\DHod} \times_X \mathrm{Spec}(k)$: each square as in \eqref{eq:HTParam2} induces a map $A/J \to W(S)/^L V(1)$ of animated rings.

\item The group scheme $T_A^\sharp \simeq \mathrm{Hom}_A(\Omega^1_{A/W},\mathbf{G}_a^\sharp) \simeq (\mathbf{G}_a^\sharp)^n$ acts on $\mathcal{F}$ through the action in (1) using the homomorphism $\mathrm{Hom}_A(\Omega^1_{A/W},\mathbf{G}_a^\sharp) \xrightarrow{d^\vee} \mathrm{Hom}_A(J,\mathbf{G}_a^\sharp)$. 

\item The standard $\mathbf{G}_m^\sharp$-action on the quasi-ideal $\left(W(S) \xrightarrow{V(1)} W(S)\right)$ induces an action on $\mathcal{F}$.

\item The actions in (3) and (4) amalgamate to an action of $T_A^\sharp \rtimes \mathbf{G}_m^\sharp$ on $\mathcal{F}$. Moreover, the map in (2) yields an isomorphism
\[ \mathcal{F}/(T_A^\sharp \rtimes \mathbf{G}_m^\sharp) \simeq \left(X^{\DHod} \times_X \mathrm{Spec}(k)\right)/\mathbf{G}_m^\sharp \simeq \WCart_X^{\mathrm{HT}} \times_X \mathrm{Spec}(k).\]

\end{enumerate}

Using this discussion and unwinding definitions, we learn that the $k$-stack $\WCart_X^{\mathrm{HT}} \times_X \mathrm{Spec}(k)$ is the quotient of $\mathbf{G}_a^\sharp$ by an action of $(\mathbf{G}_a^\sharp)^n \rtimes \mathbf{G}_m^\sharp$, where the action of the $i$-th basis vector in $(\mathbf{G}_a^\sharp)^n$ is additive translation by $\frac{\partial g}{\partial x_i}$, while the action of the $t \in \mathbf{G}_m^\sharp(S)$ on the point $0 \in \mathbf{G}_a^\sharp(S)$ is given via
\[ t \cdot 0 = (t-1) \tau_0(g) \in W[F](S) = \mathbf{G}_a^\sharp(S) \subset W(S),\]
so the $\mathbf{G}_m^\sharp$-orbit of $0$ coincides with the subideal $\mathbf{G}_a^\sharp \cdot \tau_0(g) \subset \mathbf{G}_a^\sharp$ (inside $W(-)$). We are assuming that $\WCart_X^{\mathrm{HT}} \to X$ is a gerbe, so the action of $(\mathbf{G}_a^\sharp)^n \rtimes \mathbf{G}_m^\sharp$ on $\mathbf{G}_a^\sharp$ is assumed to be transitive. If one of the derivatives $\frac{\partial g}{\partial x_i} \in k$ is nonzero, then this action is clearly transitive, and also $R$ is formally smooth over $W$ by the Jacobian criterion, so we are done. If $\frac{\partial g}{\partial x_i} = 0$ in $k$ for all $i$, then the $(\mathbf{G}_a^\sharp)^n$ factor is acting trivially, so the transitivity assumption translates to the statement that the inclusion 
\[ \mathbf{G}_a^\sharp \cdot \tau_0(g) \subset \mathbf{G}_a^\sharp\]
of ideals of $W(-)$ is an equality. But this can only happen if $\tau_0(g)$ is a unit: scalar multiplication gives an isomorphism of rings $\mathbf{G}_a \simeq \mathrm{End}_W(\mathbf{G}_a^\sharp)$ (see \cite[Lemma 3.8.1 (ii)]{drinfeld-prismatic}) and thus identifies units in $\mathbf{G}_a$ with automorphisms of $\mathbf{G}_a^\sharp$.  Now $\tau_0(g)$ is a unit exactly when $\xi(g) = \tau_0(g) V(1) \in W(k)$ is a unit multiple of $V(1)$, which in turn is equivalent to $g \in A$ being distinguished. But it is easy to see then that $R = A/(g)$ must be regular: if $g \in (p,x_1...,x_n)^2$, then its image in $W(k)$ would lie in $(p^2)$, which violates the distinguishedness.
\end{example}

\begin{remark}
It seems plausible that the equivalence of (1) and (2) in Conjecture~\ref{conj:HTregular} is related to the mixed characteristic Jacobian criterion for regularity proven in \cite{JeffriesHochsterMixedChar}.
\end{remark}

\newpage
\appendix

\section{Animated \texorpdfstring{$\delta$}{delta}-rings}\label{ss:CAlgDeltaAnim}

In this appendix, we introduce a notion of {\it animated $\delta$-ring}, which plays an important background role in several parts of this paper. Motivated by the relationship between Joyal's classical theory of $\delta$-rings \cite{JoyalDelta} and Frobenius lists, we define an animated $\delta$-ring to be an animated ring $R$ equipped with a lift of the Frobenius endomorphism on $R \otimes_{\mathbf{Z}}^L \mathbf{F}_p$ (Definition \ref{definition:animated-delta-ring}). Using this definition,  we check that the resulting $\infty$-category is obtained by ``animating'' the classical notion of a $\delta$-ring (Remark~\ref{UnivPropAnimDelta}), justifying the name ``animated $\delta$-ring''; the latter approach is the one adopted in other places, such as \cite{MaoDerivedCrys} and also implicitly in various arguments of \cite{prisms}. We begin with a brief digression concerning Witt vectors.

\begin{notation}[Derived Functors of Witt Vectors]\label{notation:animated-Witt-vectors}
Fix an integer $n \geq 0$. Let us regard the functor $R \mapsto W_{n}(R)$ as a functor from the
ordinary category of commutative rings to itself. Applying Proposition \APCref{proposition:universal-of-animated},
we deduce that there is an essentially unique endofunctor of the $\infty$-category $\CAlg^{\anim}$
which commutes with sifted colimits and which coincides with $W_{n}$ on the category $\Poly_{\Z}$
of finitely generated polynomial algebras over $\Z$. We will (temporarily) denote this endofunctor
by $LW_{n}: \CAlg^{\anim} \rightarrow \CAlg^{\anim}$ (see Notation \ref{notation:Witt-vectors-revisited} below).
\end{notation}

Let $R$ be an animated commutative ring. We let $[R]$ denote the mapping space 
$\Hom_{ \CAlg^{\anim} }( \Z[x], R)$, which we refer to as the {\it underlying space of $R$}.
In particular, if $R$ is a commutative ring, then $[R]$ denotes its underlying set. 
The construction $R \mapsto [R]$ determines a functor from $\CAlg^{\anim}$ to the
$\infty$-category $\SSet$ of spaces which commutes with sifted colimits.

\begin{remark}[Witt Components]\label{remark:witt-components}
Let $n \geq 0$ be an integer. Let us regard the constructions $R \mapsto [R]$ and $R \mapsto [ LW_n(R) ]$ as functors
from the $\infty$-category of animated commutative rings to the $\infty$-category of spaces. Both of these functors commute
with sifted colimits, and are therefore left Kan extensions of their restriction to the category of finitely generated polynomial algebras over $\Z$
(see Remark \APCref{rmk:LKECAlgAnim}). It follows that, for $0 \leq m < n$, there is an essentially unique natural transformation
$$ \mathrm{comp}_{m}: [ LW_n(R) ] \rightarrow [R]$$
with the following property:
\begin{itemize} 
\item When $R$ is a finitely generated polynomial algebra over $\Z$, the morphism $\mathrm{comp}_{m}: W_n(R) \rightarrow R$
carries a Witt vector $\sum_{i=0}^{n-1} V^{i}[ a_i]$ to its $m$th Witt component $a_m$.
\end{itemize}
We will refer to the morphism $\mathrm{comp}_{m}$ as the {\it $m$th Witt component}. 
\end{remark}

\begin{proposition}\label{wrost}
Let $R$ be an animated commutative ring. For every integer $n \geq 0$, the Witt component maps
$\{ \mathrm{comp}_{m} \}_{0 \leq m < n}$ induce a homotopy equivalence $[ LW_n(R) ] \rightarrow [ R ]^{n}$.
\end{proposition}

\begin{proof}
Since the functors $R \mapsto [ LW_n(R) ]$ and $R \mapsto [R]^{n}$ commute with sifted colimits, we
may assume without loss of generality that $R$ is a finitely generated polynomial algebra over $\Z$. In this case, the
relevant map has an explicit inverse, given by the function
$$ R^{n} \rightarrow W_n(R) \quad \quad (a_0, a_1, \cdots, a_{n-1}) \mapsto \sum V^{i} [a_i ].$$
\end{proof}

\begin{corollary}\label{corollary:usual-Witt-vectors}
Let $n \geq 0$ be an integer. For every commutative ring $R$, there is a canonical isomorphism $LW_n(R) \simeq W_n(R)$,
which is determined by the requirement that it depends functorially on $R$ and coincides with the identity map when
$R$ is a finitely generated polynomial algebra over $\Z$.
\end{corollary}

\begin{proof}
By virtue of Remark \APCref{rmk:LKECAlgAnim}, there is an essentially unique natural transformation $\alpha_{R}: LW_n(R) \rightarrow W_n(R)$
which coincides with the identity when $R$ is a finitely generated polynomial algebra over $\Z$. We have a commutative diagram of sets
$$ \xymatrix@R=50pt@C=50pt{ [LW_{n}(R)] \ar[dr]_{ \{ \mathrm{comp}_{m} \}_{0 \leq m < n}} \ar[rr]^{ [\alpha_R] } & & [W_n(R)] \ar[dl]^{ \sum V^i [a_i]  \mapsto (a_0, a_1, \cdots, a_{n-1} ) } \\
& [R]^{n}. & }$$
Proposition \ref{wrost} guarantees that the left vertical map is an isomorphism. Since the right vertical map is an isomorphism,
it follows that $\alpha_{R}$ is also an isomorphism.
\end{proof}

\begin{notation}\label{notation:Witt-vectors-revisited}
Let $R$ be an animated commutative ring. For $n \geq 0$, we write $W_{n}(R)$ for the animated commutative ring
$LW_n(R)$ introduced in Notation \ref{notation:animated-Witt-vectors}. We will refer to $W_{n}(R)$ as the
{\it Witt vectors of $R$ of length $n$}. By virtue of Corollary \ref{corollary:usual-Witt-vectors}, this agrees
with the usual definition in the case where $R$ is a commutative ring.
\end{notation}

\begin{example}
For every animated commutative ring $R$, we have a canonical isomorphism $W_{1}(R) \simeq R$.
\end{example}

\begin{remark}[Restriction Maps]\label{remark:restriction-maps}
Let $R$ be an animated commutative ring. For every $n \geq 0$, there is a restriction map
$\mathrm{Res}: W_{n+1}(R) \rightarrow W_{n}(R)$, which is determined (up to homotopy)
by the requirement that it depends functorially on $R$ and coincides with the usual restriction map
$$ \sum_{i=0}^{n} V^{i} [a_i ] \mapsto  \sum_{i=0}^{n-1} V^{i} [a_i ]$$
when $R$ is a commutative ring. We therefore obtain a tower of animated commutative rings
$$ \cdots \rightarrow W_{3}(R) \rightarrow W_{2}(R) \rightarrow W_{1}(R) \simeq R.$$
We let $W(R)$ denote the limit of this tower (formed in the $\infty$-category $\CAlg^{\anim}$). We will refer to $W(R)$
as the {\it Witt vectors of $R$}. 
\end{remark}

\begin{warning}\label{warning:left-Kan-Witt}
The construction $R \mapsto W(R)$ determines an endofunctor of the $\infty$-category $\CAlg^{\anim}$ of animated
commutative rings. This endofunctor cannot be obtained directly from Proposition \APCref{proposition:universal-of-animated},
since it does not commute with filtered colimits. However, the Witt components of Remark \ref{remark:witt-components} supply 
a natural homotopy equivalence of $[W(R)]$ with the infinite product $\prod_{n \geq 0} [R]$. It follows that the functor
$R \mapsto W(R)$ commutes with the formation of geometric realizations, and is therefore a left Kan extension of
its restriction to the ordinary category of commutative rings.
\end{warning}

\begin{remark}[Ghost Components]
Let $R$ be an animated commutative ring. For every pair of integers $0 \leq m < n$, there is a morphism of animated
commutative rings $\gamma_{m}: W_{n}(R) \rightarrow R$, which is determined (up to homotopy) by the requirement
that it depends functorially on $R$ and that it coincides with the homomorphism
$$ \sum_{i=0}^{n-1} V^{i} [a_i ] \mapsto \sum_{i=0}^{m} p^{i} a_{i}^{p^{m-i}}$$
when $R$ is a commutative ring. We will refer to $\gamma_{m}$ as the {\it $m$th ghost component}.
\end{remark}

\begin{remark}\label{remark:fundamental-pullback-square}
Let $R$ be a commutative ring. For each element $x \in R$, let $\overline{x}$ denote its image in the quotient ring
$R/pR$. We then have a commutative diagram $\sigma_{R}:$
$$ \xymatrix@R=50pt@C=50pt{ W_2(R) \ar[r]^-{ \gamma_1 } \ar[d]^{\mathrm{Res}} & R \ar[d]^{ x \mapsto \overline{x} } \\
R \ar[r]^-{ x \mapsto \overline{x}^{p} } & R/pR. }$$
When $R$ is $p$-torsion-free, this diagram is a pullback square.

Applying Proposition \APCref{proposition:universal-of-animated}, we see that the functor $R \mapsto \sigma_{R}$
has a nonabelian left derived functor, which carries each animated commutative ring $R$ to a diagram of animated commutative rings $L\sigma_{R}$:
$$ \xymatrix@R=50pt@C=50pt{ W_2(R) \ar[r]^-{ \gamma_1 } \ar[d]^{\mathrm{Res}} & R \ar[d]^{\mathrm{can}} \\
R \ar[r]^-{ \mathrm{Frob} } & \F_p \otimes^{L} R. }$$
This construction has the following features:
\begin{itemize}
\item The animated commutative ring $\F_p \otimes^{L} R$ is the coproduct of $\F_p$ with $R$ in the $\infty$-category
$\CAlg^{\anim}$. When $R$ is an ordinary commutative ring, there is a canonical map
$\F_p \otimes^{L} R \rightarrow R/pR$, which is an isomorphism if and only if $R$ is $p$-torsion-free.

\item The morphism $\mathrm{can}: R \rightarrow \F_p \otimes^{L} R$ is given by the inclusion of the second factor.

\item The morphism $\mathrm{Frob}: R \rightarrow \F_p \otimes^{L} R$ is uniquely determined by the requirement that it depends functorially
on $R$ and coincides with the homomorphism $x \mapsto \overline{x}^{p}$ when $R$ is a $p$-torsion-free commutative ring.

\item When $R$ is a $p$-torsion-free commutative ring, the diagram $L\sigma_{R}$ agrees with $\sigma_{R}$ (up to canonical isomorphism).

\item The diagram $L\sigma_{R}$ is always a pullback square (in the $\infty$-category $\CAlg^{\anim}$).
\end{itemize}
\end{remark}

\begin{definition}[Animated $\delta$-Rings]\label{definition:animated-delta-ring}
Let $A$ be an animated commutative ring. A {\it $\delta$-structure on $A$} is a section of the restriction map
$\mathrm{Res}: W_2(A) \rightarrow A$. An {\it animated $\delta$-ring} is an animated commutative ring $A$ together with a $\delta$-structure on $A$.
More precisely, an animated $\delta$-ring is a commutative diagram
$$ \xymatrix@R=50pt@C=50pt{ & W_2(A) \ar[dr]^{ \mathrm{Res} } &  \\
A \ar[ur] \ar[rr]_{ \id_{A}  } & & A }$$
in the $\infty$-category $\CAlg^{\anim}$, where the horizontal morphism is an identity morphism
and the right vertical map is given by the construction of Remark \ref{remark:restriction-maps}. 
\end{definition}

\begin{example}[$\delta$-Structures on Commutative Rings]\label{example:wroth}
Let $A$ be a commutative ring. In the category of sets, every section of the restriction map
$\mathrm{Res}: W_2(A) \rightarrow A$ is given by $x \mapsto [x] + V[ \delta_{A}(x) ]$ for a unique function $\delta_{A}: A \rightarrow A$.
The function $\delta_{A}$ determines a $\delta$-structure on $A$ (in the sense of Definition \ref{definition:animated-delta-ring}) if
and only if the construction $x \mapsto [x] + V[ \delta_{A}(x) ]$ is a ring homomorphism. Concretely, this is
equivalent to the requirement that $\delta_{A}$ satisfies the identities 
\begin{equation}
\begin{gathered}\label{equation:delta-ring-identities}
\delta_A(x+y) = \delta_A(x) + \delta_A(y) - (p-1)! \sum_{0 < i < p} \frac{ x^{i} }{i!} \frac{y^{p-i}}{ (p-i)!} \\
\delta_A(1) = 0 \quad \quad \delta_A(xy) = x^{p} \delta_A(y) + y^{p} \delta_A(x) + p \delta_A(x) \delta_A(y).
\end{gathered}
\end{equation}
Consequently, when restricted to commutative rings, Definition \ref{definition:animated-delta-ring} recovers the usual notion
of $\delta$-structure (originally introduced by Joyal in ***).
\end{example}

\begin{remark}
Let $A$ be an animated $\delta$-ring, so that the restriction map $\mathrm{Res}: W_2(A) \rightarrow A$
is equipped with a section $s: A \rightarrow W_{2}(A)$. We write $\delta_{A}$ for the composite map
$$ [A] \xrightarrow{ [s] } [ W_2(A) ] \xrightarrow{ \mathrm{comp}_{1} } [A].$$
When $A$ is a commutative ring, this agrees with the function described in Example \ref{example:wroth},
which completely determines the $\delta$-structure on $A$. Beware that, in general, this is no longer true.
\end{remark}

\begin{remark}
\label{AnimDeltapi0}
Let $A$ be an animated $\delta$-ring, so that the restriction map $\mathrm{Res}: W_2(A) \rightarrow A$ is equipped with a section $s: A \rightarrow W_{2}(A)$.  By Proposition~\ref{wrost}, the formation of $W_2(A)$ commutes with passing to connected components: the natural map $\pi_0( W_2(A)) \to W_2 (\pi_0(A))$ is an isomorphism. Applying $\pi_0(-)$ to the $\delta$-structure on $A$ then shows that $\pi_0(A)$ is also naturally an animated $\delta$-ring, and that the natural map $A \to \pi_0(A)$ is a map of $\delta$-rings.
\end{remark}

\begin{remark}[$\delta$-Structures as Frobenius Lifts]\label{remark:delta-Frobenius-lift}
Let $A$ be an animated commutative ring. Using the pullback diagram described in Remark \ref{remark:fundamental-pullback-square},
we see that $\delta$-structures on $A$ can be identified with commutative diagrams
$$\xymatrix@R=50pt@C=50pt{  & A \ar[dr]^{ \mathrm{can} } & \\
A \ar[ur]^{\varphi_A} \ar[rr]^{ \mathrm{Frob} } & & \F_p \otimes^{L} A }$$
in the $\infty$-category $\CAlg^{\anim}$, where $\mathrm{can}: A \rightarrow \F_p \otimes^{L} A$ is given by the inclusion of the second factor
and $\mathrm{Frob}$ is the Frobenius morphism of Remark \ref{remark:fundamental-pullback-square}. Stated more informally,
$\delta$-structures on $A$ can be identified with pairs $(\varphi_A, h)$, where $\varphi_{A}: A \rightarrow A$ is a morphism
of animated commutative rings and $h$ is a homotopy from $\mathrm{can} \circ \varphi_{A}$ to $\mathrm{Frob}$.
In other words, a $\delta$-structure on $A$ is given by ``lift'' of the Frobenius morphism $\mathrm{Frob}: A \rightarrow \F_p \otimes^{L} A$ to an endomorphism of $A$ (which we denote by $\varphi_{A}$ and also refer to as the Frobenius morphism of $A$).
\end{remark}

\begin{example}[The Torsion-Free Case]\label{example:torsion-free}
Let $A$ be a commutative ring which is $p$-torsion-free. Then a $\delta$-structure on $A$ can be identified with a ring homomorphism $\varphi_{A}: A \rightarrow A$ satisfying $\varphi_{A}(x) \equiv x^{p} \pmod{p}$ for each $x \in A$ (in this case, the homotopy $h$ of Remark \ref{remark:delta-Frobenius-lift} is essentially unique if it exists).
\end{example}

\begin{remark}[Perfect Animated $\delta$-Rings]\label{remark:perfect-animated-delta}
Let $A$ be an animated $\delta$-ring. The following conditions are equivalent:
\begin{itemize}
\item The underlying animated commutative ring $A$ is $p$-complete and the morphism $\varphi_{A}: A \rightarrow A$ of Remark \ref{remark:delta-Frobenius-lift}
is an isomorphism.
\item There exists an isomorphism of animated $\delta$-rings $A \simeq W(k)$, where $k$ is a perfect $\F_p$-algebra.
\end{itemize}
\end{remark}

The collection of animated $\delta$-rings forms an $\infty$-category which we will denote by $\dCAlg^{\anim}$ and refer to as the {\it $\infty$-category of animated $\delta$-rings}. It is characterized formally by the existence of a pullback diagram
\begin{equation}
\begin{gathered}\label{equation:delta-ring-pullback}
\xymatrix@R=50pt@C=50pt{ \dCAlg^{\anim} \ar[r] \ar[d] & \CAlg^{\anim} \ar[d]^{ A \mapsto ( \id_{A}, \mathrm{Res} ) } \\
\Fun( \{ 0 < 1 < 2 \}, \CAlg^{\anim} ) \ar[r] & \Fun( \{ 0 < 2 \}, \CAlg^{\anim} ) \times_{ \CAlg^{\anim} } \Fun( \{ 1 < 2 \}, \CAlg^{\anim} ). }
\end{gathered}
\end{equation}
It follows from Example \ref{example:wroth} that the ordinary category of $\delta$-rings can be identified with a full subcategory of $\dCAlg^{\anim}$,
spanned by those animated $\delta$-rings for which the underlying animated commutative ring is discrete.

\begin{example}[Polynomial $\delta$-Rings]\label{example:delta-structure-on-polynomial}
Let $A = \Z[ x_i ]_{i \in I}$ be a polynomial ring on a set of generators $\{ x_i \}_{i \in I}$. For every collection of elements $\{ a_i \}_{i \in I}$ of $A$,
there is a unique $\delta$-structure on $A$ satisfying $\delta_{A}(x_i) = a_{i}$. This follows from Example \ref{example:torsion-free}, taking $\varphi_{A}$ to be the endomorphism of $A$ given by the formula $\varphi_{A}(x_i) = x_i^{p} + p a_i$. For any animated $\delta$-ring $B$, the mapping space
$\Hom_{ \dCAlg^{\anim} }(A,B)$ can be identified with the equalizer of the diagram
$$\xymatrix@R=50pt@C=100pt{ \Hom_{\CAlg^{\anim}}( A, B) \ar@<.6ex>[r]^-{ f \mapsto \{ \delta_{B}( f(x_i)) \}_{i \in I} } 
\ar@<-.6ex>[r]_-{f \mapsto \{ f( \delta_{A}(x_i) ) \}} & [B]^{I}.}$$
\end{example}

\begin{example}[The Free $\delta$-Ring]\label{example:free-delta-ring}
Let $A = \Z[ x_0, x_1, x_2, \ldots ]$ denote the polynomial ring on a countably infinite set of variables. Then $A$ is a $p$-torsion-free commutative ring equipped with a lift of Frobenius $\varphi_{A}$, given by the construction $\varphi_{A}( x_i ) = x_{i}^{p} + p x_{i+1}$. For any animated $\delta$-ring $B$,
Example \ref{example:delta-structure-on-polynomial} identifies $\Hom_{ \dCAlg^{\anim} }(A,B)$ with the equalizer of a diagram
$$\xymatrix@R=50pt@C=50pt{ \prod_{n \geq 0} [B] \ar@<.6ex>[r] 
\ar@<-.6ex>[r] & \prod_{n \geq 0} [B],}$$
where the upper map is given by applying $\delta_{B}$ to each factor and the lower vertical map is given by 
the shift $(b_0, b_1, \cdots ) \mapsto (b_1, b_2, b_3, \cdots )$. It follows that the tautological map
$$ \Hom_{ \dCAlg^{\anim} }(A,B) \rightarrow \Hom_{ \CAlg^{\anim} }( \Z[x_0], B )  = [B]$$
is a homotopy equivalence. Stated more informally, $A$ is the free animated $\delta$-ring on the generator $x_0$.
\end{example}

\begin{proposition}\label{proposition:delta-adjunction}
\begin{itemize}
\item[$(1)$] The $\infty$-category $\dCAlg^{\anim}$ admits small limits and colimits.
\item[$(2)$] The forgetful functor $\mathrm{Forget}: \dCAlg^{\anim} \rightarrow \CAlg^{\anim}$ preserves small limits and colimits.
\item[$(3)$] The functor $\mathrm{Forget}$ admits a left adjoint $\mathrm{Free}: \CAlg^{\anim} \rightarrow \dCAlg^{\anim}$.
\item[$(4)$] The $\infty$-category $\dCAlg^{\anim}$ can be identified with the $\infty$-category of algebras
over the monad $\mathrm{Forget} \circ \mathrm{Free}: \CAlg^{\anim} \rightarrow \CAlg^{\anim}$.
\end{itemize}
\end{proposition}

\begin{proof}
Assertions $(1)$ and $(2)$ are essentially formal. To prove $(3)$, we must show that for every
animated commutative ring $R$, there exists an animated commutative $\delta$-ring $\mathrm{Free}(R)$
which corepresents the functor
$$ \CAlg^{\anim} \rightarrow \SSet \quad \quad A \mapsto \Hom_{ \CAlg^{\anim} }( R, \mathrm{Forget}(A) ).$$
By virtue of $(1)$, the collection of objects $R \in \CAlg^{\anim}$ which satisfy this condition is closed under the formation
of colimits. It will therefore suffice to treat the case where $R = \Z[x]$ is a polynomial ring on a single generator.
In this case, the desired result follows from Example \ref{example:free-delta-ring}.
Assertion $(4)$ now follows by combining $(1)$, $(2)$, and $(3)$ with the $\infty$-categorical Barr-Beck theorem.
\end{proof}

\begin{example}\label{example:rouch}
Let $\{ x_i \}_{i \in I}$ be a set of variables and let $\Z[x_i]_{i \in I}$ denote the associated polynomial ring.
Then the functor $\mathrm{Free}: \CAlg^{\anim} \rightarrow \dCAlg^{\anim}$ of Proposition \ref{proposition:delta-adjunction}
carries $\Z[x_i]_{i \in I}$ to the free $\delta$-ring $\Z\{x_i\}_{i \in I}$ generated by the variables $\{ x_i \}_{i \in I}$,
given (as a commutative ring) by the polynomial ring on generators $\{ \delta^{n}(x_i) \}_{i \in I, n \geq 0}$.
To prove this, we can assume without loss of generality that the set $I$ is a singleton, in which case
the desired result follows from Example \ref{example:free-delta-ring}.

For any animated $\delta$-ring $B$, we have homotopy equivalences
$$ \Hom_{ \dCAlg^{\anim} }( \Z\{x_i\}_{i \in I}, B) \simeq
\Hom_{ \CAlg^{\anim} }( \Z[x_i]_{i \in I}, B) \simeq [B]^{I}.$$
In particular, if the set $I$ is finite, then the functor $B \mapsto  \Hom_{ \dCAlg^{\anim} }( \Z\{x_i\}_{i \in I}, B)$
commutes with sifted colimits: that is, $\Z\{x_i\}_{i \in I}$ is a compact projective object of the
$\infty$-category $\dCAlg^{\anim}$.
\end{example}

\begin{remark}[The Universal Property of $\dCAlg^{\anim}$]
\label{UnivPropAnimDelta}
Let $\dPoly$ denote the full subcategory of $\dCAlg^{\anim}$ spanned by the finitely generated
free $\delta$-rings (that is, $\delta$-rings of the form $\Z\{x_1, x_2, \cdots, x_n\}$ for some $n \geq 0$).
Proposition \ref{proposition:delta-adjunction} implies that $\dCAlg^{\anim}$ is generated
by $\dPoly$ under sifted colimits. Since each object of $\dPoly$ is a compact projective object of $\dCAlg^{\anim}$ (Example \ref{example:rouch}),
it follows that $\dCAlg^{\anim}$ is {\em freely} generated by $\dPoly$ under sifted colimits. That is, if $\calC$
is an $\infty$-category which admits sifted colimits, then every functor $F: \dPoly \rightarrow \calC$ admits an essentially unique extension
$LF: \dCAlg^{\anim} \rightarrow \calC$ which commutes with sifted colimits (see Proposition~5.5.8.15 of \cite{HTT}).
\end{remark}

The forgetful functor of Proposition \ref{proposition:delta-adjunction} also has a right adjoint.

\begin{proposition}\label{proposition:other-delta-adjunction}
\begin{itemize}
\item[$(1)$] The forgetful functor $\mathrm{Forget}: \dCAlg^{\anim} \rightarrow \CAlg^{\anim}$ has a right
adjoint $\mathrm{Cofree}: \CAlg^{\anim} \rightarrow \dCAlg^{\anim}$.

\item[$(2)$] The composite functor
$$ \CAlg^{\anim} \xrightarrow{ \mathrm{Cofree} } \dCAlg^{\anim} \xrightarrow{ \mathrm{Forget} } \CAlg^{\anim}$$
is isomorphic to the Witt vector functor $R \mapsto W(R)$ of Remark \ref{remark:restriction-maps}.

\item[$(3)$] The $\infty$-category $\dCAlg^{\anim}$ can be identified with the $\infty$-category of algebras
over the comonad $\mathrm{Forget} \circ \mathrm{Cofree}: \CAlg^{\anim} \rightarrow \CAlg^{\anim}$.
\end{itemize}
\end{proposition}

\begin{remark}
We can summarize Proposition \ref{proposition:other-delta-adjunction} more informally as follows:
\begin{itemize}
\item For every animated commutative ring $R$, the animated commutative ring $W(R)$ 
carries a natural $\delta$-structure.
\item If $A$ is an animated $\delta$-ring, then there is a canonical homotopy equivalence
$$ \theta: \Hom_{ \dCAlg^{\anim} }( A, W(R) ) \simeq \Hom_{ \CAlg^{\anim} }( A, R ).$$
\end{itemize}
Concretely, the homotopy equivalence $\theta$ is given by composition with the restriction map $\mathrm{Res}: W(R) \rightarrow R$.
\end{remark}

\begin{proof}[Proof of Proposition \ref{proposition:other-delta-adjunction}]
Assertion $(1)$ follows from the adjoint functor theorem, and assertion $(3)$ from the $\infty$-categorical Barr-Beck theorem.
We will prove $(2)$. Let us denote the composite functor $\mathrm{Forget} \circ \mathrm{Cofree}$ by $W': \CAlg^{\anim} \rightarrow \CAlg^{\anim}$.
For every commutative ring $R$, the ring of Witt vectors $W(R)$ is equipped with
a $\delta$-structure, so the restriction map $\mathrm{Res}: W(R) \rightarrow R$ is classified by a morphism
of animated $\delta$-rings $\alpha_{R}: W(R) \rightarrow \mathrm{Cofree}(R)$ which we
can identify with a morphism of animated commutative rings $\beta_{R}: W(R) \rightarrow W'(R)$.
Since the functor $W$ is a left Kan extension of its restriction to the ordinary category of commutative
rings (Warning \ref{warning:left-Kan-Witt}), the construction $R \mapsto \beta_{R}$ admits an essentially unique extension to a natural
transformation $\beta: W \rightarrow W'$ (defined on the $\infty$-category $\CAlg^{\anim}$).
To complete the proof, it will suffice to show that $\beta$ is an isomorphism.

For every animated commutative ring $R$, we have canonical homotopy equivalences
\begin{eqnarray*}
[W'(R)] & = & \Hom_{ \CAlg^{\anim} }( \Z[x], W'(R) ) \\
& = & \Hom_{ \CAlg^{\anim} }( \Z[x], \mathrm{Forget}( \mathrm{Cofree}(R) ) ) \\
& \simeq & \Hom_{ \dCAlg^{\anim} }( \Z\{x\}, \mathrm{Cofree}(R) ) \\
& \simeq & \Hom_{ \CAlg^{\anim} }( \Z\{x\}, R ) \\
& \simeq & \prod_{n \geq 0} [R].
\end{eqnarray*}
It follows that the functor $W'$ commutes with geometric realizations of simplicial objects,
and is therefore a left Kan extension of its restriction to the ordinary category of commutative rings.
Consequently, to show that $\beta_{R}: W(R) \rightarrow W'(R)$ is an isomorphism, we may assume
without loss of generality that $R$ is an ordinary commutative ring. In this case, we wish to show
that $\alpha_{R}: W(R) \rightarrow \mathrm{Cofree}(R)$ is an isomorphism of animated $\delta$-rings.
Equivalently, we wish to show that for every animated $\delta$-ring $A$, composition with
$\alpha_{R}$ induces a homotopy equivalence.
$$ \Hom_{ \dCAlg^{\anim} }( A, W(R) ) \rightarrow \Hom_{\dCAlg^{\anim} }( A, \mathrm{Cofree}(R) ) \simeq
\Hom_{ \CAlg^{\anim} }( A, R ).$$
Since $R$ and $W(R)$ are discrete, we can assume without loss of generality that $A$ is discrete. In this case, the desired result follows from
the universal property of $\delta$-structure on $W(R)$ (see \cite[\S 2]{JoyalDelta}). 
\end{proof}

\newpage

\bibliography{bibliography}{}
\bibliographystyle{plain}

\end{document}